\documentclass[11pt]{article}
\usepackage{amsmath}
\usepackage{amssymb}
\usepackage{amsthm}
\usepackage{mathabx}
\newtheorem{theorem}{Theorem}[section]
\newtheorem{lemma}[theorem]{Lemma}
\newtheorem{corollary}[theorem]{Corollary}
\newtheorem{proposition}[theorem]{Proposition}
\newtheorem{definition}[theorem]{Definition}
\newtheorem{example}[theorem]{Example}
\newtheorem{question}[theorem]{Question}
\newtheorem{remark}[theorem]{Remark}

\newtheorem*{thm1}{Theorem KPS1}
\newtheorem*{thm2}{Theorem KPS2}
\newtheorem*{que1}{Question KPS3}
\newtheorem*{que2}{Question KPS4}
\newtheorem*{thmu}{Theorem U}

\numberwithin{equation}{section}

\long\def\symbolfootnote[#1]#2{\begingroup%
\def\thefootnote{\fnsymbol{footnote}}\footnote[#1]{#2}\endgroup}

\begin{document}

\def\C{{\mathbb C}}
\def\N{{\mathbb N}}
\def\Z{{\mathbb Z}}
\def\R{{\mathbb R}}
\def\D{{\mathbb D}}
\def\T{{\mathbb T}}
\def\Q{{\mathbb Q}}
\def\U{{\mathbb U}}
\def\F{{\cal F}}
\def\M{{\cal M}}
\def\hh{{\cal H}}
\def\E{{\cal E}}
\def\rr{{\cal R}}
\def\pp{{\cal P}}
\def\epsilon{\varepsilon}
\def\kappa{\varkappa}
\def\phi{\varphi}
\def\leq{\leqslant}
\def\geq{\geqslant}
\def\re{\text{\tt Re}\,}
\def\slim{\mathop{\hbox{$\overline{\hbox{\rm lim}}$}}\limits}
\def\ilim{\mathop{\hbox{$\underline{\hbox{\rm lim}}$}}\limits}
\def\supp{\hbox{\tt supp}\,}
\def\dim{\hbox{\tt dim}\,}
\def\ker{\hbox{\tt ker}\,}
\def\Ker{\hbox{\tt ker}^\star\,}
\def\Int{\hbox{\tt int}\,}
\def\spann{\hbox{\tt span}\,}
\def\Re{\hbox{\tt Re}\,}
\def\dist{\hbox{\tt dist}\,}
\def\deg{\hbox{\tt deg}\,}
\def\ssub#1#2{#1_{{}_{{\scriptstyle #2}}}}
\def\bin#1#2{\left({{#1}\atop {#2}}\right)}
\def\NH{{\it N\!H}}
\def\WNH{{\it W\!N\!H}}
\def\SNH{{\it S\!N\!H}}
\def\NO{{\it N\!O}}
\def\divi{{\,|\,}} 

\title{On numerically hypercyclic operators}

\author{Stanislav Shkarin}

\date{}

\maketitle

\begin{abstract} According to Kim, Peris and Song, a continuous linear operator $T$ on a complex Banach space $X$ is called {\it numerically hypercyclic} if the numerical orbit $\{f(T^nx):n\in\N\}$
is dense in $\C$ for some $x\in X$ and $f\in X^*$ satisfying $\|x\|=\|f\|=f(x)=1$. They have characterized numerically hypercyclic weighted shifts and provided an example of a numerically hypercyclic operator on $\C^2$.

We answer two questions of Kim, Peris and Song. Namely, we construct a numerically hypercyclic operator, whose square is not numerically hypercyclic as well as an operator which is not numerically hypercyclic but has two numerical orbits whose union is dense in $\C$. We characterize numerically hypercyclic operators on $\C^2$ as well as the operators similar to a numerically hypercyclic one and those operators whose conjugacy class consists entirely of numerically hypercyclic operators. We describe in spectral terms the operator norm closure of the set of numerically hypercyclic operators on a reflexive Banach space. Finally, we provide criteria for numeric hypercyclicity and decide upon the numerical hypercyclicity of operators from various classes.
\end{abstract}


\normalsize

\section{Introduction \label{s1}}\rm

Throughout this article $X$ stands for a Banach space over the field $\C$ of complex numbers, while $S(X)=\{x\in X:\|x\|=1\}$ and $B(X)=\{x\in X:\|x\|\leq 1\}$.  As usual, $\R$ is the field of real numbers, $\R_+$ is the set of non-negative real numbers, $\D=\{z\in\C:|z|<1\}$, $\T=\{z\in\C:|z|=1\}$, $\Z_+$ is the set of non-negative integers and $\N$ is the set of positive integers. If $M$ is a finite or countable set, the symbol $\ell^1_+(M)$ stands for the set of sequences in $\ell^1(M)$ whose entries are all non-negative real numbers. We say that $z_1,\dots,z_k\in\T$ are {\it independent} if $z_1^{m_1}\dots z_k^{m_k}\neq 1$ for every non-zero vector $m\in\Z^k$. Equivalently, $z_j=e^{i\theta_j}$ with $\theta_j\in\R$ are independent if and only if $\pi,\theta_1,\dots,\theta_k$ are linearly independent over the field $\Q$ of rational numbers. It is well-known and easy to prove that $z_1,\dots,z_k\in\T$ are independent if and only if the set $\{(z_1^n,\dots,z_k^n):n\in\Z_+\}$ is dense in $\T^n$.
Recall that a sequence $\{e_n\}_{n\in\N}$ in $X$ is called a {\it Schauder basis} in $X$ if every $x\in X$ can be uniquely written as $x=\sum\limits_{n=1}^\infty c_nx_n$, where $c_n\in\C$ and the series is norm-convergent. The functionals $e_n^*:X\to\C$, $e_n^*(x)=c_n$ are automatically continuous and are called the {\it coordinate functionals} of the Schauder basis $\{e_n\}_{n\in\N}$. Obviously, $e_n^*(e_k)=\delta_{n,k}$ for every $n,k\in\N$. A sequence $\{e_n\}_{n\in\N}$ in $X$ is called a {\it Schauder basic sequence} if $\{e_n\}_{n\in\N}$ is a Schauder basis in the closed linear span of the set $\{e_n:n\in\N\}$. The symbol $L(X)$ stands for the space of bounded linear operators $T:X\to X$ and $X^*$ is the space of continuous linear functionals $f:X\to\C$. For $T\in L(X)$, the dual operator $T^*\in L(X^*)$ acts according to the formula  $T^*f(x)=f(Tx)$. For $T\in L(X)$, $\sigma(T)$ stands for the spectrum of $T$, while $\sigma_p(T)$ denotes the point spectrum of $T$. Recall that $T\in L(X)$ is called {\it semi-Fredholm} if $T(X)$ is closed in $X$ and either $X/T(X)$ or $\ker T$ (or both, in which case $T$ is called {\it Fredholm}) is finite dimensional. The number ${\bf i}(T)=\dim\,\ker T-\dim X/T(X)\in \Z\cup\{-\infty,+\infty\}$ is called the {\it index} of $T$. Note that the set of semi-Fredholm operators is norm-open in $L(X)$ and that the index function ${\bf i}$ is locally constant. Recall that $\lambda\in\C$ is called a {\it normal eigenvalue} of $T\in L(X)$ if $\lambda$ is an isolated point of $\sigma(T)$ and the spectral projection corresponding to the clopen subset $\{\lambda\}$ of $\sigma(T)$ has finite rank. This rank is called the {\it multiplicity} of the normal eigenvalue $\lambda$.

Recall that $x\in X$ is called a {\it hypercyclic vector} for $T\in L(X)$ if the orbit
$$
O(T,x)=\{T^nx:n\in\Z_+\}
$$
is dense in $X$. Similarly, $x\in X$ is called a {\it weakly hypercyclic vector} for $T\in L(X)$ if $O(T,x)$ is dense in $X$ with respect to the weak topology. An operator $T$, which has a hypercyclic vector is called {\it hypercyclic}, while an operator $T$, which has a weakly hypercyclic vector, is called {\it weakly hypercyclic}. For more information on hypercyclicity see the books \cite{bama,gp-book} and references therein. Recall also that $T\in L(X)$ is called {\it power bounded} if $\sup\{\|T^n\|:n\in\N\}<\infty$. Unless stated otherwise, we assume that $\C^n$ carries the standard Euclidean norm. We shall routinely identify an operator $T\in L(\C^n)$ with its matrix. For instance, a {\it diagonal operator} on $\C^n$ is an operator whose matrix is diagonal. The same agreement stands for operators on classical sequence spaces.

\subsection{The main concept}

For $T\in L(X)$ and $(x,f)\in X\times X^*$, the {\it numerical orbit of $(x,f)$} is the set
$$
O(T,x,f)=f(O(T,x))=\{f(T^nx):n\in\Z_+\}\subset \C.
$$
Note that numerical orbits are also known as weak orbits, see \cite{wo} and references therein. The following definition, motivated by the concept of the numerical range of an operator, was introduced by Kim, Peris and Song \cite{nhy}, see also \cite{nhy1} for the generalizations to polynomials on Banach spaces.

\begin{definition}\label{nuhy1}\rm
Let
$$
\Pi(X)=\{(x,f)\in X\times X^*:\|x\|=\|f\|=f(x)=1\}.
$$
We say that $(x,f)\in \Pi(X)$ is a {\it numerically hypercyclic
vector} for $T\in L(X)$ if the numerical orbit $O(T,x,f)$ is dense in
$\C$. An operator $T$ is called {\it numerically hypercyclic} if it
has a numerically hypercyclic vector. We use the symbol $\NH(X)$ to denote the set of all
numerically hypercyclic operators $T\in L(X)$.
\end{definition}

The following proposition collects some elementary observations made in \cite{nhy}.

\begin{proposition}\label{ele} Let $T\in L(X)$.
\begin{itemize}\itemsep=-2pt
\item
If $T$ is power bounded, then it is not numerically hypercyclic.
\item
If $T$ is weakly hypercyclic, then it is numerically hypercyclic.
\item
If the restriction of $T$ to a closed invariant subspace is numerically hypercyclic, then so is $T$.
\item
If $T$ is numerically hypercyclic, then $T^*$ is numerically
hypercyclic. In particular, if $X$ is reflexive, then $T$ is
numerically hypercyclic if and only if $T^*$ is.
\item
If $X$ is a Hilbert space, then $T$ is numerically
hypercyclic if and only if there is $x\in S(X)$ such that $\{\langle T^nx,x\rangle:n\in\Z_+\}$
is dense in $\C$, where $\langle\cdot,\cdot\rangle$ is the inner product of $X$.
\end{itemize}
\end{proposition}

The last observation is due to the fact that if $\hh$ is a Hilbert space and $x\in S(\hh)$, then the only functional $f\in S(\hh^*)$ satisfying $f(x)=1$ is given by $f(y)=\langle y,x\rangle$. Thus in this case the numerical orbits featuring in the definition of numeric hypercyclicity are naturally labelled by the elements of $S(\hh)$. If $\hh$ is a Hilbert space and $T\in L(\hh)$, we use the following notation:
$$
\NO(T,x)=\{\langle T^nx,x\rangle:n\in\Z_+\}\ \ \text{for $x\in S(\hh)$.}
$$

Recall that forward $F_w$ and backward $B_w$ weighted shifts on $\ell^p(\Z_+)$ act according to the rule $(x_0,x_1,x_2,\dots)\mathop{\mapsto}\limits^{F_w}(0,w_1x_0,w_2x_1,\dots)$ and
$(x_0,x_1,x_2,\dots)\mathop{\mapsto}\limits^{B_w}(w_1x_1,w_2x_2,\dots)$, where $w=\{w_n\}_{n\in\N}$ is a bounded sequence of non-zero scalars. A bilateral weighted shift $T_w$ on $\ell^p(\Z)$ is defined by $(T_wx)_n=w_{n+1}x_{n+1}$, where again $w=\{w_n\}_{n\in\Z}$ is a bounded sequence of non-zero scalars. The main results of \cite{nhy} can be summarized as follows.

\begin{thm1} There is a diagonal numerically hypercyclic operator on $\C^2$. Consequently, there is a numerically hypercyclic operator on $\C^n$ for every $n\geq2$.
\end{thm1}

\begin{thm2} Let $1<p<\infty$ and $T$ be a backward or forward weighted shift on $\ell^p(\N)$ or a bilateral weighted shift on $\ell^p(\Z)$. Then $T$ is numerically hypercyclic if and only if $T$ is not power bounded.
\end{thm2}

\begin{remark}\label{kpsrem} \rm Theorem~KPS2 is not exactly what is stated in \cite{nhy}. Namely, the corresponding results in \cite{nhy} provide a condition on the weight sequence equivalent to the numeric hypercyclicity. For bilateral weighted shifts the authors notice that this condition is equivalent to the operator being not power bounded, while saying nothing for unilateral shifts. However upon reflection, one can readily see that this condition is again equivalent to the operator being not power bounded. Furthermore in the case of unilateral backward shifts the same condition is equivalent to hypercyclicity. One has just co compare it to the characterization of Salas \cite{sal} of hypercyclic backward shifts. On the other hand, there are no hypercyclic forward shifts and there are non-hypercyclic bilateral shifts, which are not power bounded.
\end{remark}

Recall that $T\in L(X)$ and $S\in L(Y)$ are called {\it similar} if there is an isomorphism (in the category of topological vector spaces) $J:X\to Y$ such that $J^{-1}SJ=T$. Note that hypercyclicity and weak hypercyclicity are similarity invariants, while the numeric hypercyclicity is not (below we provide plenty of examples substantiating the last claim). This prompts us to consider the following concepts.

\begin{definition}\label{nuhy2}\rm We say that $T\in L(X)$ is {\it weakly numerically hypercyclic} if $T$ is similar to a numerically hypercyclic operator. We say that $T\in L(X)$ is {\it strongly numerically hypercyclic} if every operator similar to $T$ is numerically hypercyclic. We denote the sets of weakly numerically hypercyclic operators and of strongly numerically hypercyclic operators on $X$ by the symbols $\WNH(X)$ and $\SNH(X)$ respectively.
\end{definition}

The following elementary proposition characterizes weak numeric hypercyclicity.

\begin{proposition}\label{ele1} Let $T\in L(X)$. Then $T$ is weakly numerically hypercyclic if and only if there exist $x\in X$ and $f\in X^*$ such that $O(T,x,f)$ is dense in $\C$.
\end{proposition}

\begin{proof} The 'only if' part is obvious. Let $x\in X$ and $f\in X^*$ be such that $O(T,x,f)$ is dense in $\C$. It remains to show that $T\in \WNH(X)$. Replacing $x$ by $T^nx$ for an appropriate $n\in\N$ and normalizing, we can without loss of generality assume that $\|x\|=f(x)=1$. Let $D=\{x\in X:\|x\|\leq 1,\ |f(x)|\leq 1\}$ and $\|\cdot\|_1$ be the Minkowski functional of $D$. Since $D$ is bounded and contains a neighborhood of $0$ in $X$, $\|\cdot\|_1$ is a norm on $X$ equivalent to the original norm $\|\cdot\|$. It is easy to see that $\|x\|_1=\|f\|_1=f(x)=1$, where $\|f\|_1$ is the norm of $f$ as a functional on the normed space $X_1=(X,\|\cdot\|_1)$. Thus $(x,f)\in \Pi(X_1)$ and therefore $(x,f)$ is a numerically hypercyclic vector for $T$ acting on $X_1$. Since $\|\cdot\|$ and $\|\cdot\|_1$ are equivalent, $T$ acting on $X$ is similar to $T$ acting on $X_1$. Thus $T$ is weakly numerically hypercyclic.
\end{proof}

According to Feldman \cite{fe}, $T\in L(X)$ is called $1$-{\it weakly hypercyclic} if there is $x\in X$ such that $f(O(T,x))$ is dense in $\C$ for each non-zero $f\in X^*$. Of course, every weakly hypercyclic operator is $1$-weakly hypercyclic. We collect a number of straightforward consequences of Definition~\ref{nuhy2}
in the following proposition. The proof is elementary and is left for the reader as an exercise.

\begin{proposition}\label{ele00} Let $T\in L(X)$.
\begin{itemize}\itemsep=-2pt
\item
If $T$ is power bounded, then $T\notin\WNH(X)$. 
If $T$ is $1$-weakly hypercyclic, then $T\in \SNH(X)$.
\item
If a restriction of $T$ to an invariant closed linear subspace is weakly numerically hypercyclic $($res\-pec\-ti\-vely, strongly numerically hypercyclic$)$, then $T$ is weakly numerically hypercyclic $($respectively, strongly numerically hypercyclic$)$.
\item
If $T\in \WNH(X)$, then $T^*\in \WNH(X^*)$. If additionally $X$ is reflexive, then $T\in \WNH(X)\iff T^*\in \WNH(X^*)$ and $T\in \SNH(X)\iff T^*\in \SNH(X^*)$.
\end{itemize}
\end{proposition}

\subsection{Finite dimensional results}

The following two questions are raised in \cite{nhy}.

\begin{que1} Let $T\in\NH(X)$. Is $T^n$ numerically hypercyclic for every $n\in\N${\rm ?}
\end{que1}

\begin{que2} Let $T\in L(X)$ and there exist $(x_j,f_j)\in\Pi(X)$ for $1\leq j\leq n$ for which $\smash{\bigcup\limits_{j=1}^n O(T,x_j,f_j)}$ is dense in $\C$. Is $T$ numerically hypercyclic{\rm ?}
\end{que2}

We answer both questions negatively.

\begin{theorem}\label{aq1} There exists a diagonal $T\in \NH(\C^2)$ such that $T^2\notin \NH(\C^2)$.
\end{theorem}

\begin{theorem}\label{aq2} There exist a diagonal $T\in L(\C^4)$ and $x_1,x_2\in S(\C^4)$ such that $T\notin \NH(\C^4)$, while $\NO(T,x_1)\cup \NO(T,x_2)$  is dense in $\C$.
\end{theorem}

The following two results are motivated by Theorem~KPS1. They provide sufficient conditions of weak/strong numeric hypercyclicity in terms of the restriction to a finite dimensional invariant subspace.

\begin{theorem}\label{suffwnh} Each of the following conditions is sufficient for the weak numeric hypercyclicity of $T\in L(X)$.
\begin{itemize}\itemsep=-2pt
\item[{\rm (\ref{suffwnh}.1)}]
There exist $\lambda_1,\lambda_2\in\sigma_p(T)$ such that $|\lambda_1|=|\lambda_2|>1$ and $\frac{\lambda_1}{|\lambda_1|}$, $\frac{\lambda_2}{|\lambda_2|}$ are independent in $\T$.
\item[{\rm (\ref{suffwnh}.2)}]
There are independent $\lambda_1,\lambda_2\in\T$ such that $\ker(T-\lambda_jI)^2\neq \ker(T-\lambda_jI)$ for $j\in\{1,2\}$.
\item[{\rm (\ref{suffwnh}.3)}]
There exist $\lambda_1,\lambda_2,\lambda_3\in\sigma_p(T)$ such that $|\lambda_1|=|\lambda_2|>|\lambda_3|>1$, $\frac{\lambda_1}{\lambda_2}$ has infinite order in the group $\T$ and $\frac{\lambda_1}{|\lambda_1|}$, $\frac{\lambda_3}{|\lambda_3|}$ are independent.
\item[{\rm (\ref{suffwnh}.4)}]
There exist $\lambda_1,\lambda_2\in\sigma_p(T)$  and $\lambda_3\in\T$ such that $\ker(T-\lambda_3I)^2\neq \ker(T-\lambda_3I)$, $|\lambda_1|=|\lambda_2|>1$, $\frac{\lambda_1}{\lambda_2}$ has infinite order in $\T$ and $\frac{\lambda_1}{|\lambda_1|}$, $\lambda_3$ are independent.
\end{itemize}
\end{theorem}

\begin{theorem}\label{suffsnh} Each of the following conditions is sufficient for the strong numeric hypercyclicity of $T\in L(X)$.
\begin{itemize}\itemsep=-2pt
\item[{\rm (\ref{suffsnh}.1)}]
There exist $\lambda_1,\dots,\lambda_n\in\sigma_p(T)$ and $c_1,\dots,c_n\in\R_+$ such that
$\Bigl\{\sum\limits_{j=1}^n c_j\lambda_j^k:k\in\Z_+\Bigr\}$ is dense in $\C$.
\item[{\rm (\ref{suffsnh}.2)}]
There exist $\lambda_1,\lambda_2,\lambda_3\in\sigma_p(T)$ such that $|\lambda_1|=|\lambda_2|=|\lambda_3|>1$ and $\frac{\lambda_1}{|\lambda_1|}$, $\frac{\lambda_2}{|\lambda_2|}$, $\frac{\lambda_3}{|\lambda_3|}$ are independent.
\end{itemize}
\end{theorem}

In the Hilbert space situation we can refine the first two parts of Theorem~\ref{suffwnh}.

\begin{proposition}\label{suffnh} Let $\hh$ be a Hilbert space and $T\in L(\hh)$. Assume that there exist $\lambda_1,\lambda_2\in\sigma_p(T)$ such that $|\lambda_1|=|\lambda_2|>1$,
the eigenspaces $\ker(T-\lambda_1 I)$ and $\ker(T-\lambda_2I)$ are non-orthogonal and
$\frac{\lambda_1}{|\lambda_1|}$, $\frac{\lambda_2}{|\lambda_2|}$ are independent in $\T$. Then $T\in\NH(\hh)$.
\end{proposition}

\begin{proposition}\label{suffnh1} Let $\hh$ be a Hilbert space and $T\in L(\hh)$. Assume that there exist $\lambda_1,\lambda_2\in\C$ such that $\ker(T-\lambda_jI)^2\neq \ker(T-\lambda_jI)$ for $j\in\{1,2\}$, $|\lambda_1|=|\lambda_2|\geq1$ and $\frac{\lambda_1}{|\lambda_1|}$, $\frac{\lambda_2}{|\lambda_2|}$ are independent in $\T$. Then $T\in\NH(\hh)$.
\end{proposition}

We characterize $\NH(\C^2)$, $\SNH(\C^2)$, $\WNH(\C^2)$ and $\WNH(\C^3)$.

\begin{theorem}\label{2dim} Let $T\in L(\C^2)$. Then $T\in\WNH(\C^2)$ if and only if $\sigma(T)=\{\lambda_1,\lambda_2\}$, where $|\lambda_1|=|\lambda_2|>1$ and $\frac{\lambda_1}{|\lambda_1|}$, $\frac{\lambda_2}{|\lambda_2|}$ are independent in $\T$. Furthermore, $T\in \SNH(\C^2)$ if and only if $\sigma(T)=\{\lambda_1,\lambda_2\}$ with $\{\lambda_1^k+\lambda_2^k:k\in\Z_+\}$ being dense in $\C$. Finally, $T\in \NH(\C^2)$ if and only if either $T\in\SNH(\C^2)$ or $T\in \WNH(\C^2)$ and $T$ is not unitarily equivalent to a diagonal operator.
\end{theorem}

\begin{theorem}\label{3dim} Let $T\in L(\C^3)$. Then $T$ is weakly numerically hypercyclic if and only if
either there are $\lambda_1,\lambda_2\in\sigma(T)$ such that $|\lambda_1|=|\lambda_2|>1$ and $\frac{\lambda_1}{|\lambda_1|}$, $\frac{\lambda_2}{|\lambda_2|}$ are independent in $\T$ or $\sigma(T)=\{\lambda_1,\lambda_2,\lambda_3\}$, where $|\lambda_1|=|\lambda_2|>|\lambda_3|>1$, $\frac{\lambda_1}{\lambda_2}$ has infinite order in the group $\T$ and $\frac{\lambda_1}{|\lambda_1|}$, $\frac{\lambda_3}{|\lambda_3|}$ are independent in $\T$.
\end{theorem}

\begin{remark}\label{examples}\rm
The above results provide easy explicit examples of numerically hypercyclic operators on $\C^n$. For instance, applying Theorems~\ref{2dim}, Theorem~\ref{suffsnh} and Proposition~\ref{suffnh1} respectively, we have
\begin{align*}
&\left(\begin{array}{cc} 2e^{i\pi\log 2}&1\\ 0&2e^{i\pi\log 3} \end{array}\right)\in \NH(\C^2),\quad \left(\begin{array}{ccc} 2e^{i\pi\log 2}&0&0\\
0&2e^{i\pi\log 3}&0\\
0&0&2e^{i\pi\log 5}
\end{array}\right)\in \SNH(\C^3)\subset \NH(\C^3)
\\
&\text{and}\quad\left(\begin{array}{cccc}
e^{i\pi\log 2}&1&0&0\\
0&e^{i\pi\log 2}&0&0\\
0&0&e^{i\pi\log 3}&1\\
0&0&0&e^{i\pi\log 3}
\end{array}\right)\in \NH(\C^4).
\end{align*}
\end{remark}

In relation to Theorem~\ref{2dim} it makes sense to mention the following fact.

\begin{proposition}\label{znwn} For every $R>1$, 
$\bigl\{(z,w)\in\T^2:\{R^n(z^n+w^n):n\in\Z_+\}\ \text{is dense in $\C$}\bigr\}$
is a dense $G_\delta$-subset of $\T^2$ of Lebesgue measure $0$.
\end{proposition}

\begin{remark}\label{wnhs} \rm
Let $\{x,y\}$ be a fixed linear basis in $\C^2$ and $R>1$. For each $(z,w)\in \T^2$, consider $T=T_{z,w}\in L(\C^2)$ given by $Tx=Rzx$ and $Ty=Rwy$. By Theorem~\ref{2dim} and Proposition~\ref{znwn}, for almost all $(z,w)\in \T^2$ in the Lebesgue measure sense, $T_{z,w}\in \WNH(\C^2)\setminus \NH(\C^2)$ if $\langle x,y\rangle =0$ and $T_{z,w}\in \NH(\C^2)\setminus \SNH(\C^2)$ if $\langle x,y\rangle \neq  0$. In particular, there are numerically hypercyclic operators on $\C^2$ which are not strongly numerically hypercyclic and there are weakly numerically hypercyclic operators on $\C^2$ that are not numerically hypercyclic.
\end{remark}

Due to  Le\'on-Saavedra and M\"uller \cite{muller}, if $T\in L(X)$ is hypercyclic or weakly hypercyclic, then so is $zT$ for each $z\in\T$. Moreover, Theorem~KPS2 ensures that if a weighted shift $T$ is numerically hypercyclic, then so is $zT$ for each $z\in\C$ satisfying $|z|\geq1$. In general, the scalar multiples of numerically hypercyclic operators do not exhibit this nice behavior even in the friendly realm of diagonal operators on $\C^2$.

\begin{proposition}\label{scalar} For every $T\in L(\C^2)$, there is $w\in\T$ such that $wT\notin\WNH(\C^2)$.
\end{proposition}

\begin{proposition}\label{scalar2} For each $z,w\in\T$, 
$\displaystyle M(z,w)=\biggl\{r>1:\left(\!\begin{array}{cc}\!\!rz\!\!&0\\ 0&\!\!rw\!\!\end{array}\!\right)\in\NH(\C^2)\biggr\}$ is a $G_\delta$-subset of $\R$, being either infinite or empty. There exist $w,z\in\T$ such that $M(z,w)$ is a dense $G_\delta$-subset of $(1,\infty)$ and therefore is uncountable. For every $w,z\in\T$, $M(z,w)$ has zero Lebesgue measure.
\end{proposition}

By Proposition~\ref{scalar2}, not a single diagonal $T\in L(\C^2)$ satisfies $rT\in \NH(\C^2)$ for each $r>1$.

\subsection{Infinite dimensional results}

The following two theorems provide further sufficient conditions for weak/strong numeric hypercyclicity.

\begin{theorem}\label{suffwnhid} Each of the following conditions is sufficient for the weak numeric hypercyclicity of $T\in L(X)$.
\begin{itemize}\itemsep=-2pt
\item[{\rm (\ref{suffwnhid}.1)}] There is a sequence $\{z_n\}_{n\in\N}$ in $\sigma_p(T)$ such that $1<|z_1|<|z_2|<{\dots}$
\item[{\rm (\ref{suffwnhid}.2)}] There is $r>1$ such that $\sigma_p(T)\cap r\T$ is infinite.
\item[{\rm (\ref{suffwnhid}.3)}] There is a cyclic vector $x$ for $T$ satisfying $\liminf\limits_{n\to\infty}\frac{1+\|T^nx\|}{\|T^n\|}=0$.
\item[{\rm (\ref{suffwnhid}.4)}] There are $x\in X$ and an infinite set $A\subseteq \N$ such that $\|T^nx\|\to\infty$ as $n\to\infty$, $n\in A$ and the sequence $\bigl\{\frac{T^nx}{\|T^nx\|}\bigr\}_{n\in A}$ is weakly convergent but is not norm convergent.
\end{itemize}
\end{theorem}

The following examples show how to apply Theorem~\ref{suffwnhid}.

\begin{example}\label{vol} $I+V\in\WNH(L^2[0,1])$, where $Vf(x)=\int_0^x f(t)\,dt$ is the classical Volterra operator. \end{example}

\begin{proof}
It is easy to verify (see, for instance, \cite{mbs}) that $\|T^n\|\to\infty$ and $\|T^ng\|=o(\|T^n\|)$, where $T=I+V$ and $g(x)=x$ for $x\in[0,1]$. Since $g$ is a cyclic vector for $T$, (\ref{suffwnhid}.3) is satisfied and $T\in \WNH(L^2[0,1])$ by Theorem~\ref{suffwnhid}. 
\end{proof}

\begin{example}\label{conti} Let $a\in C[0,1]$ be non-constant and $\|a\|>1$. Then $M\in \WNH(C[0,1])$, where $M$ is the multiplication by $a$ operator: $Mf=af$ for $f\in C[0,1]$.
\end{example}

\begin{proof} Let $K=\{t\in [0,1]:|a(t)|=\|a\|\}$. If $K\neq [0,1]$, we can pick $f\in C[0,1]$ such that $K=\{t\in[0,1]:f(t)=0\}$ and let $X$ be the closed linear span of $O(M,f)$. It is easy to see that $\|M^n\bigr|_X\|=\|a\|^n$ for every $n\in\Z_+$ and $\|T^nf\|=o(\|a\|^n)$ as $n\to\infty$. Thus (\ref{suffwnhid}.3) is satisfied for $M\bigr|_X$ and therefore $M\bigr|_X\in\WNH(X)$. By Proposition~\ref{ele}, $M\in\WNH(C[0,1])$. It remains to consider the case $K=[0,1]$. Let $R=\|a\|$. Since $a$ is non-constant and $|a|\equiv R$, Proposition~\ref{znwn} ensures that we can pick $w,z\in\T$ such that $Rz,Rw\in a([0,1])$ and $\{R^n(z^n+w^n):n\in\Z_+\}$ is dense in $\C$. Choose $s,t\in[0,1]$ such that $a(s)=Rz$ and $a(t)=Rw$ and let ${\bf 1}$ be the constant $1$ function on $[0,1]$. Set $\phi\in C[0,1]^*$ to be $\phi(f)=\frac12(f(s)+f(t))$. Clearly, $({\bf 1},\phi)\in \Pi(C[0,1])$ and $\phi(M^n{\bf 1})=\frac12R^n(z^n+w^n)$ for $n\in\Z_+$. Thus $O(M,{\bf 1},\phi)$ is dense in $\C$ and therefore $M\in \NH(C[0,1])\subset \WNH(C[0,1])$.
\end{proof}

\begin{theorem}\label{suffsnhid} Each of the following conditions is sufficient for the strong numeric hypercyclicity of $T\in L(X)$.
\begin{itemize}\itemsep=-2pt
\item[{\rm (\ref{suffsnhid}.1)}]
$X$ is reflexive and there are a Schauder basic sequence $\{e_n\}_{n\in\N}$ in $X$ and $c\in\ell^1_+(\N)$ such that $Te_n=\lambda_ne_n$ with $\lambda_n\in\C$ for each $n\in\N$ and 
$\Bigl\{\sum\limits_{j=1}^\infty c_j\lambda_j^k:k\in\Z_+\Bigr\}$ is dense in $\C$.
\item[{\rm (\ref{suffsnhid}.2)}] $X$ is reflexive and there is a Schauder basic sequence $\{e_n\}_{n\in\N}$ in $X$ such that $Te_n=\lambda_ne_n$ for each $n\in\N$, where $\lambda_n\in\C$ are such that $|\lambda_1|>1$, the sequence $\{|\lambda_n|\}_{n\in\N}$ is $($maybe non-strictly$)$ increasing and the numbers $\frac{\lambda_n}{|\lambda_n|}$ are pairwise distinct elements of $\T$.
\item[{\rm (\ref{suffsnhid}.3)}]
There is $\lambda\in\C$ such that $|\lambda|\geq 1$ and $T-\lambda I$ is semi-Fredholm of positive index.
\item[{\rm (\ref{suffsnhid}.4)}]
$X$ is reflexive and there is $\lambda\in\C$ such that $|\lambda|\geq 1$ and $T-\lambda I$ is semi-Fredholm of negative index.
\end{itemize}
\end{theorem}

The reflexivity condition in the above theorem can not be removed entirely as illustrated by the following examples.

\begin{example}\label{c0e1} Let $\{z_n\}_{n\in\N}$ be a sequence of finite order elements of $\T$, $r>1$ and $T\in L(c_0(\N))$ be the diagonal operator with the numbers $rz_n$ on the diagonal. Then $T\notin \NH(c_0(\N))$.
\end{example}

\begin{proof} Let $(x,f)\in \Pi(c_0(\N))$ and $A=\{n\in\N:|x_n|=1\}$. Clearly $A$ is non-empty and finite and there is $c\in\R_+^A$ such that $f$ acts according to the formula
$f(y)=\sum\limits_{n\in A} c_n\overline{x}_ny_n$ for each $y\in c_0(\N)$.
Hence $f(T^kx)=r^k\sum\limits_{n\in A} c_nz_n^k$ for each $k\in\Z_+$. Since $A$ is finite and each $z_n$ has finite order, $O(T,x,f)$ is contained in the union of finitely many lines in $\C$ through the origin. Hence $O(T,x,f)$ is nowhere dense in $\C$ and $T\notin\NH(c_0(\N))$.
\end{proof}

\begin{example}\label{c0e2} Let $r>1$ and $T\notin \NH(c_0(\N))$ be given by $(Tx)_1=0$ and $(Tx)_n=rx_{n-1}$ if $n>1$. Then $T-I$ is Fredholm of index $-1$ and $T\notin\NH(c_0(\N))$.
 \end{example}

\begin{proof} Let $(x,f)\in \Pi(c_0(\N))$ and the finite set $A\subset \N$ and $c\in\R_+^A$ be as in the proof of Example~\ref{c0e1}. It is easy to see that $f(T^kx)=0$ if $k\geq \max(A)$. Thus $O(T,x,f)$ is a finite set. Hence $T\notin\NH(c_0(\N))$. On the other hand, $T-I$ is injective and $(T-I)(c_0(\N))$ is exactly the kernel of $\phi\in c_0(\N)^*$ given by $\phi(x)=\sum\limits_{n=1}^\infty r^{-n}x_n$. Thus $T-I$ is Fredholm of index $-1$. 
\end{proof}

Obviously, a self-adjoint operator $T$ on a Hilbert space $\hh$ can not be numerically hypercyclic. Indeed, $\NO(T,x)\subset\R$ for each $x\in S(\hh)$. The following theorem nearly (but not quite) characterizes weakly numerically hypercyclic normal operators. It also allows a funny characterisation of  weakly numerically hypercyclic self-adjoint operators.

\begin{theorem}\label{normaLLL} Let $\hh$ be a Hilbert space, $T\in L(\hh)$ and $k\in\N$ be such that $T^k$ is normal. Then the following statements are true.
\begin{itemize}\itemsep=-2pt
\item[{\rm (\ref{normaLLL}.1)}]
$T\in\WNH(\hh)$ if there is a sequence $\{\lambda_n\}_{n\in\N}$ in $\sigma(T)$ such that $1<|\lambda_1|<|\lambda_2|<{\dots}$
\item[{\rm (\ref{normaLLL}.2)}]
$T\in\NH(\hh)$ if there is a sequence $\{\lambda_n\}_{n\in\N}$ in $\sigma(T)$ such that $1<|\lambda_1|<|\lambda_2|<{\dots}$ and the numbers $\frac{\lambda_j}{|\lambda_j|}$ are pairwise distinct.
\item[{\rm (\ref{normaLLL}.3)}]
If $T^k$ is self-adjoint, then $T\in\WNH(\hh)$ if and only if the set $\{-|z|:z\in\sigma(T),\ |z|>1\}$ is not well-ordered by the natural ordering of $\R$.
\item[{\rm (\ref{normaLLL}.4)}]
$T\in\NH(\hh)$ if there exist $r>1$ and a $T^k\!$-invariant subspace ${\cal K}$ such that $\frac1rT^k\bigr|_{\cal K}\in L({\cal K})$ is a unitary operator with infinite spectrum.
\end{itemize}
\end{theorem}

\begin{corollary}\label{sanh} A self-adjoint operator $T$ is weakly numerically hypercyclic if and only if the set\break $\{-|z|:z\in\sigma(T),\ |z|>1\}$ is not well-ordered by the natural ordering of $\R$.
\end{corollary}

\begin{corollary}\label{uni} Let $r>1$ and $T$ be a unitary operator with infinite spectrum. Then $rT$ is numerically hypercyclic.
\end{corollary}

Herrero \cite{her,herw} described the norm closure of the set of hypercyclic operators on $\ell^2(\N)$ in terms of the spectrum. Pr\v ajitur\v a \cite{pra} demonstrated that the set of weakly hypercyclic operators on $\ell^2(\N)$ has the same closure. We provide a spectral description of the closures of $\SNH(X)$, $\WNH(X)$ and $\NH(X)$ for an arbitrary reflexive Banach space $X$. 

\begin{theorem}\label{clos} Let $X$ be reflexive and $T\in L(X)$. Then the following statements are equivalent$:$
\begin{itemize}\itemsep=-2pt
\item[{\rm (\ref{clos}.1)}]
$T$ does not belong to the operator norm closure of $\SNH(X);$
\item[{\rm (\ref{clos}.2)}]
$T$ does not belong to the operator norm closure of $\NH(X);$
\item[{\rm (\ref{clos}.3)}]
$T$ does not belong to the operator norm closure of $\WNH(X);$
\item[{\rm (\ref{clos}.4)}]
the set $\{\lambda\in\sigma(T):|\lambda|\geq 1\}$ consists of finitely many normal eigenvalues of multiplicity $1$ with pairwise distinct absolute values.
\end{itemize}
\end{theorem} 

\subsection{The structure} 

In Section~\ref{s2} we prove Theorems~\ref{aq1} and~\ref{aq2} by means of explicit examples. In the same section we show that the existence of a numerical orbit $\NO(T,x)$ dense in a non-empty open subset of $\C$ does not guarantee the numeric hypercyclicity of $T\in L(\C^2)$. In Section~\ref{s3} we prove Proposition~\ref{znwn} and build up a toolbox used later on in order to prove the main results. Section~\ref{s4} is devoted to the proof of Theorems~\ref{suffwnh} and~\ref{suffsnh}, which provide sufficient conditions of weak and of strong numeric hypercyclicity in terms of a restriction to a finite dimensional invariant subspace. In Section~\ref{s5}, we prove Propositions~\ref{suffnh} and~\ref{suffnh1}, which deal with the similar issues in the Hilbert space setting.
Theorem~\ref{suffwnhid}, providing genuinely infinite dimensional sufficient conditions of weak numeric hypercyclicity, is proved in Section~\ref{s6}. We prove Theorems~\ref{2dim}, describing the numerically hypercyclic operators on $\C^2$, and Theorem~\ref{3dim}, describing the weakly numerically hypercyclic operators on $\C^3$, in Section~\ref{s7}. Theorem~\ref{suffsnhid}, which provides infinite dimensional sufficient conditions for the strong numeric hypercyclicity, is proved in Section~\ref{s8}. In Section~\ref{s9}, we prove Theorem~\ref{clos}, which describes the operator norm closure of the set of numerically hypercyclic operators on a reflexive Banach space. Theorem~\ref{normaLLL}, dealing with numeric hypercyclicity of normal operators, is proved in Section\ref{s10}. In Section~\ref{s11}, we prove Propositions~\ref{scalar} and~\ref{scalar2} on scalar multiples of numerically hypercyclic operators on $\C^2$. In Section~\ref{s12}, we make some further remarks, construct few examples and raise several questions on the numeric hypercyclicity.

\section{Proof of Theorems~\ref{aq1} and~\ref{aq2}\label{s2}}

Our construction is similar to the proof of Theorem~KPS1 in \cite{nhy}. The main difference is that we need and achieve better control of the numbers in the numerical orbit. 
Note also that there are natural generalizations and some estimates below are far from optimal.

\subsection{Funny numbers}

As usual, for $m,n\in \N$, we write $m\divi n$ if $m$ is a divisor of $n$.
Let $p$ be an odd prime number and $\nu_k=\nu_k(p)$ for $k\in\N$ be powers of $p$ defined recurrently by the formula
\begin{equation}\label{nuk}
\text{$\nu_1=p$ and $\nu_{k+1}=\nu_k p^{k+\nu_k}$ for $k\in\N$.} 
\end{equation}
Let also ${\bf w}=\{w_k\}_{k\in\N}$ be a fixed sequence of complex numbers satisfying
\begin{equation}\label{wk}
\textstyle \text{$\frac1k\leq |w_k|\leq k$ \ for each $k\in\N$}.
\end{equation}
We define a sequence $\{m_k=m_k(p,{\bf w})\}_{k\in\N}$ of positive integers according to the following rule:
\begin{equation}\label{mk}\textstyle
\text{$m_1=1$ and $m_k=\min\{m\in\N:m>\frac{|w_{k-1}|p^{k-1}}{\pi},\ 2\divi m,\ p\ndivides m\}$ for $k\geq 2$}.
\end{equation}
Since the gaps between consecutive members of $\{m\in\N:2|m,\ p\ndivides m\}$ do not exceed $4$, we have
\begin{equation}\label{mk1}\textstyle
\text{$0\leq m_k-\frac{|w_{k-1}|p^{k-1}}{\pi}<4$ for every $k\geq 2$}.
\end{equation}
According to (\ref{wk}) and (\ref{mk1}),
\begin{equation}\label{mk2}
m_k=O(kp^k)\ \ \text{as $k\to\infty$}.
\end{equation}
By (\ref{nuk}) and (\ref{mk2}), the series $\sum \frac{m_k}{\nu_k}$ converges, which allows us to define
\begin{equation}\label{tauz}\textstyle
\text{$\tau=\tau(p,{\bf w})=\sum\limits_{k=1}^\infty \frac{m_k}{\nu_k}\in\R$\ \ and\ \
$z=z(p,{\bf w})=e^{\pi \tau i}\in \T$}.
\end{equation}

\begin{lemma}\label{esti1} In the above notation the following asymptotic relations hold$:$
\begin{align}\label{linuk}
&\lim_{k\to\infty}\bigl(p^{\nu_k}|1+z^{\nu_k}|-|w_k|\bigr)=0;
\\
&\liminf_{n\to\infty\atop n\notin\{\nu_k:k\in\N\}}|1+z^n|^{1/n}\geq p^{-1/3}. \label{linnuk}
\end{align}
\end{lemma}

\begin{proof} Note that for each $k\in\N$,
\begin{equation}\label{NK}\textstyle
\text{$\nu_k\tau =N_k+\sum\limits_{s=k+1}^\infty \frac{m_s\nu_k}{\nu_s}$,\ \ where $N_k=\sum\limits_{s=1}^k \frac{m_s\nu_k}{\nu_s}$ is an odd integer.}
\end{equation}
Since $e^{\pi N_ki}=-1$, we have
$z^{\nu_k}=-\exp\Bigl(\pi i\sum\limits_{s=k+1}^\infty \frac{m_s\nu_k}{\nu_s}\Bigr)$.
Hence 
$$\textstyle
1+z^{\nu_k}=-\pi i\frac{m_{k+1}\nu_k}{\nu_{k+1}}+O\biggl(\Bigl(\frac{m_{k+1}\nu_k}{\nu_{k+1}}\Bigr)^2+
\sum\limits_{s=k+2}^\infty \frac{m_{s}\nu_k}{\nu_{s}}\biggr)\ \ \text{as $k\to\infty$}.
$$
Multiplying by $p^{\nu_k}$ and using (\ref{nuk}), we get
$$\textstyle
p^{\nu_k}|1+z^{\nu_k}|=\frac{\pi m_{k+1}}{p^k}+O\biggl(\frac{m^2_{k+1}\nu_k}{\nu_{k+1}p^k}+
\sum\limits_{s=k+2}^\infty \frac{m_{s}\nu_kp^{\nu_k}}{\nu_{s}}\biggr)\ \ \text{as $k\to\infty$}.
$$
By (\ref{nuk}) and (\ref{mk2}), the sequence in the $O$ sign converges to 0 as $k\to\infty$. Hence
$p^{\nu_k}|1+z^{\nu_k}|-\frac{\pi m_{k+1}}{p^k}\to 0$ as $k\to\infty$. The inequality (\ref{mk1}) yields
$\frac{\pi m_{k+1}}{p^k}-|w_k|\to 0$ as $k\to\infty$ and (\ref{linuk}) follows.

We prove (\ref{linnuk}) in two steps. Denote $\Lambda=
\Bigl\{\nu_km:k\in\N,\ m\in2\N+1,\ m<\sqrt{\frac{\nu_{k+1}}{\nu_k}}\Bigr\}$. First, we shall verify that
\begin{equation}\label{linnuk1}
\lim_{n\to\infty\atop n\notin\Lambda}|1+z^n|^{1/n}=1.
\end{equation}
If $n\in\N$ is large enough, there is a unique $k=k(n)\geq 2$ such that $\sqrt{\nu_{k-1}\nu_k}<n\leq \sqrt{\nu_{k}\nu_{k+1}}$.
By (\ref{NK}),
$$\textstyle
n\tau=\frac{nN_k}{\nu_k}+\sum\limits_{s=k+1}^\infty \frac{m_sn}{\nu_s}.
$$
Using the estimate $n\leq \sqrt{\nu_{k}\nu_{k+1}}$, (\ref{nuk}) and (\ref{mk2}), we obtain
$$\textstyle
\sum\limits_{s=k+1}^\infty \frac{m_sn}{\nu_s}=O\Bigl(\frac1{\nu_k^2}\Bigr).
$$
If $n\notin\Lambda$, $\frac{n}{\nu_k}$ is not an odd integer.
Since $p\ndivides N_k$, $\frac{nN_k}{\nu_k}$ can not be an odd integer. Hence the distance from $\frac{nN_k}{\nu_k}$ to the nearest odd integer is at least $\frac1{\nu_k}$.
Then by the last two displays,
$$\textstyle
|1+z^n|=|1+e^{\pi n\tau i}|\geq \frac1{2\nu_k}\ \ \text{if $n\notin\Lambda$ is sufficiently large.}
$$
Since $\sqrt{\nu_k}<n$, $2\geq |1+z^n|\geq \frac1{2n^2}$
for all sufficiently large $n\notin\Lambda$, which immediately implies (\ref{linnuk1}).

In order to complete the proof of (\ref{linnuk}) and that of the lemma, it now suffices to show that
\begin{equation}\label{linnuk2}
\liminf_{n\to\infty\atop n\in\Lambda}|1+z^{n}|^{1/n}\geq p^{-1/3}.
\end{equation}

By the already verified equality (\ref{linuk}), $\lim\limits_{k\to\infty}\bigl(p^{\nu_k}|1+z^{\nu_k}|-|w_k|\bigr)=0$. Combining it with (\ref{wk}), we have
$\frac{p^{-\nu_k}}{k}=O(|1+z^{\nu_k}|)$ and $|1+z^{\nu_k}|=O(kp^{-\nu_k})$ as $k\to\infty$.
Hence, we can find a sequence $r_k$ of real numbers such that
$$\textstyle
z^{\nu_k}=-e^{ir_k}\ \ \text{for every $k\in\N$},\ \  \frac{p^{-\nu_k}}{k}=O(|r_k|)\ \ \text{and}\ \ r_k=O(kp^{-\nu_k})\ \ \text{as $k\to\infty$}.
$$
The last relation yields $r_k\sqrt{\nu_{k+1}\nu_k}\to 0$ as $k\to\infty$. Thus $|nr_k|<\pi$ for $n=\nu_km\in\Lambda$ for all sufficiently large $k$. Using the inequality $|e^{it}-1|\geq \frac2\pi |t|$ if $|t|\leq \pi$, we see that if $|nr_k|<\pi$ and $n=\nu_km\in\Lambda$ is sufficiently large, then
$$\textstyle
|1+z^n|=|1-e^{ir_km}|\geq\frac2\pi |r_k|m\geq c \frac{mp^{-\nu_k}}{k}\geq \frac{cp^{-\nu_k}}{n}
$$
for some constant $c>0$ independent on $n$ ($c$ does exist because $\frac{p^{-\nu_k}}{k}=O(|r_k|)$).
Hence
$$
|1+z^n|^{1/n}\geq  (c/n)^{1/n} \bigl(p^{-\nu_k}\bigr)^{1/\nu_km}= (c/n)^{1/n} p^{-1/m}\geq
(c/n)^{1/n} p^{-1/3}\to p^{-1/3}
$$
since $m\geq 3$. Thus the lower limit of $|1+z^n|^{1/n}$ is at least $p^{-1/3}$ as $n\to\infty$, $n\in\Lambda$. This completes the proof of (\ref{linnuk}) and that of the lemma.
\end{proof}

Keeping the above notation, for each $k\in\N$, we choose $q_k\in\{0,1,\dots,p^k-1\}$ such that
\begin{equation}\label{qk} \textstyle
\Bigl|\frac{w_k|1+z^{\nu_k}|}{|w_k|(1+z^{\nu_k})}-\exp\Bigl(\frac{2\pi q_ki}{p^k}\Bigr)\Bigr|=
\min\biggl\{\Bigl|\frac{w_k|1+z^{\nu_k}|}{|w_k|(1+z^{\nu_k})}-\exp\Bigl(\frac{2\pi qi}{p^k}\Bigr)\Bigr|:
q\in\{0,1,\dots,p^k-1\}\biggr\}.
\end{equation}
The inequality $0\leq q_k<p^k$ and (\ref{nuk}) ensure that the series $\sum \frac{q_k}{p^k\nu_k}$ converges. Thus we can define
\begin{equation}\label{theu}\textstyle
\theta=\theta(p,{\bf w})=\sum\limits_{k=1}^\infty \frac{q_k}{p^k\nu_k}\in\R,\ \ \text{and}\ \
u=u(p,{\bf w})=e^{2\pi \theta i}\in \T.
\end{equation}

\begin{lemma}\label{esti2} In the above notation,
\begin{equation}\label{linuk11}
\lim_{k\to\infty}\bigl(u^{\nu_k}p^{\nu_k}(1+z^{\nu_k})-w_k\bigr)=0.
\end{equation}
\end{lemma}

\begin{proof} By (\ref{qk}), $\Bigl|\frac{w_k|1+z^{\nu_k}|}{|w_k|(1+z^{\nu_k})}-\exp\Bigl(\frac{2\pi q_ki}{p^k}\Bigr)\Bigr|\leq\frac{\pi}{p^k}$.
By Lemma~\ref{esti1}, $\bigl(p^{\nu_k}|1+z^{\nu_k}|-|w_k|\bigr)\to 0$. These two relations and the estimate $|w_k|\leq k$ imply that
\begin{equation}\textstyle\label{e11}
\exp\Bigl(\frac{2\pi q_ki}{p^k}\Bigr)p^{\nu_k}(1+z^{\nu_k})-w_k\to 0\ \ \text{as $k\to\infty$}.
\end{equation}

By (\ref{theu}),
$$\textstyle
u^{\nu_k}\exp\Bigl(-\frac{2\pi q_ki}{p^k}\Bigr)=\exp\Bigl(\sum\limits_{s=k+1}^\infty \frac{2\pi\nu_kq_k i}{\nu_sp^s}\Bigr)=
1+O\Bigl(\sum\limits_{s=k+1}^\infty \frac{\nu_kq_k}{\nu_sp^s}\Bigr).
$$
Since $\sum\limits_{s=k+1}^\infty \frac{\nu_kq_k}{\nu_sp^s}=O(\nu_k/\nu_{k+1})=O(p^{-k-\nu_k})$ and 
$p^{\nu_k}|1+z^{\nu_k}|=O(|w_k|)=O(k)$, we have\break $(u^{\nu_k}-\exp(2\pi q_ki/p^k))p^{\nu_k}(1+z)^{\nu_k}\to 0$, which together with (\ref{e11}) implies (\ref{linuk11}).
\end{proof}

\subsection{$T\in L(\C^2)$ with a prescribed numerical orbit}

The following result shows that the somewhere dense orbit theorem fails miserably for numerical orbits.

\begin{proposition}\label{pno} Let a non-empty open subset $U$ of $\C$ be the interior of a closed set. Then there exist a diagonal $T\in L(\C^2)$ and $x_0\in S(\C^2)$ such that $U$ is the interior of $\overline{\NO(T,x_0)}$ and for every $x\in S(\C^2)$ either $\NO(T,x)=\NO(T,x_0)$ or $\NO(T,x)$ is nowhere dense in $\C$.
\end{proposition}

\begin{proof} Let $p$ be an odd prime number. Choose a sequence ${\bf w}=\{w_k\}_{k\in\N}$ in $\C$ such that $\frac1k\leq |w_k|\leq k$ for every $k\in\N$ and $\{w_k:k\geq k_0\}$ is a dense subset of $2U$ for some $k_0\in\N$. By Lemmas~\ref{esti1} and~\ref{esti2}, (\ref{linuk11}) and (\ref{linnuk}) hold for $\nu_k=\nu_k(p)$, $z=z(p,{\bf w})$ and $u=u(p,{\bf w})$ defined in (\ref{nuk}), (\ref{tauz}) and (\ref{theu}) respectively. Now let $T\in L(\C^2)$ be the diagonal operator with the numbers $pu$ and $puz$ on the diagonal.
If $x=(a,b)\in S(\C^2)$, then 
$$
\langle T^nx,x\rangle =p^nu^n(|a|^2+|b|^2z^n)\ \ \text{for every $n\in\Z_+$}.
$$
If $|a|\neq |b|$, the above display yields $|\langle T^nx,x\rangle|\geq p^n||a|^2-|b|^2|\to \infty$ and $\NO(T,x)$ is nowhere dense in $\C$. If $|a|=|b|$, then $|a|=|b|=\frac1{\sqrt 2}$ and the above display reads
$$
2\langle T^nx,x\rangle =p^nu^n(1+z^n)\ \ \text{for every $n\in\Z_+$}.
$$
By (\ref{linuk11}), $\langle T^{\nu_k}x,x\rangle-\frac{w_k}2\to 0$ as $k\to \infty$. Since $\{w_k/2:k\geq k_0\}$ is a dense subset of $U$ and $U$ is the interior of a closed set, $U$ is the interior of the closure of $\{\langle T^{\nu_k}x,x\rangle:k\in\N\}$. By the above display and (\ref{linnuk}),  $|\langle T^nx,x\rangle|\to \infty$ as $n\to\infty$, $n\notin\{\nu_k:k\in\N\}$. Hence $U$ is the interior of the closure of $\{\langle T^{n}x,x\rangle:n\in\N\}=\NO(T,x)$. It remains to note that if $|a|=|b|$, then $\NO(T,x)$ does not depend on $x$.
\end{proof}

\subsection{Proof of Theorem~\ref{aq1}}

Let $p$ be an odd prime number. Choose a sequence ${\bf w}=\{w_k\}_{k\in\N}$ in $\C$ such that $\frac1k\leq |w_k|\leq k$ for every $k\in\N$ and $\{w_k:k\in\N\}$ is dense in $\C$. By Lemmas~\ref{esti1} and~\ref{esti2}, (\ref{linuk11}) and (\ref{linnuk}) hold for $\nu_k=\nu_k(p)$, $z=z(p,{\bf w})$ and $u=u(p,{\bf w})$ defined in (\ref{nuk}), (\ref{tauz}) and (\ref{theu}) respectively. Now let $T\in L(\C^2)$ be the diagonal operator with the numbers $pu$ and $puz$ on the main diagonal.

First, if $x_0=\bigl(\frac1{\sqrt 2},\frac1{\sqrt 2}\bigr)\in S(\C^2)$, then
$2\langle T^nx_0,x_0\rangle =p^nu^n(1+z^n)$ for every $n\in\Z_+$. By (\ref{linuk11}), $\langle T^{\nu_k}x_0,x_0\rangle-\frac{w_k}2\to 0$ as $k\to \infty$. Since $\{w_k:k\in\N\}$ is dense in $\C$, $\{\langle T^{\nu_k}x_0,x_0\rangle:k\in\N\}$ is dense in $\C$. Hence $\NO(T,x)$ is dense in $\C$ and $T\in \NH(\C^2)$.

Let $x=(a,b)\in S(\C^2)$. If $|a|\neq |b|$, as in the proof of Proposition~\ref{pno}, we have $|\langle T^nx,x\rangle|\geq p^n||a|^2{-}|b|^2|\to \infty$ and $\NO(T,x)$ is nowhere dense in $\C$. Hence $\NO(T^2,x)$, being a subset of $\NO(T,x)$,  is nowhere dense in $\C$. If $|a|=|b|$, then $2\langle (T^{2})^nx,x\rangle =p^{2n}u^{2n}(1+z^{2n})$ for $n\in\Z_+$. Since each $\nu_k$ is odd, (\ref{linnuk}) guarantees that
$\liminf\limits_{n\to\infty}|\langle (T^{2})^nx,x\rangle|^{1/2n}=p\liminf\limits_{n\to\infty}|1+z^{2n}|^{1/2n}\geq p^{2/3}>1$. Hence $|\langle (T^{2})^nx,x\rangle|\to \infty$ and $\NO(T^2,x)$ is nowhere dense in $\C$. Thus $T^2\notin \NH(\C^2)$. The proof of Theorem~\ref{aq1} is complete.

\subsection{Proof of Theorem~\ref{aq2}}

Let $p$ and $q$ be two odd prime numbers such that $p>q^{3/2}$, $\nu_k=\nu_k(p)$ and $\nu'_k=\nu_k(q)$ be as defined in (\ref{nuk}). Let also ${\bf w}=\{w_k\}_{k\in\N}$ and ${\bf w}'=\{w'_k\}_{k\in\N}$ be two sequences of complex numbers such that $\frac1k\leq |w_k|,|w'_k|\leq k$ for every $k\in\N$, $\{w_k:k\in\N\}$ is a dense subset of the half-plane $\{t\in\C:{\rm Re}\,t>0\}$ and $\{w'_k:k\in\N\}$ is a dense subset of the half-plane $\{t\in\C:{\rm Re}\,t<0\}$. Let also $z=z(p,{\bf w})$, $z'=z(q,{\bf w}')$, $u=u(p,{\bf w})$ and $u'=u(q,{\bf w}')$ be as defined in (\ref{tauz}) and (\ref{theu}) and let $T$ be the diagonal operator on $\C^4$ with the numbers $pu$, $puz$, $qu'$ and $qu'z'$ on the diagonal.

First, observe that if $x,y\in S(\C^4)$ are given by $x=\bigl(\frac1{\sqrt 2},\frac1{\sqrt 2},0,0\bigr)$ and
$y=\bigl(0,0,\frac1{\sqrt 2},\frac1{\sqrt 2}\bigr)$, then
$$
2\langle T^nx,x\rangle =p^nu^n(1+z^n)\ \ \text{and}\ \ 2\langle T^ny,y\rangle =q^n(u')^n(1+(z')^n)\ \ \text{for every $n\in\Z_+$}.
$$
By Lemma~\ref{esti2}, $\langle T^{\nu_k}x,x\rangle-\frac{w_k}{2}\to 0$  and $\langle T^{\nu'_k}y,y\rangle-\frac{w'_k}{2}\to 0$. Since $\{\frac{w_k}{2}:k\in\N\}\cup \{\frac{w'_k}{2}:k\in\N\}$ is dense in $\C$, $\NO(T,x)\cup \NO(T,y)$ is dense in $\C$. Hence the union of 2 numerical orbits of $T$ corresponding to elements of $\Pi(\C^4)$ is dense in $\C$. The proof will be complete if we prove that $T$ is not numerically hypercyclic. In order to do that, let
$x=(a,b,c,d)$ be an arbitrary element of $S(\C^4)$. It remains to show that $\NO(T,x)$ is not dense in $\C$. A direct computation yields
\begin{equation}\label{hdno}
\langle T^nx,x\rangle=p^nu^n(|a|^2+|b|^2z^n)+q^n(u')^n(|c|^2+|d|^2(z')^n)\ \ \text{for every $n\in\Z_+$}.
\end{equation}

{\bf Case 1:} $|a|\neq |b|$. By (\ref{hdno}), $|\langle T^nx,x\rangle|\geq ||a|^2-|b|^2|p^n(1+O(q^n/p^n))\to \infty$ as $n\to\infty$
since $q<p$. Hence $\NO(T,x)$ is non-dense in $\C$.

{\bf Case 2:} $|a|=|b|\neq 0$ and $|c|+|d|\neq 0$. Since $p>q^{3/2}$, we can pick $r\in (q,p^{2/3})$. By (\ref{linnuk}), there is $c>0$ such that
$p^n|1+z^n|\geq cr^n$ for every $n\in\N\setminus\{\nu_k:k\in\N\}$. According to (\ref{hdno}),
we have
$$
\text{$|\langle T^nx,x\rangle|\geq |a|^2cr^n-2q^n\to \infty$ as $n\to\infty$, $n\notin\{\nu_k:k\in\N\}$.}
$$
By (\ref{linuk}), $p^{\nu_k}|1+z^{\nu_k}|-|w_k|\to 0$. Since $|w_k|\leq k$, $p^{\nu_k}|1+z^{\nu_k}|=O(k)$ as $k\to\infty$. Pick $s\in(1,q^{2/3})$. If $|c|=|d|\neq 0$, the obvious observation that $\nu'_k\neq \nu_m$ for every $k,m\in\N$ and (\ref{linnuk}) imply that there is $c_1>0$ such that $|q^{\nu_k}(u')^{\nu_k}(|c|^2+|d|^2(z')^{\nu_k})|=q^{\nu_k}|c|^2|1+|(z')^{\nu_k})|\geq c_1s^{\nu_k}$ for every $k\in\N$. If $|c|\neq |d|$, then $q^n(u')^n(|c|^2+|d|^2(z')^n)\geq ||c|^2-|d|^2|q^n$ and therefore
$|q^{\nu_k}(u')^{\nu_k}(|c|^2+|d|^2(z')^{\nu_k})|\geq ||c|^2-|d|^2|q^{\nu_k}\geq ||c|^2-|d|^2|s^{\nu_k}$. Thus in any case there is $c_2>0$ such that $|q^{\nu_k}(u')^{\nu_k}(|c|^2+|d|^2(z')^{\nu_k})|\geq c_2s^{\nu_k}$ for every $k\in\N$. Combining these estimates with (\ref{hdno}), we get 
$|\langle T^{\nu_k}x,x\rangle|\geq c_2s^{\nu_k}+O(k)\to \infty$ as $k\to\infty$. 
Then by the above display $|\langle T^nx,x\rangle|\to \infty$ as $n\to\infty$ and $\NO(T,x)$ is non-dense in $\C$.

{\bf Case 3:} $a=b=0$ and $|c|\neq |d|$. By (\ref{hdno}), $|\langle T^{n}x,x\rangle|\geq ||c|^2-|d|^2|q^n\to \infty$ as $n\to\infty$.
Hence $\NO(T,x)$ is non-dense in $\C$.

{\bf Case 4:} $a=b=0$ and $|c|=|d|$.
By (\ref{hdno}), $2\langle T^{n}x,x\rangle=q^n(u')^n(1+(z')^n)$ for $n\in\N$. Then by (\ref{linnuk}), the lower limit of $|\langle T^{n}x,x\rangle|^{1/n}$ as $n\to\infty$, $n\notin\{\nu'_k:k\in\N\}$ is at least $q^{2/3}>1$. Hence 
$|\langle T^{n}x,x\rangle|\to \infty$ as $n\to\infty$, $n\notin\{\nu'_k:k\in\N\}$.
On the other hand, according to Lemma~\ref{esti2}, $\lim\limits_{k\to\infty}\Bigl(\langle T^{\nu'_k}x,x\rangle-\frac{w'_k}{2}\Bigr)=0$ .
Since each $w'_k$  is in the left half-plane, $\NO(T,x)$ has no accumulation points in the open right half-plane. Thus $\NO(T,x)$ is non-dense in $\C$.

{\bf Case 5:} $|a|=|b|$ and $c=d=0$.
By (\ref{hdno}), $2\langle T^{n}x,x\rangle=p^nu^n(1+z^n)$ for every $n\in\N$. Then according to (\ref{linnuk}), the lower limit of $|\langle T^{n}x,x\rangle|^{1/n}$ as $n\to\infty$, $n\notin\{\nu_k:k\in\N\}$ is at least $p^{2/3}>1$. Hence $|\langle T^{n}x,x\rangle|\to \infty$ as $n\to\infty$, $n\notin\{\nu_k:k\in\N\}$.
On the other hand, by Lemma~\ref{esti2}, $\lim\limits_{k\to\infty} \Bigl(\langle T^{\nu_k}x,x\rangle-\frac{w_k}{2}\Bigr)=0$. Since
each $w_k$  is in the right half-plane, $\NO(T,x)$ has no accumulation points in the open left half-plane. Hence
$\NO(T,x)$ is non-dense in $\C$. The proof of Theorem~\ref{aq2} is complete.

\section{Dense sums of powers\label{s3}}

We need a number of technical results. First, we state the Universality Theorem, which will be repeatedly used in this and in the subsequent sections. A topological space $X$ is called {\it Baire} if the intersection of countably many dense open subsets of $X$ is always dense in $X$. Recall that a {\it Polish space} is a separable topological space $X$, whose topology can be defined by a metric $d$ such that the metric space $(X,d)$ is complete. By the Baire Theorem, every Polish space is Baire. The following result is a particular case of the Universality Theorem \cite[p.~348--349]{ge1}, see also \cite{bama,gp-book}.

\begin{thmu}
Let $X$ be a Baire topological space, $Y$ be a Polish space and ${\cal F}=\{f_a:a\in A\}$ be a collection of continuous
maps $f_a:X\to Y$ such that the set $\{(x,f_a(x)):x\in X,\ a\in A\}$ is dense in $X\times Y$. Then
$\bigl\{x\in X:\{f_a(x):a\in A\}\ \ \text{is dense
in $Y$}\bigr\}$ is a dense $G_\delta$-subset of $X$. In particular, there is $x\in X$ such that $\{f_a(x):a\in A\}$ is dense in $Y$.
\end{thmu}

\subsection{Finite sums}

\begin{lemma}\label{wnh21} Let $z,w\in\T$ be independent and
$\{R_n\}_{n\in\N}$ be a sequence of positive numbers such that $R_n\to\infty$. Then
$$
\bigl\{(a,b)\in\T^2:\{R_n(az^n+bw^n):n\in\Z_+\}\ \ \text{is dense in $\C$}\bigr\}\ \ \ \text{is a dense $G_\delta$-subset of $\T^2$.}
$$
\end{lemma}

\begin{proof} Obviously, $f_n(a,b)=R_n(az^n+bw^n)$ is a continuous map from $\T^2$ to $\C$. By Theorem~U, the proof will be complete if we verify that
$$
\Lambda=\{(a,b,R_n(az^n+bw^n)):(a,b)\in\T^2,\ n\in\Z_+\}\ \ \text{is a dense subset of $\T\times\T\times\C$.}
$$
Pick $\alpha,\beta\in\T$ and $v\in\C\setminus\{0\}$ and let $s=\frac{v}{|v|}$ and $t_n=\frac{|v|}{2R_n}$. Clearly, $t_n\to 0$ and $v=2R_nst_n$ for each $n\in\N$. Since $z$ and $w$ are independent, $\{(z^n,w^n):n\in\Z_+\}$ is dense in $\T^2$ and there is a strictly increasing sequence $\{n_k\}_{k\in\N}$ of positive integers such that $t_{n_k}<1$ for each $k$,
$z^{n_k}\to is\alpha^{-1}$ and $w^{n_k}\to -is\beta^{-1}$ as $k\to\infty$. Set $y_k=t_{n_k}+i\sqrt{1-t_{n_k}^2}\in \T$, $a_k=sz^{-n_k}y_k$ and $b_k=sw^{-n_k}y_k^{-1}$. Then
$R_{n_k}(a_kz^{n_k}+b_kw^{n_k})=2st_{n_k}R_{n_k}=v$ and therefore $(a_k,b_k,v)\in\Lambda$ for each $k\in\N$. Since $y_k\to i$, $z^{n_k}\to is\alpha^{-1}$ and $w^{n_k}\to -is\beta^{-1}$,
 we have $a_k\to \alpha$ and $b_k\to\beta$. Hence the sequence $\{(a_k,b_k,v)\}$ of elements of $\Lambda$ converges to $(\alpha,\beta,v)$. Since $\alpha,\beta\in\T$ and $v\in\C\setminus\{0\}$ are arbitrary, $\Lambda$ is dense in $\T\times\T\times\C$.
\end{proof}

\begin{lemma}\label{wnh31} Let $z,w\in\T$ be independent, $k\in\N$, $m\in\Z$,  and
$\{R_n\}_{n\in\N}$, $\{r_n\}_{n\in\N}$  be two sequences of positive numbers such that $R_n\to\infty$, $r_n\to\infty$ and $\frac{R_n}{r_n}\to\infty$. Then
$$
\bigl\{(a,b)\in\T^2:\{R_nz^{mn}(az^{kn}+a^{-1}z^{-kn})+br_nw^n:n\in\Z_+\}\ \ \text{is dense in $\C$}\bigr\}\ \ \text{is a dense $G_\delta$-subset of $\T^2$.}
$$
\end{lemma}

\begin{proof} Clearly, $f_n(a,b)=R_nz^{mn}(az^{kn}+a^{-1}z^{-kn})+br_nw^n$ is a continuous map from $\T^2$ to $\C$. By Theorem~U, the proof will be complete if we verify that
$$
\Lambda=\{(a,b,R_nz^{mn}(az^{kn}+a^{-1}z^{-kn})+br_nw^n):(a,b)\in\T^2,\ n\in\Z_+\}\ \ \text{is a dense subset of $\T\times\T\times\C$}.
$$
Pick $\alpha,\beta\in\T$ and $v\in\C$. Since $r_n\to\infty$, for each sufficiently large $n$, the circle $z^{-mn}v+r_n\T$ intersects the real line at 2 points with exactly one $s_n$ being in $(0,\infty)$. Let $u_n\in\T$ be the unique number such that $z^{-mn}v+r_nu_n=s_n$. Obviously, $u_n\to1$ and $\frac{s_n}{r_n}\to1$. Denote $b_n=-z^{mn}w^{-n}u_n$. Since $\frac{s_n}{r_n}\to 1$ and $\frac{r_n}{R_n}\to 0$, we have $\frac{s_n}{R_n}\to 0$. Thus $\frac{s_n}{R_n}\leq 2$ for all sufficiently large $n$. For such $n$, we can consider $q_n=\frac{s_n}{2R_n}+\sqrt{1-\frac{s_n^2}{4R_n^2}}\,i\in\T$. Clearly, $q_n\to i$ as $n\to\infty$. Let $a_n=q_nz^{-kn}$. It is easy to verify that
$$
\text{$R_n(a_nz^{kn}+a_n^{-1}z^{-kn})+r_nb_nw^n=v$ for each sufficiently large $n$.}
$$
Since $z$ and $w$ are independent, $\{(z^{kn},w^nz^{-mn}):n\in\Z_+\}$ is dense in $\T^2$. Thus there is a strictly increasing sequence $\{n_j\}_{j\in\N}$ of positive integers such that $z^{kn_j}\to i\alpha^{-1}$ and $w^{n_j}z^{-mn_j}\to -\beta^{-1}$ as $j\to\infty$.
Since $u_n\to 1$ and $q_n\to i$, we have $b_{n_j}=-w^{-n_j}z^{-mn_j}u_{n_j}\to \beta$ and $a_{n_j}=q_{n_j}z^{-kn_j}\to \alpha$. By the above display, $(a_{n_j},b_{n_j},v)\in\Lambda$ and therefore $(\alpha,\beta,v)$ is the limit of a sequence of elements of $\Lambda$. Since $\alpha,\beta\in\T$ and $v\in\C$ are arbitrary, $\Lambda$ is dense in $\T\times\T\times\C$.
\end{proof}

\begin{lemma}\label{3point} Let $u,w,z\in\T$ be independent and $\{R_n\}_{n\in\N}$ be a sequence of positive numbers such that $R_n\to\infty$. Then there exist $a,b,c\in\R_+$ such that $\{R_n(au^n+bw^n+cz^n):n\in\Z_+\}$ is dense in $\C$.
\end{lemma}

\begin{proof} Let $A=\{(a,b)\in\R^2:a>0,\ b>0,\ a+b<1\}$. Obviously, $A$ is a non-empty open subset of $\R^2$. Clearly, $f_n:A\to\C$, $f_n(a,b)=R_n(au^n+bw^n+cz^n)$ is continuous. By Theorem~U, the proof will be complete if we show that
$$
\Lambda=\{(a,b,R_n(au^n+bw^n+(1-a-b)z^n)):n\in\Z_+,\ (a,b)\in A\}\ \ \text{is dense in $A\times\C$.}
$$

Clearly, $M=\{(\alpha,\beta,\gamma)\in\T^3:\alpha\neq\beta,\ \alpha\neq\gamma,\ \beta\neq\gamma\}$ is a dense open subset of $\T^3$. Consider
$$
F:\R^2\times\T^3\to\C,\qquad F(x,\xi)=x_1\xi_1+x_2\xi_2+(1-x_1-x_2)\xi_3.
$$
Differentiating $F$, we easily confirm that
\begin{equation}\label{rankk}
\text{for every $(x,\xi)\in A\times M$, the differentials of both $F(x,\cdot)$ at $\xi$ and $F(\cdot,\xi)$ at $x$ have rank $2$.}
\end{equation}
Furthermore, for every $x\in A$, $F(x,M)$ is an open subset of $\C$ containing $0$.

Let $x=(a,b)\in A$ and $y\in\C$. Pick $\xi=(\alpha,\beta,\gamma)\in M$ such that $F(x,\xi)=0$. According to (\ref{rankk}), a standard implicit function argument yields that
\begin{equation}\label{rankk0}
\begin{array}{l}\text{for every sequence $\{\xi_j=(\alpha_j,\beta_j,\gamma_j)\}_{j\in\N}$ in $\T^3$ convergent to $\xi$ and for every}\\ \text{sequence $\{z_j\}_{j\in\N}$ in $\C$
convergent to $0$, there is a sequence $\{x_j=(a_j,b_j)\}_{j\in\N}$ in $\R^2$}\\ \text{convergent to $x$ such that $F(x_j,\xi_j)=z_j$ for each sufficiently large $j\in\N$.}
\end{array}
\end{equation}

Since $u,w,z$ are independent, $\{(u^n,w^n,z^n):n\in\Z_+\}$ is dense in $\T^3$. Thus there is a strictly increasing sequence $\{n_j\}_{j\in\N}$ of positive integers such that $u^{n_j}\to\alpha$, $w^{n_j}\to\beta$ and $z^{n_j}\to\gamma$. By (\ref{rankk0}), we can pick a sequence $\{x_j=(a_j,b_j)\}_{j\in\N}$ in $\R^2$ convergent to $x$ such that $F(x_j,\xi_j)=\frac{y}{R_{n_j}}$ for each sufficiently large $j$, where $\xi_j=(u^{n_j},w^{n_j},z^{n_j})$. Since $A$ is open, $x_j\in A$ for each sufficiently large $j$. By the definition of $F$, $R_{n_j}(a_ju^{n_j}+b_jw^{n_j}+(1-a_j-b_j)z^{n_j})=y$ and therefore $(a_j,b_j,y)\in \Lambda$ for each sufficiently large $j$. Since $(a_j,b_j,y)\to (a,b,y)$ and $x=(a,b)\in A$, $y\in\C$ are arbitrary, $\Lambda$ is dense in $A\times\C$.
\end{proof}

\begin{lemma}\label{fnh1} Let $z,w\in\T$ be independent, $\{R_n\}_{n\in\N}$ and $\{r_n\}_{n\in\N}$ be sequences of positive numbers such that $R_n\to\infty$ and $\frac{R_n}{r_n}\to\infty$, $U$ be a non-empty open subset of $\C$ and $\phi:U\to\C$ be continuous. Then
$$
\bigl\{a\in U:\{(R_na+r_n\phi(a))z^n+(R_n\overline{a}+r_n\overline{\phi(a)})w^n:n\in\N\}\ \text{is dense in $\C$}\bigr\}\ \ \text{is a dense $G_\delta$-subset of $U$.}
$$
\end{lemma}

\begin{proof} Clearly, it is enough to prove the statement with $U$ replaced by an arbitrary non-empty open set whose closure is contained in $U$. Thus, without loss of generality, we can assume that $\phi$ is bounded: $c=\sup\{|\phi(u)|:u\in U\}<\infty$. Since $z,w$ are independent, we can pick independent $s,u\in\T$ such that $z=us$ and $w=us^{-1}$. Then
$$
(R_na+r_n\phi(a))z^n+(R_n\overline{a}+r_n\overline{\phi(a)})w^n=u^nR_n\bigl((a+\rho_n\phi(a))s^n+(\overline{a}+\rho_n\overline{\phi(a)})s^{-n}\bigr)
$$
for each $n\in\N$, where $\rho_n=\frac{r_n}{R_n}$.
By Theorem~U, it suffices to verify that
$$
\Lambda=\bigl\{\bigl(a,u^nR_n\bigl((a+\rho_n\phi(a))s^n+(\overline{a}+\rho_n\overline{\phi(a)})s^{-n}\bigr)\bigr):a\in U,\ n\in\N\bigr\}\ \ \text{is dense in $U\times\C$.}
$$

Fix $a\in U\setminus\{0\}$ and $y\in\C\setminus\{0\}$. Since $s,u$ are independent, $\{(s^n,u^n):n\in\N\}$ is dense in $\T^2$. Hence, we can find a strictly increasing sequence $\{n_k\}_{k\in\N}$ of positive integers such that $s^{n_k}\to \frac{i|a|}{a}$ and $u^{n_k}\to \frac{y}{|y|}$ as $k\to\infty$. Since $i|a|s^{-n_k}\to a\in U$ and $U$ is open, there is $\epsilon>0$ such that for each sufficiently large $k$, $is^{-n_k}|a|e^{ix}\in U$ for every $x\in[-\epsilon,\epsilon]$. Thus for each sufficiently large $k$, we can define the map
$$
\textstyle\xi_k:[-\epsilon,\epsilon]\to \R, \quad \xi_k(x)=-\frac{|y|}{R_{n_k}}+2\,{\rm Re}\,\bigl(i|a|e^{ix}+\rho_{n_k}\phi(is^{-n_k}|a|e^{ix})s^{n_k}\bigr).
$$
For every $\delta\in(0,\epsilon)$, we have
\begin{equation}\label{exi}
\begin{array}{l}
\min\limits_{x\in[-\epsilon,-\delta]}\xi_k(x)\geq 2|a|\sin \delta-\frac{|y|}{R_{n_k}}-c\rho_{n_k}\to 2|a|\sin\delta>0\ \ \text{as $k\to\infty$,}\\
\max\limits_{x\in[\delta,\epsilon]}\xi_k(x)\leq -2|a|\sin \delta-\frac{|y|}{R_{n_k}}+c\rho_{n_k}\to -2|a|\sin\delta<0\ \ \text{as $k\to\infty$.}.
\end{array}
\end{equation}
By (\ref{exi}), $\xi_k$ takes both positive and negative values for each sufficiently large $k$.
Since each $\xi_k$ is continuous, the Intermediate Value Theorem guarantees the existence of $x_k\in [-\epsilon,\epsilon]$ such that $\xi_k(x_k)=0$. Furthermore, the estimates (\ref{exi}) imply that $x_k\to 0$ as $k\to\infty$. Now we set $a_k=i|a|e^{ix_k}s^{-n_k}\in U$ and $y_k=u^{n_k}R_{n_k}\bigl((a_k+\rho_{n_k}\phi(a_k))s^{n_k}+(\overline{a_k}+\rho_{n_k}\overline{\phi(a_k)})s^{-n_k}\bigr)\in\C$. Obviously $(a_k,y_k)\in\Lambda$. Since $s^{n_k}\to \frac{i|a|}{a}$ and $x_k\to 0$, we have $a_k\to a$. Next, by definition of $a_k$,
$$
y_{k}=2R_{n_k}u^{n_k}\,{\rm Re}\,(a_ks^{n_k}+\rho_{n_k}\phi(a_k)s^{n_k})=2R_{n_k}u^{n_k}\,{\rm Re}\,(i|a|e^{ix_k}+\rho_{n_k}\phi(i|a|e^{ix_k}s^{-n_k})s^{n_k}).
$$
Using the above display together with the equality $\xi_k(x_k)=0$ and the definition of $\xi_k$, we arrive to $y_{k}=u^{n_k}|y|$. Since $u^{n_k}\to \frac{y}{|y|}$, we have $y_k\to y$. Thus $\{(a_k,y_k)\}$ is a sequence in $\Lambda$ convergent to $(a,y)$. Since $a\in U\setminus\{0\}$ and $y\in\C\setminus\{0\}$ are arbitrary, $\Lambda$ is dense in $U\times\C$.
\end{proof}

\begin{lemma}\label{123} Let $1<r<R$ and $k,m$ be distinct integers. Then 
$$
\bigl\{(z,w)\in\T^2:\{R^n(z^{kn}+z^{mn})+r^nw^n:n\in\Z_+\}\ \text{is dense in $\C$}\bigr\}\ \ \text{is a dense $G_\delta$ subset of $\T^2$.}
$$
\end{lemma}

\begin{proof} By Theorem~U, it suffices to verify that
$$
\Lambda=\{(z,w,R^n(z^{kn}+z^{mn})+r^nw^n):z,w\in\T,\ n\in\Z_+\}\ \ \text{is dense in $\T\times\T\times\C$}.
$$
Let $a,b\in\T$ and $y\in\C$. Since $k\neq m$, for every $n\in\N$, there are $s_n,t_n\in\T$ such that $s_n^k+s_n^m=0$ and $|t_n^k+t_n^m|=2$. Consider the function $G_n:\T\to\C$, $G_n(z)=R^nr^{-n}(z^k+z^m)-r^{-n}y$. Then $|G_n(s_n)|=r^{-n}|y|<1$ for each sufficiently large $n$ and $|G_n(t_n)|\geq 2\frac{R^n}{r^n}-\frac{|y|}{r^n}>1$ for each sufficiently large $n$. Thus for $n$ large enough the function $|G_n|$ must attain the value $1$. That is, there is $c_n\in\T$ such that $G_n(c_n)=d_n\in\T$. Hence $R^n(c_n^k+c_n^m)+r^n d_n=y$ for each sufficiently large $n$. Now pick $z_n\in\T$ and $w_n\in\T$ such that $z_n^n=c_n$, $w_n^n=d_n$, $|z_n-a|<\frac{\pi}{n}$ and $|w_n-b|<\frac{\pi}{n}$. Since $R^n(z_n^{nk}+z_n^{nm})+r^n w_n^n=y$, we have $(z_n,w_n,y)\in\Lambda$. Since $z_n\to a$, $w_n\to b$ and $a,b\in\T$, $y\in\C$ are arbitrary, $\Lambda$ is dense in $\T\times\T\times\C$.
\end{proof}

\subsection{Proof of Proposition~\ref{znwn}}

Let $U_R=\bigl\{(z,w)\in\T^2:\{R^n(z^n+w^n):n\in\Z_+\}\ \text{is dense in $\C$}\bigr\}$.
We have to show that $U_R$ is a dense $G_\delta$-subset of $\T^2$ of Lebesgue measure $0$. By Theorem~U,
in order to show that $U_R$ is a dense $G_\delta$-subset of $\T^2$, it suffices to verify that
$$
\Lambda=\{(z,w,R^n(z^n+w^n)):z,w\in\T,\ n\in\Z_+\}\ \ \text{is dense in $\T\times\T\times\C$.}
$$
Let $y\in\C$ and $z,w\in\T$. Since $\T+\T\supset2\D$, for every $n$ satisfying $2R^n>|y|$, we can find $a_n,b_n\in\T$ such that $R^n(a_n+b_n)=y$. For each such $n$ we can pick $z_n,w_n\in\T$ for which $z_n^n=a_n$, $w_n^n=b_n$, $|z_n-z|<\frac{\pi}{n}$ and $|w_n-w|<\frac{\pi}{n}$. Since $y=R^n(a_n+b_n)=R^n(z_n^n+w_n^n)$, we have $(z_n,w_n,y)\in\Lambda$ for every sufficiently large $n$. Since $(z_n,w_n,y)\to (z,w,y)$ and $y\in\C$, $z,w\in\T$ are arbitrary, $\Lambda$ is dense in $\T\times\T\times\C$. By Theorem~U, $U_R$ is a dense $G_\delta$-subset of $\T^2$.

Now let $\mu$ be the normalized Lebesgue measure on $\T^2$. A straightforward computation yields
\begin{equation}\label{lebes}\textstyle
\mu\{(z,w)\in\T^2:|z^n+w^n|\leq c\}=\frac2\pi \arcsin\bigl(\frac c2\bigr)\ \ \text{for every $c\in[0,2]$ and $n\in\N$.}
\end{equation}
Pick $r\in(1,R)$. By (\ref{lebes}), $\sum\limits_{n=1}^\infty \mu\bigl\{(z,w)\in\T^2:|z^n+w^n|\leq \frac{r^n}{R^n}\bigr\}<\infty$. Hence for almost all $(z,w)\in\T^2$, $|z^n+w^n|>\frac{r^n}{R^n}$ for all sufficiently large $n$ and therefore $R^n|z^n+w^n|\to\infty$. Since the condition $R^n|z^n+w^n|\to\infty$ is incompatible with the membership in $U_R$, $\mu(U_R)=0$, as required.

\subsection{Infinite sums and integrals}

Recall that $A\subseteq \Z$ is called {\it syndetic} if there is a finite subset $B$ of $\Z$ such that $A+B=\Z$. Equivalently, $A$ is syndetic if $A=\{n_k:k\in\Z\}$, where $n_{k+1}>n_k$ for each $k\in\Z$ and $\{n_{k+1}-n_k:k\in\Z\}$ is bounded. 

\begin{lemma}\label{grou} Let $G$ be a compact topological group, $U$ be an open subset of $G$ and $g\in G$. Then the set $A=\{n\in\Z:g^n\in U\}$ is either empty or syndetic. In particular, if $U$ contains the identity element, then $A$ is always syndetic.
\end{lemma}

\begin{proof} Since $G$ is a compact topological group, the subgroup $H=\{g^m:m\in\Z\}$ is pre-compact with respect to the induced topology. If $V=U\cap H$ is empty, then $A=\varnothing$. Otherwise, $V$ is a non-empty open subset of $H$. Hence we can pick $m\in\Z$ such that $g^mV$ is a neighborhood of the identity in $H$. Since $H$ is pre-compact, there is a finite set $C=\{g^{n_1},\dots,g^{n_k}\}\subseteq H$ such that $Cg^mV=H$. It follows that $B+A=\Z$, where $B=\{n_1+m,\dots,n_k+m\}$. Hence $A$ is syndetic.
\end{proof}

\begin{lemma}\label{toR} Let $M$ be a subset of $\T$ such that $1$ is an accumulation point of $M$, $A$ be a syndetic subset of $\Z$ and $k\in\N$. Then
$\{(z_1^n,\dots,z_k^n):n\in A\cap\N,\ z_1,\dots,z_k\in M\}$ is dense in $\T^k$.
\end{lemma}

\begin{proof} Let $u=(u_1,\dots,u_k)$ be an arbitrary point in $\T^k$. Then we can pick $\theta_j\in\R$ such that $u_j=e^{\theta_ji}$ for $1\leq j\leq k$. Fix temporarily $\delta\in(0,1/2)$. Since $A$ is syndetic, there is $q\in\N$ such that $A+\{1,\dots,q\}=\Z$. Since $1$ is an accumulation point of $M$, we can pick $b_1(\delta),\dots,b_k(\delta)\in\R$ such that $0<|b_k(\delta)|<\delta$, $0<|b_{j-1}(\delta)|<\delta |b_j(\delta)|$ for $2\leq j\leq k$ and $e^{b_j(\delta)i}\in M$ for $1\leq j\leq k$.

First, we choose the smallest positive integer $m_1$ such that $\dist(m_1b_1(\delta)-\theta_1,2\pi\Z)<|b_1(\delta)|$. Then we pick the smallest integer $m_2$ greater than $m_1$ such that $\dist(m_2b_2(\delta)-\theta_2,2\pi\Z)<|b_2(\delta)|$. We proceed the same way. Namely, having chosen $m_1,\dots,m_{j-1}$, we set $m_j$ to be the smallest integer greater than $m_{j-1}$ such that $\dist(m_jb_j(\delta)-\theta_j,2\pi\Z)<|b_j(\delta)|$. By definition of $q$, we can pick $n\in A$ such that $m_k\leq n\leq m_k+q-1$.
Clearly, $m_j-m_{j-1}\leq \frac{2\pi}{|b_j(\delta)|}$ for $2\leq j\leq k$. Hence
$$\textstyle
0\leq n-m_j\leq q-1+\sum\limits_{m=j+1}^k \frac{2\pi}{|b_k(\delta)|}\leq q-1+\frac{2\pi}{|b_j(\delta)|}\sum\limits_{m=1}^{k-j-2} \delta^m\leq q-1+\frac{2\pi\delta}{|b_j(\delta)|(1-\delta)} \ \
\text{for $1\leq j\leq k$,}
$$
where the third inequality is derived from $|b_{m-1}(\delta)|<\delta |b_m(\delta)|$. Hence
$$\textstyle
|nb_j(\delta)-m_jb_j(\delta)|\leq (q-1)|b_j(\delta)|+\frac{2\pi\delta}{1-\delta}\leq \delta(q-1+4\pi)\ \ \text{for $1\leq j\leq k$}
$$
since $|b_j(\delta)|<\delta$. Since $\dist(m_jb_j(\delta)-\theta_j,2\pi\Z)<|b_j(\delta)|$, we have $|u_j-e^{m_jb_j(\delta)i}|<|b_j(\delta)|<\delta$. According to the above display, $|e^{nb_j(\delta)i}-e^{m_jb_j(\delta)i}|<\delta(q-1+4\pi)$. Hence $|u_j-z_j^n|<\delta(q+4\pi)$, where $z_j=e^{b_j(\delta)i}$.
Since $z_j\in M$, $n\in A\cap \N$ and $\delta\in(0,1/2)$ is arbitrary, the desired result follows.
\end{proof}


\begin{corollary}\label{sico} Let $M$ be an infinite subset of $\T$, $A$ be a syndetic subset of $\Z$, $k,n_0\in\N$ and $u=(u_1,\dots,u_k)\in \T^k$. Then for every $\epsilon>0$, there exist $n,m\in A$ and $z,z_1,\dots,z_k\in M$ such that $m>n>n_0$, $|z^{-n}z_j^n-u_j|<\epsilon$ and $|z^{-m}z_j^m-u_j^2|<\epsilon$ for $1\leq j\leq k$.
\end{corollary}

\begin{proof} Since $A$ is syndetic, there is $q\in\N$ such that $A+\{1,\dots,q\}=\Z$. Since $\T$ is compact and $M$ is infinite, there is an accumulation point $v\in\T$ of $M$. Then for each $\delta\in(0,1)$, $1$ is an accumulation point of $M_\delta=\{v^{-1}z:z\in M,\ |z-v|<\delta\}$. By Lemma~\ref{toR}, there are $n\in A$ and $w_1,\dots,w_k\in M_\delta$ such that $n>n_0$ and $|w_j^n-u_j|<\delta$ for $1\leq j\leq k$. Since $w_j\in M_\delta$, $w_j=\frac{z_j}{v}$ with $z_j\in M$ and $|z_j-v|<\delta$. Since $v$ is an accumulation point of $M$, we can find $z\in M$ such that $|z-v|<\frac{\delta}{2n}$. Since $|w_j^n-u_j|<\delta$, $w_j=\frac{z_j}{v}$ and $|z-v|<\frac{\delta}{2n}$, we have $|z^{-n}z_j^n-u_j|<2\delta$ for $1\leq j\leq k$. It follows that $|z^{-2n}z_j^{2n}-u^2_j|<4\delta$ for $1\leq j\leq k$. Now pick $m\in A\cap\{2n,2n+1,\dots,2n+q-1\}$. Since $|z-z_j|\leq |z-v|+|v-z_j|<2\delta$, we have
$|z^{-2n}z_j^{2n}-z^{-m}z_j^m|=|z_j^{m-2n}-z^{m-2n}|\leq (m-2n)|z_j-z|<2q\delta$. Hence
$$
|z^{-m}z_j^m-u_j^2|<|z^{-2n}z_j^{2n}-u_j^2|+|z^{-2n}z_j^{2n}-z^{-m}z_j^m|<4\delta+2q\delta=(2q+4)\delta.
$$
Thus for $\delta$ satisfying $(2q+4)\delta<\epsilon$, $|z^{-n}z_j^n-u_j|<\epsilon$ and $|z^{-m}z_j^m-u_j^2|<\epsilon$ for $1\leq j\leq k$, as required.
\end{proof}

The following lemma looks painfully technical. We need it nonetheless. For each $\epsilon>0$, let
\begin{equation}\label{pe}
P_\epsilon=\{u\in\T^{10}:\text{$|u_5^{-1}u_j-e^{2\pi ji/5}|<\epsilon$ and $|u_{10}^{-1}u_{5+j}-e^{4\pi ji/5}|<\epsilon$
\ \ for $1\leq j\leq 4$}\}.
\end{equation}

\begin{lemma}\label{penta}
There exist $\epsilon,d>0$ such that for every $u\in P_\epsilon$ and every $z,w\in d\,\D$, there are positive numbers  $a_1,\dots,a_5$ satisfying $a_1+{\dots}+a_5=1$, $a_1u_1+{\dots}+a_5u_5=z$ and $a_1u_6+{\dots}+a_5u_{10}=w$.
\end{lemma}

\begin{proof} For every $u\in\C^{10}$, consider the $\R$-linear map $S_u:\R^4\to\C^2$ given by the formula
$$\textstyle
S_u a=\Bigl(\sum\limits_{j=1}^4 (u_j-u_5)a_j,\sum\limits_{j=1}^4 (u_{5+j}-u_{10})a_j\Bigr).
$$
Let $\xi=e^{2\pi i/5}$ and for $\alpha,\beta\in\T$, let $u(\alpha,\beta)=
(\alpha\xi,\alpha\xi^2,\alpha\xi^3,\alpha\xi^4,\alpha,\beta\xi^2,\beta\xi^4,\beta\xi,\beta\xi^3,\beta)\in\T^{10}$.
Then
\begin{align*}
&S_{u(\alpha,\beta)}=ST_{\alpha,\beta}, \ \ \text{where}\ \ T_{\alpha,\beta}:\C^2\to\C^2,\ T_\alpha(z,w)=(\alpha z,\beta w)\ \ \text{and}
\\
&
S:\R^4\to\C^2,\ \ Sa= \left(\begin{array}{c}(\xi-1)a_1+(\xi^2-1)a_2+(\xi^3-1)a_3+(\xi^4-1)a_4\\ (\xi^2-1)a_1+(\xi^4-1)a_2+(\xi-1)a_3+(\xi^3-1)a_4\end{array}\right).
\end{align*}
First, observe that $T_\alpha$ is an isometry. Identifying $\C^2$ with $\R^4$ in a natural way, we see that $S$ is the linear operator on $\R^4$ with the matrix
$$
S=\left(\begin{array}{cccc} \cos(2\pi/5)-1&\cos(4\pi/5)-1&\cos(6\pi/5)-1&\cos(8\pi/5)-1\\
\sin(2\pi/5)&\sin(4\pi/5)&\sin(6\pi/5)&\sin(8\pi/5)\\
\cos(4\pi/5)-1&\cos(8\pi/5)-1&\cos(2\pi/5)-1&\cos(6\pi/5)-1\\
\sin(4\pi/5)&\sin(8\pi/5)&\sin(2\pi/5)&\sin(6\pi/5)
\end{array}
\right).
$$
One can verify (with some effort by hand and on the spot using an appropriate software) that the matrix in the above display is invertible. Thus $S_{u(\alpha,\beta)}$ is invertible for each $\alpha,\beta\in\T$. Since the set of invertible $\R$-linear operators from $\R^4$ to $\C^2$ is open, there is $\epsilon_0>0$ such that $S_u$ is invertible for every $u\in P_{\epsilon_0}$. Thus for each $u\in P_{\epsilon_0}$ and every $(z,w)\in\C^2$, there exists a unique $a(u,z,w)\in\R^4$ such that $S_ua(u,z,w)=(z-u_5,w-u_{10})$. Obviously, the map $(u,z,w)\mapsto a(u,z,w)$ is continuous on $P_{\epsilon_0}\times \C^2$. It is also straightforward to verify that $S_{u(\alpha,\beta)}\bigl(\frac15,\frac15,\frac15,\frac15\bigr)=(-\alpha,-\beta)$. Hence $a(u(\alpha,\beta),0,0)=\bigl(\frac15,\frac15,\frac15,\frac15\bigr)$. By continuity, there exist $\epsilon\in(0,\epsilon_0)$ and $d>0$ such that whenever $u\in P_\epsilon$ and $z,w\in d\,\D$, the components of $a(u,z,w)$ are positive and sum up to a number less than $1$. Now for $z,w\in d\D$ and $u\in P_\epsilon$, we set $a_j=a_j(u,z,w)$ for $1\leq j\leq 4$ and $a_5=1-\sum\limits_{j=1}^4 a_j(u,z,w)$. It is straightforward to see that the equality $S_ua(u,z,w)=(z-u_5,w-u_{10})$ is equivalent to $a_1u_1+{\dots}+a_5u_5=z$ and $a_1u_6+{\dots}+a_5u_{10}=w$, which completes the proof.
\end{proof}

\begin{lemma}\label{ddiaa} Let $M$ be an infinite countable subset of $\T$, $\omega:\R_+\to\R_+$ be a strictly increasing continuous concave function vanishing at $0$ and $\{R_n\}_{n\in\N}$ be a sequence of positive numbers such that $R_n\to\infty$. Then there exists $a\in\ell^1_+(M)$ such that
$$
\sum_{z\in M}\omega(a_z)<\infty\ \ \text{and}\ \ \Bigl\{R_k\sum_{z\in M} a_zz^k:k\in\Z_+\Bigr\}\ \ \text{is dense in $\C$.}
$$
\end{lemma}

\begin{proof} First, note that $Q(M)=\Bigl\{a\in\ell^1_+(M):\sum\limits_{z\in M}\omega(a_z)<\infty\Bigr\}$ equipped with the metric $\rho(a,b)=\sum\limits_{z\in M}\omega(|a_z-b_z|)$ is a separable complete metric space. Let
$$
\Omega=\biggl\{a\in Q(M):\liminf_{k\to\infty}\Bigl|\sum_{z\in M} a_zz^k\Bigr|=0\biggr\}.
$$
Then $\Omega$ is a non-empty ($0\in\Omega$) $G_\delta$-subset of the Polish space $Q(M)$ (it is by no means dense!). Indeed,
$$
\Omega=\bigcap_{m,n\in\N}\bigcup_{k=m}^\infty \biggl\{a\in Q(M): \Bigl|\sum_{z\in M} a_zz^k\Bigr|<\frac1n\biggr\}.
$$

Let also
$$
\Omega_0=\Bigl\{a\in Q(M):\sum_{z\in M} a_zz^k=0\ \ \text{for some $k\in\N$}\Bigr\}.
$$
First, we observe that $\Omega_0$ is a subset of $\Omega$. Indeed, let $a\in\Omega_0$. Then there is $k\in\N$ such that $\sum\limits_{z\in M} a_zz^k=0$. Consider the compact metrizable topological group $G=\T^M$ and $g\in G$ defined by $g(z)=z$ for $z\in M$. By Lemma~\ref{grou}, for every neighborhood $U$ of $1_G$  in $G$ there are infinitely many $n\in\N$ such that $g^n\in U$. Hence there is a strictly increasing sequence $\{n_j\}_{j\in\N}$ of positive integers such that $g^{n_j}\to 1_G$. Then
$$
\lim_{j\to\infty}\sum_{z\in M} a_zz^{k+n_j}=\sum_{z\in M} a_zz^{k}=0\ \ \text{and therefore $a\in \Omega$}.
$$

Let $d,\epsilon>0$ be the numbers provided by Lemma~\ref{penta}. First, we shall verify the following claim:
\begin{equation}\label{claiM}
\begin{array}{l}
\text{For each $a\in\Omega$, $q\in\N$, $\delta>0$ and $y\in\delta d\,\D$, there exist $n\in\N$ and $b\in Q(M)$}\\
\text{such that $n>q$, $\rho(b,0)<5\omega(\delta)$, $a+b\in\Omega_0$ and $\sum\limits_{z\in M}(a_z+b_z)z^n=y$.}
\end{array}
\end{equation}

Let $a\in\Omega$, $q\in\N$, $\delta>0$ and $y\in\delta d\D$. Clearly, $\theta=\delta d-|y|>0$. Denote
$A=\biggl\{n\in\Z:\Bigl|\sum\limits_{z\in M}a_zz^n\Bigr|<\theta\biggr\}$. 
Since $a\in\Omega$, the set $A$ is infinite and therefore non-empty. Clearly, $A=\{n\in\Z:g^n\in U\}$, where $U=\Bigl\{f\in G:\Bigl|\sum\limits_{z\in M}a_zf(z)\Bigr|<\theta\Bigr\}$. Since $U$ is an open subset of the compact topological group $G$, Lemma~\ref{grou} implies that $A$ is syndetic. Let $\xi=e^{2\pi i/5}$. By Corollary~\ref{sico}, there exist $m,n\in A$ and $z,z_1,\dots,z_5\in M$ such that $m>n>q$ and $|z_j^n-z^n\xi^j|<\epsilon/2$, $|z_j^m-z^m\xi^{2j}|<\epsilon/2$ for $1\leq j\leq 5$. Then $(z_1^n,\dots,z_5^n,z_1^m,\dots,z_5^m)$ belongs to the set $P_\epsilon$ defined in (\ref{pe}). Since $n,m\in A$, we have
\begin{align*}
&\textstyle
\Bigl|\frac{y}{\delta}-\frac1{\delta}\sum\limits_{z\in M}a_zz^n\Bigr|\leq \frac{|y|}{\delta}+\frac1\delta\Bigl|\sum\limits_{z\in M}a_zz^n\Bigr|<\frac{|y|}{\delta}+\frac1\delta(\delta d-|y|)=d,
\\
&\textstyle
\Bigl|-\frac1{\delta}\sum\limits_{z\in M}a_zz^m\Bigr|=\frac1\delta\Bigl|\sum\limits_{z\in M}a_zz^m\Bigr|<\frac1\delta(\delta d-|y|)<d.
\end{align*}
By Lemma~\ref{penta}, there exist positive numbers $\beta_1,\dots,\beta_5$ such that
\begin{equation}\label{match}
\text{$\sum\limits_{j=1}^5\beta_j=1$, \ $\sum\limits_{j=1}^5\beta_jz_j^n=\frac{y}{\delta}-\frac1{\delta}\sum\limits_{z\in M}a_zz^n$ \ and \
$\sum\limits_{j=1}^5\beta_jz_j^m=-\frac1{\delta}\sum\limits_{z\in M}a_zz^m$.}
\end{equation}
Define $b\in Q(M)$ by setting $b_z=0$ if $z\in M\setminus\{z_1,\dots,z_5\}$ and $b_{z_j}=\delta\beta_j$ for $1\leq j\leq 5$. The equation (\ref{match}) now reads
$\sum\limits_{z\in M}(a_z+b_z)z^n=y$ and $\sum\limits_{z\in M}(a_z+b_z)z^m=0$. The latter implies $a+b\in\Omega_0$. Since the non-zero components of $b$ sum up to $\delta$ and there are only 5 of them $\rho(b,0)<5\omega(\delta)$. This concludes the proof of Claim~(\ref{claiM}).

Since $\Omega$ is a non-empty $G_\delta$-subset of a Polish space $Q(M)$, $\Omega$ is a Polish space in its own right. Since for every $k\in\N$, the map $f_k:\Omega\to\C$, $f_k(a)=R^k\sum\limits_{z\in M} a_zz^k$ is continuous, Theorem~U guarantees that the proof will be complete if we show that
$$
\Lambda=\biggl\{\Bigl(a,R^k\sum\limits_{z\in M} a_zz^k\Bigr):a\in\Omega,\ k\in\N\biggr\}\ \ \text{is dense in $\Omega\times\C$.}
$$
Since $R_n\to\infty$, (\ref{claiM}) yields that for every $\delta>0$, $y\in\C$ and $a\in\Omega$, there are $n\in\N$ and $b\in\ell^1_+(M)$ such that $\rho(b,0)<\delta$ (and therefore $\rho(a,a+b)<\delta$), $a+b\in\Omega$ and $\sum\limits_{z\in M}(a_z+b_z)z^n=\frac{y}{R^n}$. Hence $(a+b,y)\in\Lambda$. Since $a\in \Omega$, $y\in\C$  and $\delta>0$ are arbitrary, $\Lambda$ is dense in $\Omega\times\C$, which completes the proof.
\end{proof}

\begin{lemma}\label{ddIaa00} Let $(\Omega,{\cal F},\mu)$ be a measure space. Assume that the essential range of $g\in L^\infty(\mu)$ contains a sequence $\{\lambda_n\}_{n\in\N}$ such that $|\lambda_1|>1$, $\{|\lambda_n|\}_{n\in\N}$ is strictly increasing and $\frac{\lambda_n}{|\lambda_n|}$ are pairwise distinct. Then there exist a non-negative real valued function $a\in L^1(\mu)$ for which
$$
\biggl\{\int_\Omega g^k a\,d\mu:k\in\Z_+\biggr\}\ \ \text{is dense in $\C$.}
$$
\end{lemma}

Applying the above lemma to the situation of a purely atomic measure space, we immediately obtain the following corollary.

\begin{corollary}\label{ddIaa} Let $\{\lambda_n\}_{n\in\N}$ be a bounded sequence in $\C$ such that $|\lambda_1|>1$, the sequence $\{|\lambda_n|\}_{n\in\N}$ is strictly increasing and the numbers $\frac{\lambda_n}{|\lambda_n|}$ are pairwise distinct. Then there exists $a\in\ell^1_+(\N)$ such that
$$
\Bigl\{\sum_{n=1}^\infty a_n \lambda_n^k :k\in\Z_+\Bigr\}\ \ \text{is dense in $\C$.}
$$
\end{corollary}

\begin{proof}[Proof of Lemma~$\ref{ddIaa00}$] Let $r_n=|\lambda_n|$, $z_n=\frac{\lambda_n}{|\lambda_n|}$ and $r=\lim r_n$. Let $\Omega_0=\{\omega\in\Omega:|g(\omega)|<r\}$ and $Z$ be the set of all real-valued non-negative $f\in L^1(\mu)$ vanishing outside $\Omega_0$. Clearly, $Z$ is a closed subset of the Banach space $L^1(\mu)$. Since for every $k\in\N$, the map $F_k:Z\to\C$, $F_k(f)=\int_\Omega g^kf\,d\mu$ is continuous, Theorem~U guarantees that the proof will be complete if we show that
$$
\Lambda=\biggl\{\Bigl(a,\int g^kf\,d\mu\Bigr):a\in Z,\ k\in\N\biggr\}\ \ \text{is dense in $Z\times\C$.}
$$

Let $U$ be an arbitrary non-empty open subset of $Z\times\C$. Since $r_n\to r$, we can find $(a,y)\in U$ and $m\geq 2$ such that $a$ vanishes outside $\{\omega\in\Omega:|g(x)|<r_{m-1}\}$.
Clearly $M=\{z_j:j>m\}$ is an infinite subset of $\T$. Since $|g|\leq r_{m-1}$ on the support of $a$, $\Bigl|\int ag^k\,d\mu\Bigr|\leq \|a\|_1r_{m-1}^k$. Since $r_{m-1}<r_m$, we can find
$q\in\N$ such that $\frac{r_{m-1}^q\|a\|_1}{r_m^q}+\frac{|y|}{r_m^q}<\frac13$ and therefore
$$
\frac{y}{r_m^{k}}-\frac1{r_m^k}\int_\Omega ag^k \,d\mu \in \frac13\,\D\ \ \text{for each $k\geq q$}.
$$
An elementary geometric argument ensures that there is $\epsilon>0$ (one can even find the biggest $\epsilon$) such that whenever $u,v,w\in\T$ satisfy $\bigl|\frac{u}{v}-e^{2\pi i/3}\bigr|<\epsilon$ and $\bigl|\frac{w}{v}-e^{4\pi i/3}|<\epsilon$, the disk $\frac13\D$ is contained in the convex hull of $u$, $v$ and $w$. Corollary~\ref{sico} applied with $A=\Z$ guarantees that for every integer $s\geq q$ there are $u_s,v_s,w_s\in M$ and $k_s\in\N$ such that $k_s>s$, $\bigl|\frac{u_s^{k_s}}{v_s^{k_s}}-e^{2\pi i/3}\bigr|<\epsilon$ and $\bigl|\frac{u_s^{k_s}}{v_s^{k_s}}-e^{4\pi i/3}\bigr|<\epsilon$. Thus $\frac13\D$ is contained in the convex hull of $u_{s}^{k_s}$, $v_s^{k_s}$ and $w_s^{k_s}$. Since $u_s,v_s,w_s\in M$, there are $j_{1,s},j_{2,s},j_{3,s}\in\{m+1,m+2,\dots\}$ such that $u_s=z_{j_{1,s}}$, $v_s=z_{j_{2,s}}$ and $w_s=z_{j_{3,s}}$. Since $\frac13\D$ is contained in the convex hull of $u_{s}^{k_s}$, $v_s^{k_s}$ and $w_s^{k_s}$, by the above display, there  exist $\alpha_1,\alpha_2,\alpha_3\in\R_+$ such that
$$
\text{$\alpha_1+\alpha_2+\alpha_3=1$ \ and}\ \
\alpha_1z_{j_{1,s}}^{k_s}+\alpha_2z_{j_{2,s}}^{k_s}+\alpha_3z_{j_{3,s}}^{k_s}=\frac{y}{r_m^{k_s}}-\frac1{r_m^{k_s}}\int_\Omega ag^{k_s} \,d\mu.
$$
Using the definitions of $r_n$ and $z_n$ and multiplying by $r_m^{k_s}$, we get
\begin{equation}\label{eqy}
\frac{\alpha_1r_m^{k_s}}{r_{j_{1,s}}^{k_s}}\lambda_{j_{1,s}}^{k_s}+
\frac{\alpha_2r_m^{k_s}}{r_{j_{2,s}}^{k_s}}\lambda_{j_{2,s}}^{k_s}+
\frac{\alpha_3r_m^{k_s}}{r_{j_{3,s}}^{k_s}}\lambda_{j_{3,s}}^{k_s}+
\int_\Omega ag^{k_s} \,d\mu=y.
\end{equation}
Note that for every $k\in\N$, the set
$$
\biggl\{\int_\Omega bg^k:b\in Z,\ \|b\|=1\biggr\}
$$
is convex and is a dense subset of the closed convex hull of the set
$$
\{z^k:\text{$z$ is in the essential range of $g$ and $|z|<r$}\}.
$$
Thus there is $b_s\in Z$ such that
\begin{equation}\label{bsn}
\|b_s\|_1=\frac{\alpha_1r_m^{k_s}}{r_{j_{1,s}}^{k_s}}+
\frac{\alpha_2r_m^{k_s}}{r_{j_{2,s}}^{k_s}}+
\frac{\alpha_3r_m^{k_s}}{r_{j_{3,s}}^{k_s}}
\end{equation}
and
$$
\biggl|\int_\Omega b_sg^{k_s}-\biggl(\frac{\alpha_1r_m^{k_s}}{r_{j_{1,s}}^{k_s}}\lambda_{j_{1,s}}^{k_s}+
\frac{\alpha_2r_m^{k_s}}{r_{j_{2,s}}^{k_s}}\lambda_{j_{2,s}}^{k_s}+
\frac{\alpha_3r_m^{k_s}}{r_{j_{3,s}}^{k_s}}\lambda_{j_{3,s}}^{k_s}\biggr)\biggr|<2^{-s}.
$$
The above display together with (\ref{eqy}) gives
$$
y_s=\int_\Omega (a_s+b_s)g^{k_s}\,d\mu\to y\ \ \text{as $s\to\infty$.}
$$
Since $r_{l,s}>r_m$ for $1\leq l\leq 3$, (\ref{bsn}) implies that $\|b_s\|_1\to 0$ as $s\to\infty$. By definition of $y_s$, $(a+b_s,y_s)\in \Lambda$ for each $s\geq q$. Since $a+b_s\to a$, $y_s\to y$ and $U$ is open, $\Lambda$ meets $U$. Since $U$ is an arbitrary open subset of $Z\times\C$, $\Lambda$ is dense in $Z\times\C$, which completes the proof.
\end{proof}

For the sake of brevity, we use the following fairly standard notation:
$$
\widehat{\mu}(n)=\int_\T z^n\,d\mu(z)\ \ \text{for $n\in\Z$ for a Borel measure $\mu$ on $\T$.}
$$
As usual, for $g\in L^1(\mu)$, $g\mu$ stands for the measure absolutely continuous with respect to $\mu$ with the density (=Radon--Nykodim derivative) $g$.
For a non-negative Borel measure on $\T$, we denote the set of all real-valued non-negative $f\in L^1(\mu)$ by the symbol $L^1_+(\mu)$.

\begin{lemma}\label{proba} Let $\mu$ be a purely non-atomic Borel probability measure on $\T$ and $\{R_n\}_{n\in\N}$ be a sequence of positive numbers such that $R_n\to\infty$. Then there is $a\in L^1_+(\mu)$ such that $\{R_n\widehat{a\mu}(n):n\in\N\}$ is dense in $\C$.
\end{lemma}

\begin{proof} We start by proving the following claim.
\begin{equation}\label{sysy}
\begin{array}{l}
\text{For every purely non-atomic Borel probability measure $\nu$ on $\T$}\\ \text{and each $\epsilon>0$, the set $\{n\in\Z:|\widehat{\nu}(n)|<\epsilon\}$ is syndetic.}
\end{array}
\end{equation}
By the Fubini theorem,
$$
|\widehat{\nu}(k)|^2=\widehat{\nu}(k)\overline{\widehat{\nu}(k)}=\int_\T z^k\,d\nu(z)\int_\T w^{-k}\,d\nu(w)=
\int_{\T\times\T}\Bigl(\frac zw\Bigr)^k\,d(\nu\times\nu)(z,w).
$$
For each $m\in\Z$ and $n\in\N$, we sum up these equalities for $m\leq k\leq m+n-1$:
\begin{align*}
&\qquad\frac1{n}\sum_{k=m}^{m+n-1}|\widehat{\nu}(k)|^2=\int_{\T\times\T}h_{m,n}(z,w)\,d(\nu\times\nu)(z,w),
\\
&\text{where}\ \ h_{m,n}(z,w)=\frac1{n}\Bigl(\frac zw\Bigr)^m\sum_{k=0}^{n-1} \frac{z^k}{w^k}=
\frac{z^m(w^n-z^n)}{n\,w^{m+n-1}(w-z)}.
\end{align*}
For each $\delta\in(0,2)$, we can split the above integral:
$$
\frac1{n}\sum_{k=m}^{m+n-1}|\widehat{\nu}(k)|^2=\int_{A_\delta}h_{m,n}(z,w)\,d(\nu\times\nu)(z,w)+
\int_{B_\delta}h_{m,n}(z,w)\,d(\nu\times\nu)(z,w),
$$
where $A_\delta=\{(z,w)\in\T^2:|z-w|<\delta\}$ and $B_\delta=\{(z,w)\in\T^2:|z-w|\geq\delta\}$.
Note that $|w-z|\geq \delta$ and $|w^n-z^n|\leq 2$ for $(z,w)\in B_\delta$. Hence $|h_{m,n}(z,w)|\leq \frac{2}{\delta n}$ for $(z,w)\in B_\delta$.
Thus
$$
\biggl|\int_{B_\delta}h_{m,n}(z,w)\,d(\nu\times\nu)(z,w)\biggr|\leq \frac{2}{\delta n}(\nu\times\nu)(B_\delta)\leq \frac{2}{\delta n}.
$$
On the other hand, for every $z,w\in\T$ and $n\in\N$, $h_{m,n}(z,w)$, being the average of several elements of $\T$, satisfies $|h_{m,n}(z,w)|\leq 1$. Hence
$$
\biggl|\int_{A_\delta}h_{m,n}(z,w)\,d(\nu\times\nu)(z,w)\biggr|\leq (\nu\times\nu)(A_\delta).
$$
Combining the last three displays, we get
$$
\frac1{n}\sum_{k=m}^{m+n-1}|\widehat{\nu}(k)|^2\leq \frac{2}{\delta n}+(\nu\times\nu)(A_\delta).
$$
Since $\nu$ is purely non-atomic, we have $\lim\limits_{\delta\to0}(\nu\times\nu)(A_\delta)=0$. Thus we can find $\delta\in(0,2)$ such that
$(\nu\times\nu)(A_\delta)<\frac{\epsilon^2}2$. Having $\delta$ fixed, we pick $n\in\N$ such that $\frac{2}{\delta n}<\frac{\epsilon^2}{2}$. By the above display,
$$
\frac1{n}\sum_{k=m}^{m+n-1}|\widehat{\nu}(k)|^2< \frac{\epsilon^2}{2}+\frac{\epsilon^2}{2}=\epsilon^2\ \ \text{for every $m\in\Z$}.
$$
According to the above display, among $n$ consecutive numbers $|\widehat{\nu}(k)|$ there is at least one strictly less than $\epsilon$. Hence the set $\{n\in\Z:|\widehat{\nu}(n)|<\epsilon\}$ is syndetic, which proves Claim~(\ref{sysy}).

We also need the following fact.
\begin{equation}\label{meas} \begin{array}{l} \text{Let $\nu$ be a Borel probability measure on $\T$ with infinite support, $r\in(0,1)$,}\\ \text{$m\in\N$, $A\subseteq \Z$ be syndetic,
and $\{w_n\}_{n\in A,\ n> m}$ be a sequence in $\D$.}\\ \text{Then there exist $a\in L^1_+(\nu)$ and $n\in A$ satisfying $n>m$, $\|a\|_1=1$ and $\widehat{a\nu}(n)=rw_n$.}\end{array}
\end{equation}

Let $K$ be the support of $\nu$. For each $n\in\Z$ denote $K^n=\{z^n:z\in K\}$. It is easy to see that the set $B_n=\{\widehat{a\nu}(n):a\in L^1_+(\nu),\ \|a\|_1=1\}$ is convex, is contained and is dense in the convex span of $K^n$. In particular, $B_n$ contains the interior of the convex span of $K^n$. Since $r<1$, we can pick $q\in\N$ large enough in such a way that $r\overline{\D}$ is contained in the interior of the convex span of the points $e^{2\pi ij/q}$ for $1\leq j\leq q$. Then there is $\epsilon>0$ such that $r\D$ is contained in the convex span of $u_1,\dots,u_q$ whenever $\bigl|\frac{u_j}{u_q}-e^{2\pi ij/q}\bigr|<\epsilon$ for $1\leq j\leq q-1$. By Corollary~\ref{sico} we can find $n\in A$ and $z,z_1,\dots,z_q\in K$ such that $n>m$ and $\bigl|\frac{z_j^n}{z^{n}}-e^{2\pi ij/q}\bigr|<\epsilon/2$ for $1\leq j\leq q$. It follows that $\bigl|\frac{z_j^n}{z_q^n}-e^{2\pi ij/q}\bigr|<\epsilon$ for $1\leq j\leq q-1$. Thus $r\D$ is in the convex span of $z_1^n,\dots,z_q^n$. Hence $r\D\subset B_n$. In particular, $rw_n\in B_n$ and therefore there is $a\in L^1_+(\nu)$, such that $\|a\|_1=1$ and $\widehat{a\nu}(n)=rw_n$, which completes the proof of Claim~(\ref{meas}).

Since each of the maps $a\mapsto R_n\widehat{a\mu}(n)$ from $L^1_+(\mu)$ to $\C$ is continuous, Theorem~U ensures that in order to complete the proof Lemma~\ref{proba} it suffices to verify that
$$
\Lambda=\{(a,R_n\widehat{a\mu}(n)):n\in\N,\ a\in L^1_+(\mu)\}\ \ \text{is  dense in $L^1_+(\mu)\times \C$}.
$$
Fix arbitrary $a\in L^1_+(\mu)$, $z\in \C$ and $\epsilon>0$. By (\ref{sysy}), $A=\bigl\{n\in\Z:\widehat{a\mu}(n)<\frac\epsilon4\bigr\}$
is syndetic. Since $R_n\to \infty$, there is $n_0\in\N$ such that $\bigl|\frac{z}{R_n\epsilon}\bigr|<\frac14$ for each $n\geq n_0$. Thus
$$\textstyle
w_n\in \frac12 \D,\ \ \text{for every $n\in A$ satisfying $n>n_0$, where}\ \ w_n=\frac{z}{R_n\epsilon}-\frac{\widehat{a\mu}(n)}{\epsilon}.
$$
By (\ref{meas}), there is $\beta\in L^1_+(\nu)$ and $n\in A$ satisfying $n>n_0$, $\|\beta\|_1=1$ and $\widehat{\beta\mu}(n)=w_n$. Using the definition of $w_n$, we can rewrite the last equality as $\widehat {b\mu}(n)=z$, where $b=a+\epsilon\beta$.
Then $(b,z)\in\Lambda$ and $\|b-a\|_1=\epsilon$. Since $\epsilon>0$ is arbitrary, $(a,z)$ belongs to the closure of $\Lambda$. Since $a\in L^1_+(\mu)$ and $z\in \C$ are arbitrary, $\Lambda$ is dense in $L^1_+(\mu)\times \C$, which completes the proof.
\end{proof}

\section{Proof of Theorems~\ref{suffwnh} and~\ref{suffsnh}\label{s4}}

We start with the following lemma, which must be known. We include the proof for the reader's convenience and since we were unable to find it in the literature anyway.

\begin{lemma}\label{cnvefu} Let $\{e_1,\dots,e_n\}$ be a basis in an $n$-dimensional Banach space $X$ and
$\{e^*_1,\dots,e^*_n\}$ be the dual basis in $X^*:$ $e^*_j(e_k)=\delta_{j,k}$. Then for every $c\in\R_+^n$ such that $c_1+{\dots}+c_n=1$, there is $(x,f)\in \Pi(X)$ for which $e^*_j(x)f(e_j)=c_j$ for $1\leq j\leq n$.
\end{lemma}

\begin{proof} It is easy to check that $\{(e^*_1(x)f(e_1),\dots,e^*_n(x)f(e_n)):(x,f)\in\Pi(X)\}$ is a closed subset of $\C^n$. Hence it is enough to prove the result in the case $c_j>0$ for every $j$. From now on we shall just assume
that $c_j>0$ for $1\leq j\leq n$.
Consider the continuous map
$$\textstyle
F:X\to [-\infty,+\infty),\quad F(w)=\sum\limits_{j=1}^n c_j \log |e^*_j(w)|.
$$
Since $B(X)$ is compact and the map $t\mapsto F(tx)$ increases on $[0,1]$ for each $x\in S(X)$, $F\bigr|_{B(X)}$ attains its maximal value at some point $x\in S(X)$. Obviously, $e^*_j(x)\neq 0$ for $1\leq j\leq n$. Consider
$$\textstyle
f\in X^*,\quad f(y)=\sum\limits_{j=1}^n \frac{c_je^*_j(y)}{e^*_j(x)}\ \ \text{for $y\in X$}.
$$
Clearly, $f(x)=1$ and $e^*_j(x)f(e_j)=c_j$ for $1\leq j\leq n$. In order to clinch the proof it suffices to show that $|f|$ is bounded by $1$ on $S(X)$. Indeed, then $(x,f)\in\Pi(X)$ and $(x,f)$ satisfies all desired conditions.

Assume the contrary. Then there is $y\in S(X)$ such that $f(y)$ is real and $f(y)>1$. Clearly, $x+tu\in B(X)$ for every $t\in[0,1]$, where $u=y-x$. Moreover, $f(u)=f(y)-f(x)=f(y)-1>0$. Since $F\bigr|_{B(X)}$ attains its maximum at $x$, $F(x+tu)\leq F(x)$ for $t\in[0,1]$ and therefore
$$\textstyle
\limsup\limits_{t\to 0+} \frac{F(x+tu)-F(x)}t\leq 0.
$$
Choose $z\in\T$ other than $1$ or each of $\frac{e^*_j(x)}{|e^*_j(x)|}$, let $\C_-=\C\setminus z\R_+$ and $X_-=\{w\in X:e^*_j(w)\in \C_-\ \ \text{for $1\leq j\leq n$}\}$. Set $\log$ to be the standard (real on the positive half of the real line) branch of the logarithm on $\C_-$. Then $F\bigr|_{X_-}$ is exactly the real part of the function $G:X_-\to\C$, $G(w)=\sum\limits_{j=1}^n c_j \log e^*_j(w)$. Since $x+tu\in X_-$ for each sufficiently small $t$, the function $H(t)=G(x+tu)$ is defined in a neighborhood of $0$. Differentiating, we obtain
\begin{align*}\textstyle
\lim\limits_{t\to 0} \frac{F(x+tu)-F(x)}t&= \textstyle {\rm Re}\,H'(0)=
{\rm Re}\sum\limits_{j=1}^n \Bigl[c_j\frac{d}{dt}\log(e^*_j(x)+te^*_j(u))\Bigr]_{t=0}
\\
&= \textstyle {\rm Re}\sum\limits_{j=1}^n \Bigl[\frac{c_ju_j}{e^*_j(x)+te^*_j(u)}\Bigr]_{t=0}={\rm Re}\sum\limits_{j=1}^n \frac{c_je^*_j(u)}{e^*_j(x)}={\rm Re}\,f(u)=f(u)>0.
\end{align*}
The obvious contradiction between the last two displays completes the proof.
\end{proof}

\begin{lemma}\label{asasas} Let $T\in L(X)$, $\lambda_1,\dots,\lambda_n\in\sigma_p(T)$ and $c_1,\dots,c_n\in\R_+$ satisfy $c_1+{\dots}+c_n=1$. Then there exists $(x,f)\in\Pi(X)$ such that
$f(T^kx)=c_1\lambda_1^k+{\dots}+c_n\lambda_n^k$ for every $k\in\Z_+$.
\end{lemma}

\begin{proof} Without loss of generality we can assume that $\lambda_j$ are pairwise distinct. Choose $e_j\in S(X)$ such that $Te_j=\lambda_je_j$ for $1\leq j\leq n$. Since $\lambda_j$ are pairwise distinct, $e_j$ form a basis of the $n$-dimensional space $E=\spann\{e_1,\dots,e_n\}$. Let $e^*_1,\dots,e^*_n$ be the dual basis in $E^*$: $e^*_k(e_j)=\delta_{i,j}$. By Lemma~\ref{cnvefu}, there is $(x,g)\in \Pi(E)$ such that $e^*_j(x)g(e_j)=c_j$ for $1\leq j\leq n$. By the Hahn--Banach Theorem, there is $f\in X^*$ such that $f\bigr|_E=g$ and $\|f\|=1$. Since $\|x\|=1$ and $f(x)=g(x)=1$, $(x,f)\in\Pi(X)$. Finally, using the equalities $e^*_j(x)g(e_j)=c_j$, we get
\[
f(T^kx)=g(T^nx)=g(e^*_1(x)\lambda_1^ke_1+{\dots}+e^*_n(x)\lambda_n^ke_n)=c_1\lambda_1^k+{\dots}+c_n\lambda_n^k\ \ \text{for each $k\in\Z_+$.} \qedhere
\]
\end{proof}

\subsection{Proof of Theorem~\ref{suffwnh}}

Let $T\in L(X)$ and $\lambda_1,\lambda_2\in\sigma_p(T)$ be such that $|\lambda_1|=|\lambda_2|=R>1$ and $z=\frac{\lambda_1}{R}$, $w=\frac{\lambda_2}{R}$ are independent in $\T$. Then we can pick linearly independent $e_1,e_2\in X$ such that $Te_1=\lambda_1e_1$ and $Te_2=\lambda_2 e_2$. By Lemma~\ref{wnh21}, there are $a,b\in\T$ for which
$$
\text{$O=\{R^n(az^n+bw^n):n\in\Z_+\}=\{a\lambda_1^n+b\lambda_2^n:n\in\Z_+\}$ is dense in $\C$.}
$$
By the Hahn--Banach Theorem, there is $f\in X^*$ such that $f(e_1)=a$ and $f(e_2)=b$. Then $O(T,e_1+e_2,f)=\{a\lambda_1^n+b\lambda_2^n:n\in\Z_+\}$ is dense in $\C$. By Proposition~\ref{ele1}, $T$ is weakly numerically hypercyclic. That is, (\ref{suffwnh}.1) is sufficient for the weak  numeric hypercyclicity of $T$.

Assume that $T\in L(X)$ satisfies (\ref{suffwnh}.2). That is, there exist independent $\lambda_1,\lambda_2\in\T$ such that $\ker(T-\lambda_1I)^2\neq \ker(T-\lambda_1I)$ and $\ker(T-\lambda_2I)^2\neq \ker(T-\lambda_2I)$. An elementary linear algebra argument allows us to pick linearly independent $e_1,e_2,e_3,e_4\in X$ such that $(T-\lambda_1I)e_1=(T-\lambda_2I)e_3=0$, $(T-\lambda_1I)e_2=\lambda_1e_1$ and $(T-\lambda_2I)e_4=\lambda_2e_3$. Then $T^n e_2=\lambda_1^n(e_2+ne_1)$ and $T^n e_4=\lambda_2^n(e_4+ne_3)$ for every $n\in\N$. By Lemma~\ref{wnh21}, there are $a,b\in\T$ such that $\{n(a\lambda_1^{n}+b\lambda_2^{n}):n\in\N\}$ is dense in $\C$. By the Hahn--Banach Theorem, there is $f\in X^*$ for which $f(e_4)=f(e_2)=0$, $f(e_1)=a$ and $f(e_3)=b$. Then $f(T^n(e_2+e_4))=n(a\lambda_1^{n}+b\lambda_2^{n})$ for every $n\in\N$. Hence $O(T,e_2+e_4,f)$ is dense in $\C$ and  $T\in\WNH(X)$ according to Proposition~\ref{ele1}. Thus (\ref{suffwnh}.2) is sufficient for the weak  numeric hypercyclicity of $T$.

Assume that $T\in L(X)$ satisfies (\ref{suffwnh}.3). Then there exist $\lambda_1,\lambda_2,\lambda_3\in\sigma_p(T)$ such that $|\lambda_1|=|\lambda_2|=R>|\lambda_3|=r>1$, $\frac{\lambda_1}{\lambda_2}$ has infinite order in $\T$ and $s=\frac{\lambda_1}{R}$, $y=\frac{\lambda_3}{r}$ are independent. If $s$ and $u=\frac{\lambda_2}{R}$ are independent, then (\ref{suffwnh}.1) is satisfied and therefore $T$ is weakly numerically hypercyclic. It remains to consider the case when $s$ and $u$ are not independent. In this case there are $z,f,g\in \T$ and $l,j\in\Z$ such that $f$ and $g$ have finite order, $s=fz^l$ and $u=gz^j$ and $l\geq j$. Since $\frac{s}{u}$ has infinite order, we have $l\neq j$ and therefore $l>j$. Since $s$ and $y$ are independent, $z$ and $y$ are independent. Since $f$ and $g$ have finite order, there is an even $d\in\N$ such that $f^d=g^d=1$. Then $s^{dn}=z^{kn}z^{mn}$, $u^{dn}=z^{-kn}z^{mn}$ and $y^{dn}=w^n$ for every $n\in\Z_+$, where $k=\frac d2(l-j)\in\N$, $m=\frac d2(l+j)\in\Z$ and $w=y^d\in\T$. Since $z$ and $y$ are independent, $z$ and $w$ are independent as well. By Lemma~\ref{wnh31}, there exist $a,b\in\T$ such that
$$
O=\{R^{dn}z^{mn}(az^{kn}+a^{-1}z^{-kn})+br^{dn}w^n:n\in\Z_+\}\ \ \text{is dense in $\C$.}
$$
Since $\lambda_1$, $\lambda_2$ and $\lambda_3$ are distinct eigenvalues of $T$, we can pick linearly independent $e_1,e_2,e_3\in X$ such that $Te_1=\lambda_1e_1$, $Te_2=\lambda_2e_2$ and $Te_3=\lambda_3e_3$. By the Hahn--Banach Theorem, we can find $h\in X^*$ for which $h(e_1)=a$, $h(e_2)=a^{-1}$ and $h(e_3)=b$. Then for each $n\in\Z_+$,
$$
h(T^{dn}(e_1+e_2+e_3))=a\lambda_1^{dn}+a^{-1}\lambda_2^{dn}+b\lambda_3^{dn}=
R^{dn}z^{mn}(az^{kn}+a^{-1}z^{-kn})+br^{dn}w^n.
$$
Thus $O$ is a subset of $O(T,e_1+e_2+e_3,h)$. Hence $O(T,e_1+e_2+e_3,h)$ is dense in $\C$ and $T\in\WNH(X)$ according to Proposition~\ref{ele1}. That is, (\ref{suffwnh}.3) is sufficient for the weak  numeric hypercyclicity of $T$.

Finally, assume that $T\in L(X)$ satisfies (\ref{suffwnh}.4). That is, there exist $\lambda_1,\lambda_2\in\sigma_p(T)$  and $\lambda_3\in\T$ such that $|\lambda_1|=|\lambda_2|=R>1$, $\frac{\lambda_1}{\lambda_2}$ has infinite order in the group $\T$,  $s=\frac{\lambda_1}{R}$, $\lambda_3$ are independent and $\ker(T-\lambda_3I)^2\neq \ker(T-\lambda_3I)$.
Exactly as in the proof of the previous part, we can find $d,k\in\N$, $m\in\Z$ and independent $z,w\in\T$ such that $s^{dn}=z^{kn}z^{mn}$, $u^{dn}=z^{-kn}z^{mn}$ and $y=w^n$ for every $n\in\Z_+$, where $u=\frac{\lambda_2}{R}$ and $y=\lambda_3$. By Lemma~\ref{wnh31}, there exist $a,b\in\T$ for which
$$
O=\{R^{dn}z^{mn}(az^{kn}+a^{-1}z^{-kn})+bdnw^n:n\in\Z_+\}\ \ \text{is dense in $\C$.}
$$
Pick linearly independent $e_1,e_2,e_3,e_4\in X$ such that $Te_1=\lambda_1e_1$, $Te_2=\lambda_2e_2$, $(T-\lambda_3I)e_3=0$ and $(T-\lambda_3I)e_4=\lambda_3e_3$. By the Hahn--Banach Theorem, we can find $h\in X^*$ for which $h(e_1)=a$, $h(e_2)=a^{-1}$, $h(e_3)=b$ and $h(e_4)=0$. Then for each $n\in\Z_+$,
$$
h(T^{dn}(e_1+e_2+e_4))=a\lambda_1^{dn}+a^{-1}\lambda_2^{dn}+bdn\lambda_3^{dn}=
R^{dn}z^{mn}(az^{kn}+a^{-1}z^{-kn})+bdnw^n.
$$
Thus $O$ is a subset of $O(T,e_1+e_2+e_4,h)$. Hence $O(T,e_1+e_2+e_4,h)$ is dense in $\C$ and $T\in\WNH(X)$ according to Proposition~\ref{ele1}. Thus (\ref{suffwnh}.4) is sufficient for the weak  numeric hypercyclicity of $T$. The proof of Theorem~\ref{suffwnh} is complete.

\subsection{Proof of Theorem~\ref{suffsnh}}

Assume that $T\in L(X)$ satisfies (\ref{suffsnh}.1). By Lemma~\ref{asasas}, there is $(x,f)\in \Pi(X)$ such that $f(T^kx)=c_1\lambda_1^k+{\dots}+c_n\lambda_n^k$ for each $k\in\Z_+$. Hence $O(T,x,f)$ is dense in $\C$ and $T\in\NH(X)$. Since every operator similar to $T$ satisfies the same conditions, $T\in \SNH(X)$. Thus (\ref{suffsnh}.1) implies the strong numeric hypercyclicity of $T$. For further references we prove the following slight modification of this result.

\begin{lemma}\label{snsnsn0} Let $T\in L(\C^n)$ be the diagonal operator with the numbers $\lambda_1,\dots,\lambda_n$ on the diagonal. Then $T\in \NH(\C^n)$ if and only if $T\in \SNH(\C^n)$ if and only if there exist $c_1,\dots,c_n\in\R_+$ such that $\{c_1\lambda_1^k+{\dots}+c_n\lambda_n^k:k\in \Z_+\}$ is dense in $\C$.
\end{lemma}

\begin{proof}If there exist $c_1,\dots,c_n\in\R_+$ such that $\{c_1\lambda_1^k+{\dots}+c_n\lambda_n^k:k\in \Z_+\}$ is dense in $\C$, (\ref{suffsnh}.1) is satisfied, which implies that  $T\in \SNH(\C^n)$ and therefore $T\in \NH(\C^n)$. Now assume that $T\in \NH(\C^n)$. Then there is $x\in S(\C^n)$ for which $\NO(T,x)$ is dense in $\C$. A direct calculation shows that $\NO(T,x)=\{|x_1|^2\lambda_1^k+{\dots}+|x_n|^2\lambda_n^k:k\in\Z_+\}$. That is, $\{c_1\lambda_1^k+{\dots}+c_n\lambda_n^k:k\in \Z_+\}$ is dense in $\C$ with $c_j=|x_j|^2$.
\end{proof}

Assume that $T\in L(X)$ satisfies (\ref{suffsnh}.2). Then there exist $\lambda_1,\lambda_2,\lambda_3\in\sigma_p(T)$ such that $|\lambda_1|=|\lambda_2|=|\lambda_3|=R>1$ and $\frac{\lambda_1}{R}$, $\frac{\lambda_2}{R}$, $\frac{\lambda_3}{R}$ are independent. By Lemma~\ref{3point}, there are $c_1,c_2,c_3\in\R_+$ such that $\{c_1\lambda_1^k+c_2\lambda_2^k+c_3\lambda_3^k:k\in\Z_+\}$ is dense in $\C$. Thus (\ref{suffsnh}.1) is satisfied and $T\in\SNH(X)$. That is, (\ref{suffsnh}.2) implies the strong numeric hypercyclicity of $T$. The proof of Theorem~\ref{suffsnh} is complete.

\section{Proof of Propositions~\ref{suffnh} and~\ref{suffnh1}\label{s5}}

Throughout this section $\hh$ is a Hilbert space and $T\in L(\hh)$.

\subsection{Proof of Proposition~\ref{suffnh}}

Assume that $\lambda_1,\lambda_2\in\sigma_p(T)$ are such that $|\lambda_1|=|\lambda_2|=R>1$, the eigenspaces $\ker(T-\lambda_1 I)$ and $\ker(T-\lambda_2I)$ are non-orthogonal and
$\frac{\lambda_1}{R}$, $\frac{\lambda_2}{R}$ are independent in $\T$. Pick $x,y\in S(\hh)$ such that $Tx=\lambda_1x$, $Ty=\lambda_2 y$ and $\langle x,y\rangle =c>0$. Note that there is $\epsilon=\epsilon(c)>0$ such that the set $\bigl\{\frac{1+cu}{1+cu^{-1}}:u\in\T\bigr\}$ contains $J_\epsilon=\{e^{i\theta}:-\epsilon<\theta<\epsilon\}$.
By Lemma~\ref{wnh21}, there is $a\in J_\epsilon$ such that the set $O=\{\lambda^n+a\mu^n:n\in\Z_+\}$ is dense in $\C$. Since $a\in J_\epsilon$, there is $u\in\T$ such that $a=\frac{1+cu}{1+cu^{-1}}$. Consider the vector $q=\frac{x+uy}{\|x+uy\|}\in S(\hh)$ and denote $b=\|x+uy\|^{-2}$. Then for each $n\in\Z_+$,
$$
\langle T^n q,q\rangle = b\langle \lambda^nx+u\mu^ny,x+uy\rangle=b((1+cu)\lambda^n+(1+cu^{-1})\mu^n)=
b((1+cu^{-1}))(\lambda^n+a\mu^n).
$$
Thus the density of $O$ implies the density of $\NO(T,q)$ and therefore $T\in\NH(\C^2)$. This completes the proof of Proposition~\ref{suffnh}.

\subsection{Proof of Proposition~\ref{suffnh1}}

Assume that $\lambda_1,\lambda_2\in\C$ are such that $\ker(T-\lambda_jI)^2\neq \ker(T-\lambda_jI)$ for $j\in\{1,2\}$, $|\lambda_1|=|\lambda_2|\geq1$ and $\frac{\lambda_1}{|\lambda_1|}$, $\frac{\lambda_2}{|\lambda_2|}$ are independent in $\T$. Since $\ker(T-\lambda_jI)^2\neq \ker(T-\lambda_jI)$, there are two 2-dimensional $T$-invariant subspaces $E_1$ and $E_2$ of $\hh$ such that $E_1\cap E_2=\{0\}$, $(T-\lambda_j I)^2\bigr|_{E_j}=0$ and $(T-\lambda_j I)\bigr|_{E_j}\neq 0$ for $j\in\{1,2\}$. In particular, $L_j=E_j\cap \ker (T-\lambda_jI)$ are one-dimensional for $j\in\{1,2\}$. If $L_1$ and $L_2$ are non-orthogonal, $T\in\NH(\hh)$ according to Proposition~\ref{suffnh}. It remains to consider the case when $L_1$ and $L_2$ are orthogonal. It is straightforward to see that there are $e_1,\dots,e_4\in S(\hh)$ such that $e_1\in L_1$, $e_3\in L_2$, $\{e_1,e_2\}$ is a linear basis in $E_1$, $\{e_3,e_4\}$ is a linear basis in $E_2$, $\langle e_1,e_2\rangle=\langle e_1,e_3\rangle=\langle e_3,e_4\rangle=0$, while $g_{1,4}=\langle e_1,e_4\rangle$, $g_{2,3}=\langle e_2,e_3\rangle$ and $g_{2,4}=\langle e_2,e_4\rangle$ are non-negative real numbers. The above properties of the restrictions of $T$ to $E_1$ and $E_2$ yield the existence of non-zero $\alpha,\beta\in\C$ such that $Te_1=\lambda_1 e_1$, $Te_2=\lambda_1(e_2+\alpha e_1)$, $Te_3=\lambda_2 e_3$ and $Te_4=\lambda_2(e_4+\beta e_3)$. Then
\begin{align}\notag
\langle T^nx,x\rangle&=\lambda_1^n\bigl(|a|^2+|b|^2+g_{1,4}a\overline{d}+g_{2,3}b\overline{c}+g_{2,4}b\overline{d}+\alpha nb(\overline{a}+g_{1,4}\overline{d})\bigr)
\\
&\ \ +\lambda_2^n\bigl(|c|^2+|d|^2+g_{1,4}d\overline{a}+g_{2,3}c\overline{b}+g_{2,4}d\overline{b}+\beta nd(\overline{c}+g_{2,3}\overline{b})\bigr)\label{ORB}
\\
&\qquad\text{for $a,b,c,d\in\C$ and $n\in\N$, where $x=ae_1+be_2+ce_3+de_4$.}\notag
\end{align}
Consider the degree $2$ $\R$-polynomial map
$$
\Phi:\C^4\to \R^2\times\C^2,\quad
\Phi\left(\begin{array}{c}a\\ b\\ c\\ d\end{array}\right)=\left(\begin{array}{l}{\rm Re}\,(|a|^2+|b|^2+g_{1,4}a\overline{d}+g_{2,3}b\overline{c}+g_{2,4}b\overline{d})\\
{\rm Re}\,(|c|^2+|d|^2+g_{1,4}d\overline{a}+g_{2,3}c\overline{b}+g_{2,4}d\overline{b})\\
\alpha b(\overline{a}+g_{1,4}\overline{d})\\
\beta d(\overline{c}+g_{2,3}\overline{b})\end{array}\right).
$$
First, observe that $\Phi(1,0,1,0)=(1,1,0,0)$. Looking at $\Phi$ as at a function of $8$ real variables ${\rm Re}\,a$, ${\rm Im}\,a$, ${\rm Re}\,b$, ${\rm Im}\,b$,
${\rm Re}\,c$, ${\rm Im}\,c$, ${\rm Re}\,d$ and ${\rm Im}\,d$ taking values in $\R^6$ with the components ordered as $\Phi_1$, $\Phi_2$, ${\rm Re}\,\Phi_3$, ${\rm Im}\,\Phi_3$, ${\rm Re}\,\Phi_4$, ${\rm Im}\,\Phi_4$, we can calculate the Jacobi matrix $d\,\Phi$ of $\Phi$ at the point $(1,0,1,0)$:
$$
d\,\Phi(1,0,1,0)=\left(\begin{array}{cccccccc}2&0&g_{2,3}&0&0&0&g_{1,4}&0\\
0&0&g_{2,3}&0&2&0&g_{1,4}&0\\
0&0&{\rm Re}\,\alpha&-{\rm Im}\,\alpha&0&0&0&0\\
0&0&{\rm Im}\,\alpha&{\rm Re}\,\alpha&0&0&0&0\\
0&0&0&0&0&0&{\rm Re}\,\beta&-{\rm Im}\,\beta\\
0&0&0&0&0&0&{\rm Im}\,\beta&{\rm Re}\,\beta
\end{array}\right).
$$
Since the complex numbers $\alpha$ and $\beta$ are non-zero, it is an easy exercise to see that the rank of the above $6$ by $8$ matrix is $6$. Indeed, after removing the two zero columns, we are left with an invertible $6$ by $6$ matrix. The Implicit Function Theorem says that there is $\epsilon>0$ and a smooth (even real-analytic) map $\Psi:W=(1-\epsilon,1+\epsilon)\times (1-\epsilon,1+\epsilon)\times \epsilon\D\times \epsilon\D\to\C^4$ such that $\Psi(1,1,0,0)=(1,0,1,0)$ and $\Phi\circ\Psi={\rm Id}_W$. Now we define
$$
a,b,c,d:\epsilon\D\to\C, \qquad (a(z),b(z),c(z),d(z))=\Psi(1,1,z,\overline{z}).
$$
The equality $\Phi\circ\Psi={\rm Id}_W$ yields
\begin{align}
&\alpha b(z)(\overline{a(z)}+g_{1,4}\overline{d(z)})=z,\quad
\beta d(z)(\overline{c(z)}+g_{2,3}\overline{b(z)})=\overline{z},\label{one1}
\\
&{\rm Re}\,(|a(z)|^2+|b(z)|^2+g_{1,4}a(z)\overline{d(z)}+g_{2,3}b(z)\overline{c(z)}+g_{2,4}b(z)\overline{d(z)})=1,\label{one3}
\\
&{\rm Re}\,(|c(z)|^2+|d(z)|^2+g_{1,4}d(z)\overline{a(z)}+g_{2,3}c(z)\overline{b(z)}+g_{2,4}d(z)\overline{b(z)})=1\label{one4}
\end{align}
for every $z\in\epsilon\D$.
Denote
$$
\phi(z)=|a(z)|^2+|b(z)|^2+g_{1,4}a(z)\overline{d(z)}+g_{2,3}b(z)\overline{c(z)}+g_{2,4}b(z)\overline{d(z)}\ \ \text{for $z\in\epsilon\D$}.
$$
Using (\ref{one3}) and (\ref{one4}), we obtain 
$$
\overline{\phi(z)}=|c(z)|^2+|d(z)|^2+g_{1,4}d(z)\overline{a(z)}+g_{2,3}c(z)\overline{b(z)}+g_{2,4}d(z)\overline{b(z)}\ \ \text{for $z\in\epsilon\D$}.
$$
By the above two displays, (\ref{ORB}) and (\ref{one1}),
$$
\langle T^nx(z),x(z)\rangle =\lambda_1^n(\phi(z)+nz)+\lambda_2^n(\overline{\phi(z)}+n\overline{z})\ \ \text{for every $z\in\epsilon\D$ and $n\in\N$,}
$$
where $x(z)=a(z)e_1+b(z)e_2+c(z)e_3+d(z)e_4$. Applying Lemma~\ref{fnh1} with $z=\frac{\lambda_1}{|\lambda_1|}$, $w=\frac{\lambda_2}{|\lambda_2|}$, $R_n=n|\lambda_1|^n$, $r_n=|\lambda_1|^n$ and $U=\epsilon\D$, we find that there is $u\in\epsilon\D$ such that
$$
\text{$O=\{\lambda_1^n(\phi(u)+nu)+\lambda_2^n(\overline{\phi(u)}+n\overline{u}):n\in\N\}$ is dense in $\C$.}
$$
By the above two displays $\NO(T,y)$, being a positive scalar multiple of $O$, is dense in $\C$, where $y=\frac{x(u)}{\|x(u)\|}$. Thus $T\in\NH(\hh)$. The proof of Proposition~\ref{suffnh1} is complete.

\section{Proof of Theorem~\ref{suffwnhid}\label{s6}}

\begin{lemma}\label{weco} Let $\{x_n\}_{n\in\N}$ be a weakly convergent sequence in a Banach spaces $X$, which is not norm-convergent. Then for every sequence $\{r_n\}_{n\in\N}$ of positive numbers such that $r_n\to\infty$, there is $f\in X^*$ such that $\{r_nf(x_n):n\in\N\}$ is dense in $\C$.
\end{lemma}

\begin{proof} Let $s_n=\sup\{\dist(x_m,L_n):m>n\}$, where $L_n=\spann\{u,x_1,\dots,x_n\}$ and $u$ is the weak limit of $\{x_n\}$. Clearly, $\{s_n\}$ is a decreasing sequence of non-negative numbers and therefore $s_n\to 2s\geq 0$. It is an elementary exercise to show that the equality $s=0$ implies the norm convergence of $\{x_n\}$ to $u$. Thus $s>0$ and passing to a subsequence, if necessary, we can assume that $\dist(x_{n+1},L_n)\geq s>0$ for every $n\in\N$. This estimate combined with the Hahn--Banach Theorem provides $f_n\in X^*$ such that $f_n(x_n)=1$, $f_n(u)=0$, $f_n(x_j)=0$ for $j<n$ and $\|f_n\|\leq \frac1s$. Since $f_n(u)=0$ and $x_k\to u$ weakly, we have $f_n(x_m)\to 0$ as $m\to\infty$. Thus we can choose a strictly increasing sequence $\{n_k\}_{k\in\N}$ of positive integers such that $|f_{n_k}(x_{n_m})|<2^{-m}$ whenever $m>k$ and $r_{n_k}>2^k$ for each $k$.

Now we shall verify that the map $R:\ell^1(\N)\to\ell^1(\N)$ defined by $(Ra)_m=\sum\limits_{k=1}^\infty a_kf_{n_k}(x_{n_m})$ is a well-defined continuous linear operator. Indeed, since $\{f_{n_k}\}$ is bounded,
for every $a\in \ell^1(\N)$, the series $\sum\limits_{k=1}^\infty a_kf_{n_k}$ is absolutely convergent in $X^*$ and therefore defines $g_a=\sum\limits_{k=1}^\infty a_kf_{n_k}\in X^*$. Thus $a\mapsto (Ra)_m=g_a(x_{n_m})$ is a continuous linear functional on $\ell^1(\N)$. Next,
$$
\|Ra-a\|_1=\sum_{m=1}^\infty |(Ra)_m-a_m|\leq \sum_{k,m\in\N,\ k\neq m}|a_k||f_{n_k}(x_{n_m})|.
$$
Since $|f_{n_k}(x_{n_m})|<2^{-m}$ for $m>k$ and $|f_{n_k}(x_{n_m})|=0$ for $m<k$, the above display yields
$$
\|Ra-a\|_1\leq \sum_{k=1}^\infty |a_k|\sum_{m=k+1}^\infty 2^{-m}\leq \frac12 \sum_{k=1}^\infty |a_k|=\frac{\|a\|_1}{2}.
$$
Thus $R$ is bounded. Furthermore, $\|R-I\|\leq\frac12<1$ and therefore $R$ is invertible.

Pick a dense sequence $\{w_k\}_{k\in\N}$ in $\C$ such that $|w_k|\leq k$ for every $k\in\N$. Then the sequence $b$ defined by $b_k=\frac{w_k}{r_{n_k}}$ belongs to $\ell^1(\N)$ since $r_{n_k}>2^k$. Since $R$ is invertible, there is $a\in\ell^1(\N)$ such that $Ra=b$. That is, $g_a(x_{n_m})=b_m$ for every $m\in\N$. According to the definition of $b_m$, $g_a(r_{n_k}x_{n_k})=w_k$ for each $k\in\N$. Hence $\{r_ng_a(x_n):n\in\N\}$ is dense in $\C$.
\end{proof}

\begin{lemma}\label{opera} Let $X$ and $Y$ be Banach spaces and $\{T_n\}_{n\in\N}$ be a sequence of continuous linear operators from $Y$ to $X$. Assume also that
$$\textstyle
\Omega=\biggl\{(y,f)\in Y\times X^*:\liminf\limits_{n\to\infty}\frac{1+|f(T_ny)|}{\bigl\|T_n\bigr\|}=0\biggr\}\ \ \text{is not nowhere dense in $Y\times X^*$.}
$$
Then there is $(y,f)\in Y\times X^*$ such that $\{f(T_ny):n\in\N\}$ is dense in $\C$.
\end{lemma}

\begin{proof} Since $\Omega$ is not nowhere dense, $\Omega\cap U$ is dense in $U$ for some non-empty open set $U\subset Y\times X^*$. Let $V$ be an arbitrary non-empty open subset of $U$.
We shall prove that
\begin{equation}\label{clai}
\text{there exist $n\in \N$ and $(x,f)\in V$ such that $n>m$ and $f(T_nx)=z$.}
\end{equation}
Since $\Omega\cap U$ is dense in $U$, we can pick $(h,u)\in \Omega\cap V$. By definition of $\Omega$, there is an infinite subset $A$ of $\N$ such that
$$\textstyle
\lim\limits_{n\to\infty\atop n\in A}\|T_n\|=\infty\ \ \ \text{and}\ \ \
\lim\limits_{n\to\infty\atop n\in A}\frac{h(T_nu)}{\|T_n\|}=0.
$$
For every $n\in A$, choose $y_n\in S(Y)$ for which $2\|T_ny_n\|\geq \|T_n\|$. For each $n\in A$, the Hahn--Banach Theorem provides $h_n\in S(X^*)$ such that $h_n(T_ny_n)=\|T_ny_n\|$. In particular, $2h_n(T_ny_n)\geq \|T_n\|\to \infty$ as $n\to\infty$, $n\in A$.

\smallskip

{\bf Case 1:} $\liminf\limits_{n\to\infty\atop n\in A}\frac{|h(T_ny_n)|}{h_n(T_ny_n)}>0$.

\smallskip

In this case there are an infinite subset $B$ of $A$ and $c>0$ such that $|h(T_ny_n)|>2ch_n(T_ny_n)$ for each $n\in B$. For $n\in B$ denote $\delta_n=\frac{z-h(T_n u)}{h(T_ny_n)}$. From the above display and the inequality $|h(T_ny_n)|>2ch_n(T_ny_n)\geq c\|T_n\|$ it follows that $\delta_n\to 0$ as $n\to\infty$, $n\in B$. Then we can choose $n\in B$  such that $n>m$ and $(u+\delta_ny_n,h)\in V$. Plugging the definition of $\delta_n$ into the expression $h(T_n(u+\delta_ny_n))$, we after cancellations get $h(T_n(u+\delta_ny_n))=z$.
Thus $n$  and $(x,f)=(u+\delta_ny_n,h)$ satisfy (\ref{clai}).

\smallskip

{\bf Case 2:} $\liminf\limits_{n\to\infty\atop n\in A}\frac{|h_n(T_nu)|}{h_n(T_ny_n)}>0$.

\smallskip

In this case there are an infinite subset $B$ of $A$ and $c>0$ such that $|h_n(T_nu)|>2ch_n(T_ny_n)$ for each $n\in B$. For $n\in B$ denote $\delta_n=\frac{z-h(T_n u)}{h_n(T_nu)}$. From the above display and the inequality $|h(T_ny_n)|>2ch_n(T_ny_n)\geq c\|T_n\|$ it follows that $\delta_n\to 0$ as $n\to\infty$, $n\in B$. Then we can choose $n\in B$  such that $n>m$ and $(u,h+\delta_nh_n)\in V$. Plugging the definition of $\delta_n$ into the expression $(h+\delta_nh_n)(T_nu)$, we after cancellations get $(h+\delta_nh_n)(T_nu)=z$.
Thus $n$  and $(x,f)=(u,h+\delta_nh_n)$ satisfy (\ref{clai}).

\smallskip

{\bf Case 3:} $\frac{h(T_ny_n)}{h_n(T_ny_n)}\to 0$ and $\frac{h_n(T_nu)}{h_n(T_ny_n)}\to 0$ \ \ as $n\to\infty$, $n\in A$.

\smallskip

Fix $\delta>0$ such that $(u+\delta B(Y))\times (h+\delta B(X^*))\subseteq V$. In our case $h_n(T_n(u+\delta y_n))=h_n(T_nu)+\delta h_n(T_ny_n)\neq 0$ for all sufficiently large $n\in A$. This allows us to define
$$\textstyle
r_n=\frac{z-h(T_n(u+\delta y_n))}{h_n(T_n(u+\delta y_n))}=\frac{z-h(T_nu)-\delta h(T_ny_n)}{\delta h_n(T_ny_n)+h_n(T_nu)}.
$$
Since $\frac{h(T_ny_n)}{h_n(T_ny_n)}\to 0$ and $\frac{h_n(T_nu)}{h_n(T_ny_n)}\to 0$ \ \ as $n\to\infty$, $n\in A$, we have $r_n\to 0$  as $n\to\infty$, $n\in A$. Thus we can pick $n\in A$ such that $n>m$ and $|r_n|<\delta$. Set
$x=u+\delta y_n$ and $f=h+r_nh_n$. Since $(u+\delta B(Y))\times (h+\delta B(X^*))\subseteq V$, we have $(x,f)\in V$.
Finally, plugging the definition of $r_n$ into the expression $f(T_nx)=(h+r_nh_n)(T_nu+\delta T_ny_n)$, we after cancellations get $f(T_nx)=z$. This completes the proof of Claim (\ref{clai}).

According to (\ref{clai}), 
$$
\Lambda=\{(y,f,f(T_ny)):(y,f)\in U,\ n\in\Z_+\}\ \ \text{is dense in $U\times \C$.}
$$
By Theorem~U, $W=\bigl\{(y,f)\in U:\{f(T_ny):n\in\Z_+\}\ \text{is dense in}\ \C\bigr\}$ 
is a dense $G_\delta$ subset of $U$. In particular, $W\neq\varnothing$ and there is $(y,f)\in Y\times X^*$ such that $\{f(T_ny):n\in\N\}$ is dense in $\C$.
\end{proof}

\begin{corollary}\label{sufwe1} Let $T\in L(X)$, $A$ be an infinite subset of $\N$ and $\{x_k\}_{k\in\N}$ be a sequence in $X$ such that $T^nx_k\neq 0$ for every $(n,k)\in\Z_+\times\N$ and 
\begin{equation}\label{fir}
\textstyle \text{$\|T^nx_1\|\to\infty$ and for each $k\in\N$, $\frac{\|T^n x_{k+1}\|}{\|T^nx_k\|}\to\infty$  as $n\to\infty$, $n\in A$.}
\end{equation}
Then $T\in \WNH(X)$.
\end{corollary}

\begin{proof} Let $E=\spann\{x_k:k\in\N\}$ and $Y$ be the closure of $E$. By (\ref{fir}), $\|T^nx\|=o\bigl(\bigl\|T^n\bigr|_Y\bigr\|\bigr)$ for each $x\in E$ as $n\to\infty$, $n\in A$. By Lemma~\ref{opera}, applied to $T_n=T^n\bigr|_{Y}:Y\to X$, there are $(x,f)\in Y\times X^*$ such that $\{f(T^nx):n\in\Z_+\}=O(T,x,f)$ is dense in $\C$. Hence $T\in\WNH(X)$.
\end{proof}

\begin{corollary}\label{sufwe2} Let $T\in L(X)$ be such that $\Bigl\{x\in X:\liminf\limits_{n\to\infty}\frac{1+\|T^nx\|}{\|T^n\|}=0\Bigr\}$ is dense in $X$. Then $T\in \WNH(X)$.
\end{corollary}

\begin{proof} Just apply Lemma~\ref{opera} with $Y=X$ and $T_n=T^n$.
\end{proof}

Assume that (\ref{suffwnhid}.1) is satisfied. Then there is a sequence $\{z_n\}_{n\in\N}$ in $\sigma_p(T)$ such that $1<|z_1|<|z_2|<{\dots}$ Pick $x_k\in S(X)$ such that $Tx_k=z_kx_k$ for each $k\in\N$. Then $\|T^n x_1\|=|z_1|^n\to \infty$ and $\frac{\|T^n x_{k+1}\|}{\|T^n x_k\|}=\Bigl|\frac{z_{k+1}}{z_k}\Bigr|^n\to\infty$ as $n\to\infty$. By Corollary~\ref{sufwe1}, $T\in \WNH(X)$.
Thus (\ref{suffwnhid}.1) implies the weak numeric hypercyclicity of $T$.

Assume that (\ref{suffwnhid}.3) is satisfied. That is, there is a cyclic vector $x$ for $T$ satisfying $\liminf\limits_{n\to\infty}\frac{1+\|T^nx\|}{\|T^n\|}=0$. Let $p\in\C[z]$ and $y=p(T)x$. Then $\|T^ny\|\leq \|p(T)\|\|T^nx\|$ and therefore $\liminf\limits_{n\to\infty}\frac{1+\|T^ny\|}{\|T^n\|}=0$. Since the set $\{p(T)x:p\in\C[z]\}$ is dense in $X$, $T\in\WNH(X)$ by Corollary~\ref{sufwe2}. Thus (\ref{suffwnhid}.3) implies the weak numeric hypercyclicity of $T$.

Assume that (\ref{suffwnhid}.4) is satisfied. That is, there are $x\in X$ and an infinite set $A\subseteq \N$ such that $\|T^nx\|\to\infty$ as $n\to\infty$, $n\in A$ and the sequence $\bigl\{\frac{T^nx}{\|T^nx\|}\bigr\}_{n\in A}$ is weakly convergent but is not norm convergent. Let $r_n=\|T^nx\|$ and $x_n=\frac{T^nx}{\|T^nx\|}$. Then $\{r_n\}_{n\in A}$ converges to $\infty$, while $\{x_n\}_{n\in A}$ converges weakly but not in the norm. By Lemma~\ref{weco}, there is $f\in X^*$ such that $\{r_nf(x_n):n\in A\}$ is dense in $\C$. Since $\{r_nf(x_n):n\in A\}$ is a subset of $O(T,x,f)$, $O(T,x,f)$ is dense in $\C$. By Proposition~\ref{ele1}, $T\in \WNH(X)$. Thus (\ref{suffwnhid}.4) implies the weak numeric hypercyclicity of $T$.

Finally, assume that (\ref{suffwnhid}.2) is satisfied. That is, there is $r>1$ such that $\sigma_p(T)\cap r\T$ is infinite. Then we can choose a countable infinite set $M\subset\T$ and $e_z\in S(X)$  such that $Te_z=rze_z$ for each $z\in M$. Note that $e_z$ are linearly independent. Let $E=\spann\{e_z:z\in M\}$. Then the closure of $E$ is a closed $T$-invariant subspace. By Proposition~\ref{ele00}, it suffices to prove the weak numeric hypercyclicity of the restriction of $T$ to the closure of $E$. Thus, without loss of generality, we can assume that $E$ is dense in $X$. If $\limsup\limits_{n\to\infty}\frac{\|T^n\|}{r^n}=\infty$, then $\liminf\limits_{n\to\infty}\frac{1+\|T^nx\|}{\|T^n\|}=0$ for every $x\in E$ since $\|T^nx\|=O(r^n)$. In this case $T\in\WNH(X)$ by Corollary~\ref{sufwe2}. It remains to consider the case $\limsup\limits_{n\to\infty}\frac{\|T^n\|}{r^n}<\infty$. Then there is $c\geq 1$ such that $\|T^n\|\leq cr^n$ for each $n\in\Z_+$. For every $z\in M$, consider the unique linear functional $e^*_z:E\to \C$ such that $e^*_z(e_z)=1$ and $e^*_z(e_w)=0$ if $w\neq z$. First, we shall prove that each $e^*_z$ is bounded and satisfies $\|e^*_z\|\leq c$. Assume the contrary. Then there is $z_0\in M$  and $x=\sum\limits_{j=0}^m c_jz_j$ ($z_j$ are assumed pairwise distinct) such that $\|x\|=1$ and $c_0=e^*_{z_0}(x)=a>c$. Since $\|T^n\|\leq cr^n$, we have $\|x_n\|\leq c$ for each $n\in\Z_+$, where $x_n=z_0^{-n}r^{-n}T^nx=ae_{z_0}+\sum\limits_{j=1}^m c_jz_j^nz_0^{-n}e_{z_j}$. Then
$$
\biggl\|\sum_{n=0}^{k-1}x_n\biggr\|=\biggl\|kae_{z_0}+\sum_{j=1}^mc_j \Bigl(\sum_{n=0}^{k-1}z_j^nz_0^n\Bigr)e_{z_j}\biggr\|\leq ck\ \ \text{for each $k\in\N$}.
$$
By the triangle inequality,
$$
ka\leq kc+\sum_{j=1}^m |c_j|\Bigl|\sum_{n=0}^{k-1}z_j^nz_0^n\Bigr|=ck+\sum_{j=1}^m \frac{|c_j||1-z_j^kz_0^{-k}|}{|1-z_jz_0^{-1}|}\leq ck+\sum_{j=1}^m \frac{2|c_j|}{|z_j-z_0|}\ \ \text{for every $k\in\N$.}
$$
Since the last sum does not depend on $k$, it follows that $a\leq c$. This contradiction proves that each $e^*_z$ is bounded and has the norm at most $c$. Hence we can uniquely extend each $e_z^*$ to $X$ by continuity and thus treat them as elements of $X^*$. By Lemma~\ref{ddiaa}, there is $a\in\ell^1_+(M)$ such that $\sum\limits_{z\in M}\sqrt{a_z}<\infty$ and $O=\Bigl\{r^k\sum\limits_{z\in M}a_zz^k:k\in\N\Bigr\}$ is dense in $\C$. 
Since $\sum\limits_{z\in M}\sqrt{a_z}<\infty$, $\|e_z\|=1$  and $\|e^*_z\|\leq c$ for every $z\in M$, the series $x=\sum\limits_{z\in M}\sqrt{a_z}e_z$ and $f=\sum\limits_{z\in M}\sqrt{a_z}e^*_z$ converge absolutely in the Banach spaces $X$ and $X^*$ respectively. Using the equalities $Te_z=rze_z$ and $e^*_z(e_w)=\delta_{z,w}$ for $z,w\in M$, we have
$$
f(T^kx)=r^kf\Bigl(\sum_{w\in M}\sqrt{a_w}w^ke_w\Bigr)=r^k\sum_{z,w\in M}\sqrt{a_za_w}w^ke^*_z(e_w)=r^k\sum_{z\in M}a_zz^k\ \ \text{for each $k\in\N$.}
$$
Hence $O(T,x,f)=O$ is dense in $\C$. By Proposition~\ref{ele00}, $T\in\WNH(X)$. That is, (\ref{suffwnhid}.2) implies the weak numeric hypercyclicity of $T$.
 This completes the proof of Theorem~\ref{suffwnhid}.

\section{Proof of Theorems~\ref{2dim} and \ref{3dim}\label{s7}}

First, we shall establish some obstacles for an operator to be weakly numerically hypercyclic.

\begin{lemma}\label{abcde} Let $T\in L(X)$ and $X=Y\oplus Z$, where $Y$ and $Z$ are closed $T$-invariant subspaces of $X$ such that $T\bigr|_Y\in L(Y)$ is power bounded, $Z$ is finite dimensional and the distinct eigenvalues of $T\bigr|_Z\in L(Z)$ have distinct absolute values. Then $T\notin \WNH(X)$.
\end{lemma}

\begin{proof} Since the eigenvalues of $T\bigr|_Z$ have distinct absolute values, we can write $\sigma(T\bigr|_Z)=\{z_1,\dots,z_m\}$ with $|z_1|<{\dots}<|z_m|$. Considering the Jordan normal form of $T\bigr|_Z$, it is easy to see that for every $x\in X$ and $f\in X^*$ there exist $y\in Y$, $g\in Y^*$ and polynomials $p_1,\dots,p_m\in \C[z]$ such that
$$\textstyle
f(T^nx)=g(T^ny)+\sum\limits_{j=1}^m p_j(n)z_j^n\ \ \text{for every $n\in\Z_+$}.
$$
If $p_1={\dots}=p_m=0$, $\{f(T^nx)\}_{n\in\Z_+}=\{g(T^ny)\}_{n\in\Z_+}$ is bounded since $T\bigr|_Y$ is power bounded. Thus $O(T,x,f)$ is bounded in $\C$. Otherwise, we can choose $q\in\N$ such that $1\leq q\leq m$, $p_q\neq 0$ and $p_j=0$ whenever $j>q$. Then
$$\textstyle
f(T^nx)=g(T^ny)+z_q^n\Bigl(p_q(n)+\sum\limits_{j=1}^{q-1} p_j(n)\bigl(\frac{z_j}{z_q}\bigr)^n\Bigr)\ \ \text{for every $n\in\Z_+$}.
$$
Since $p_q\neq 0$, $\{g(T^ny)\}_{n\in\Z_+}$ is bounded and $\bigl|\frac{z_j}{z_q}\bigr|<1$ for each $j$ in the above sum, the last display ensures that $\{f(T^nx)\}_{n\in\Z_+}$ is bounded if $|z_q|<1$ and if $|z_q|=1$ and $p_q$ is constant and that $|f(T^nx)|\to\infty$ as $n\to\infty$ otherwise. In any case $O(T,x,f)$ is non-dense in $\C$ and therefore $T\notin\WNH(X)$.
\end{proof}

\begin{lemma}\label{tech} Let $f$ be a holomorphic function defined in a neighborhood of $\T$ and $B$ be a discrete closed subset of $\R$ $(=$a subset of $\R$ with no accumulation points$)$. Then $Bf(\T)$ is nowhere dense in $\C$.
\end{lemma}

\begin{proof} If $f$ is identically $0$, the result is trivial. Otherwise $f$ has only finitely many zeros $s_1,\dots,s_k$ (if any) in $\T$. Looking at the Taylor expansion of $f$ about $s_j$, we see that for each $j$, $f(z)=a_j(z-s_j)^{k_j}(1+O(z-s_j))$ as $z\to s_j$, where $a_j\in\C\setminus\{0\}$ and $k_j\in\N$. Let $v\in\C\setminus \{0\}$ be an accumulation point of $Bf(\T)$, which does not belong to $Bf(\T)$. Then it is easy to see that $v=\lim r_n f(w_n)$, where $\{r_n\}_{n\in\N}$ and $\{w_n\}_{n\in\N}$ are sequences in $B$ and $\T$ respectively such that $|r_n|\to\infty$. Convergence of $\{r_nf(w_n)\}$ yields $f(w_n)\to 0$. Passing to a subsequence, if necessary, we can assume that $w_n\to s_j$ as $n\to\infty$ for some $j$. Then $r_n f(w_n)=r_n a_j(w_n-s_j)^{k_j}(1+O(w_n-s_j))$ as $n\to\infty$. This equality together with $r_nf(w_n)\to v$ implies that $a_jv^{-1}(iw_j)^{k_j}\in\R$. That is $v$ belongs to the union $\Lambda$ of the lines $a_j(iw_j)^{k_j}\R$ for $1\leq j\leq k$. Hence $Bf(\T)\cup\Lambda$ is closed in $\C$. Furthermore, for every line $L$ through the origin, which is not in $\Lambda$, $L\cap (Bf(\T)\cup\Lambda)$ is countable. It follows that $Bf(\T)\cup\Lambda$ has empty interior and therefore, being closed, is nowhere dense in $\C$. Thus $Bf(\T)$ is nowhere dense in $\C$.
\end{proof}

\begin{lemma}\label{tech1} Let $R>r>1$, $a,b,c\in\C$ and $k,m,j\in\Z$. Then the set $\Omega=\{R^n(az^k+bz^m)+cr^nz^j:n\in\Z_+,\ z\in\T\}$ is nowhere dense in $\C$.
\end{lemma}

\begin{proof} It is clear that the set of accumulation points of $\Omega$ that do not belong to $\Omega$ is contained in the set $\Lambda$ of $\lim\limits_{q\to\infty} \bigl( R^{n_q}(az_q^k+bz_q^m)+cr^{n_q}z_q^j\bigr)$, where $\{z_q\}_{q\in\N}$ is a sequence in $\T$ and $\{n_q\}_{q\in\N}$ is a strictly increasing sequence of positive integers. Since the closure of $\Omega$ coincides with $\Omega\cup\Lambda$, it suffices to show that $\Omega\cup\Lambda$ has empty interior. Since the set $\Omega$ does not change if we replace $z$ by $z^2$, we can, without loss of generality, assume that $k$ and $m$ are even.

If $|a|\neq |b|$, then every sequence in the last limit runs off to infinity and therefore $\Lambda=\varnothing$. Same happens if $a=b=0$ and $c\neq 0$. In these cases $\Omega\cup\Lambda=\Omega$ has empty interior as a union of countably many smooth curves. If $c=0$, the desired result follows from Lemma~\ref{tech}. The case $k=m$ is trivial.
Thus without loss of generality, we can assume that $k\neq m$, $|a|=|b|=1$ and $c\neq 0$. Let $v=\lim\limits_{q\to\infty} R^{n_q}(az_q^k+bz_q^m)+cr^{n_q}z_q^j\in\Lambda$. Since $R>r>1$, this limit can only exist if $az_q^k+bz_q^m\to 0$. Passing to a subsequence, if necessary, we may assume that $z_q\to w$ and $aw^k+bw^m=0$. Since $k$ and $m$ are even, we can write $k=s+t$ and $m=s-t$, where $s=\frac{k+m}2\in\Z$ and $t=\frac{k-m}2\in\Z\setminus\{0\}$. As with every pair of elements of $\T$, we can find $\alpha,\beta\in\T$ such that $a=\beta \alpha^t$ and $b=\beta\alpha^{-t}$. Then
$$
R^{n_q}(az_q^k+bz_q^m)+cr^{n_q}z_q^j=\beta\alpha^{-s} R^{n_q}(u_q^t+u_q^{-t})u_q^s +cr^{n_q}\alpha^{-j}u_q^j, 
\ \ \ \text{where $u_q=\alpha z_q$}.
$$
Then $u_q\to x$ with $x^t+x^{-t}=0$ and $R^{n_q}(u_q^t+u_q^{-t})+dr^{n_q}u_q^l\to gv$, where $l=j-s$, $d=c\alpha^{s-j}\beta^{-1}$ and $g=\alpha^s\beta^{-1}$. Since $u_q\to x$, we can write $u_q=xe^{\rho_q i}$, where $\{\rho_q\}$ is a convergent to 0 sequence of real numbers. The equation $x^t+x^{-t}=0$ implies that $x^t=i$ or $x^t=-i$. Passing to a subsequence, if necessary, we can assume that either each $\rho_q$ is positive or that each $\rho_q$ is negative. We consider the case when each $\rho_q$ is positive and $x^t=-i$ (the other 3 cases are similar). Since $2ix^tR^{n_q}\sin(t\rho_q)+dr^{n_q}x^le^{l\rho_qi} \to gv$, we have $2R^{n_q}\sin(t\rho_q)+dr^{n_q}x^le^{l\rho_qi} \to gv$. Since $e^{l\rho_qi}\to 1$, it follows that
$dx^l=-h$, where $h>0$ (otherwise $|2R^{n_q}\sin(t\rho_q)+dr^{n_q}x^le^{l\rho_qi}|\to\infty$). Thus
$(2R^{n_q}\sin(t\rho_q)-hr^{n_q}\cos(l\rho_q))-ir^{n_q}\sin(l\rho_q)\to gv$. Hence $r^{n_q}\sin(l\rho_q)\to-{\rm Im}\,(gv)$ and therefore $r^{n_q}\sin(t\rho_q)\to-\frac tl{\rm Im}\,(gv)$. Then either $gv$ is real or $\{r^{n_q}\sin(t\rho_q)\}$ converges to a non-zero number. In the latter case the boundedness of $\{2R^{n_q}\sin(t\rho_q)-hr^{n_q}\cos(l\rho_q)\}$ (which follows from the convergence) implies that $R=r^2$. Indeed, if $R\neq r^2$, one of the sequences $\{2R^{n_q}\sin(t\rho_q)\}$ and $\{hr^{n_q}\cos(l\rho_q)\}$ runs to infinity faster than the other. In the case $R=r^2$, by looking at the Taylor series expansions of $\sin$ and $\cos$, one sees that $2R^{n_q}\sin(t\rho_q)-hr^{n_q}\cos(l\rho_q)=O(r^{-n_q})$ and therefore $2R^{n_q}\sin(t\rho_q)-hr^{n_q}\cos(l\rho_q)\to 0$. Hence ${\rm Re}\,(gv)=0$. Thus in any case $gv$ is either real or purely imaginary. Hence $v$ belongs to the union of two lines $g^{-1}\R$ and $ig^{-1}\R$. Since there are finitely many possibilities to choose $x$, $\Lambda$ is contained in the union of finitely many lines through the origin. Thus $\Omega\cup\Lambda$ is a part of the union of countably many smooth closed curves and finitely many lines. Hence $\Omega\cup\Lambda$ has empty interior, as required.
\end{proof}

\begin{lemma}\label{ddiag} Let $T\in L(\C^n)$ be a diagonal operator with the diagonal entries $d_1,\dots,d_n$ such that $|d_1|={\dots}=|d_n|\neq0$ and for every $j,k\in\{1,\dots,n\}$, $\frac{d_j}{|d_j|}$ and $\frac{d_k}{|d_k|}$ are not independent. Then $T\notin\WNH(\C^n)$.
\end{lemma}

\begin{proof} Let $R=|d_1|$. If $R\leq 1$, then every numerical orbit of $T$ is bounded and therefore $T$ is not weakly numerically hypercyclic. It remains to consider the case $R>1$. Since every pair $\frac{d_j}{R}$, $\frac{d_k}{R}$ is not independent, there are $w\in\T$, $u_1,\dots,u_n\in\T$ of finite order and $k_1,\dots,k_n\in\Z$ such that $d_j=Ru_jw^{k_j}$ for $1\leq j\leq n$. Since $u_j$ have finite order, we can pick $m\in\N$ such that $u_1^m={\dots}=u_n^m=1$.
It is easy to see that every numerical orbit $O(T,f,x)$ has the shape
$$
O(T,f,x)=\{a_1d_1^k+{\dots}+a_nd_n^k:k\in\Z_+\},\ \ \text{where $a_j\in\C$.}
$$
Taking into account that $d_j=Ru_jw^{k_j}$ and $u_1^m={\dots}=u_n^m=1$, we see that
$$
O(T,f,x)=\bigcup_{q=0}^{m-1} M_q,\ \ \text{where}\ \
M_q=\{R^{km+q}(a_1u_1^qw^{qk_1}w^{kmk_1}+{\dots}+a_nu_n^q w^{qk_n}w^{kmk_n}):k\in\Z_+\}.
$$
Now for $0\leq q\leq m-1$, consider the Laurent polynomials $f_q(z)=\sum\limits_{j=1}^n a_ju_j^qw^{qk_j}z^{k_j}$. Clearly $M_q\subseteq f_q(\T)B_q$, where $B_q=\{R^{km+q}:k\in\Z_+\}$. By Lemma~\ref{tech}, each $M_q$ is nowhere dense in $\C$. By the above display $O(T,f,x)$ is nowhere dense in $\C$. Hence $T\notin \WNH(\C^n)$.
\end{proof}

\begin{lemma}\label{ddiagon} Let $T\in L(\C^3)$ be a diagonal operator with non-zero diagonal entries $d_1,d_2,d_3$ such that for every $j,k\in\{1,2,3\}$ for which $d_jd_k\neq 0$, $\frac{d_j}{|d_j|}$ and $\frac{d_k}{|d_k|}$ are not independent. Then $T\notin\WNH(\C^3)$.
\end{lemma}

\begin{proof} If $|d_1|,|d_2|,|d_3|$ are pairwise distinct, the result follows from Lemma~\ref{abcde}. If $|d_1|=|d_2|=|d_3|$, the result follows from Lemma~\ref{ddiag}. Thus we can assume that $|d_1|=|d_2|\neq |d_3|$.
It is easy to see that every numerical orbit $O(T,f,x)$ has the shape
$$
O(T,f,x)=\{a_1d_1^k+a_2d_2^k+a_3d_3^k:k\in\Z_+\},\ \ \text{where $a_j\in\C$.}
$$
If $|d_1|=|d_2|<|d_3|$ and $a_3\neq 0$, then the sequence $\{a_1d_1^k+a_2d_2^k+a_3d_3^k\}$ is bounded if $|d_3|\leq 1$ and converges to infinity otherwise, ensuring that $O(T,f,x)$ is non-dense in $\C$. If $a_3=0$, then we fall under the jurisdiction of Lemma~\ref{ddiag} with $n=2$ and again $O(T,f,x)$ is non-dense in $\C$. In the case $|d_1|=|d_2|>|d_3|$ and $|d_3|\leq 1$, one can proceed similar to the proof of Lemma~\ref{ddiag} (we leave this easy bit to the reader). It remains to consider the case $|d_1|=|d_2|>|d_3|>1$ and $a_3\neq 0$. Let $R=|d_1|=|d_2|$ and $r=|d_3|$. Since every pair $\frac{d_j}{|d_j|}$ and $\frac{d_k}{|d_k|}$ is not independent, there is $w\in\T$, $u_1,u_2,u_3\in\T$ of finite order and $k_1,k_2,k_3\in\Z$ such that $d_1=Ru_1w^{k_1}$, $d_2=Ru_2w^{k_2}$ and $d_3=ru_3w^{k_3}$. Since $u_j$ have finite order, we can pick $m\in\N$ such that $u_1^m=u_2^m=u_3^m=1$. Then
\begin{align*}
O(T,f,x)&=\bigcup_{q=0}^{m-1} M_q,\ \ \text{where}\ \
\\
M_q&=\{R^{km+q}(a_1u_1^qw^{qk_1}w^{kmk_1}+a_2u_2^qw^{qk_2}w^{kmk_2})+r^{km+q}a_3u_3^qw^{qk_3}w^{kmk_3}:k\in\Z_+\}.
\end{align*}
Now each $M_q$ is contained in a set of the shape $\{R^{mk}(\alpha z^{k_1}+\beta z^{k_2})+r^{mn}\gamma z^{k_3}:k\in\Z_+,\ z\in\T\}$ with $\alpha,\beta,\gamma\in\C$ being constants. Then
 Lemma~\ref{tech1} guarantees that each $M_q$ is nowhere dense in $\C$. By the above display $O(T,f,x)$ is nowhere dense in $\C$. Hence $T\notin \WNH(\C^3)$.
\end{proof}

\subsection{Proof of Theorem~\ref{2dim}}

Let $T\in L(\C^2)$ and $\sigma(T)=\{\lambda_1,\lambda_2\}$ with $\lambda_1,\lambda_2\in\C$. By Theorem~\ref{suffwnh}, if $|\lambda_1|=|\lambda_2|>1$ and $\frac{\lambda_1}{|\lambda_1|}$, $\frac{\lambda_2}{|\lambda_2|}$ are independent, then $T\in\WNH(\C^2)$. If $|\lambda_1|\neq|\lambda_2|$ or $\lambda_1=\lambda_2$, then $T\notin\WNH(\C^2)$ by Lemma~\ref{abcde}. If $|\lambda_1|=|\lambda_2|\leq 1$ and $\lambda_1\neq\lambda_2$, then $T$ is power bounded and therefore $T\notin\WNH(\C^2)$. Finally, if $|\lambda_1|=|\lambda_2|>1$, $\lambda_1\neq\lambda_2$ and $\frac{\lambda_1}{|\lambda_1|}$ ,$\frac{\lambda_2}{|\lambda_2|}$ are not independent, then $T\notin\WNH(\C^2)$ according to Lemma~\ref{ddiag}. Thus $T\in\WNH(\C^2)$ if and only if $|\lambda_1|=|\lambda_2|>1$ and $\frac{\lambda_1}{|\lambda_1|}$, $\frac{\lambda_2}{|\lambda_2|}$ are independent in $\T$.

If $\{\lambda_1^k+\lambda_2^k:k\in\Z_+\}$ is dense in $\C$, then $T\in\SNH(\C^2)$ by Theorem~\ref{suffsnh}. Assume now that $T\in \SNH(\C^2)$. Then $T\in \WNH(\C^2)$ and therefore $|\lambda_1|=|\lambda_2|>1$ and $\lambda_1\neq \lambda_2$. Then $T$ is similar to the diagonal operator with the numbers $\lambda_1$ and $\lambda_2$ on the diagonal. By Lemma~\ref{snsnsn0}, there exist $a,b\geq 0$ such that $\{a\lambda_1^n+b\lambda_2^n:n\in\Z_+\}$ is dense in $\C$. Since $|a\lambda_1^n+b\lambda_2^n|\geq |a-b|R^n$, it follows that $a=b$. Hence $\{\lambda_1^n+\lambda_2^n:n\in\Z_+\}$ is dense in $\C$. Thus $T\in \SNH(\C^2)$ if and only if $\sigma(T)=\{\lambda_1,\lambda_2\}$ with $\{\lambda_1^k+\lambda_2^k:k\in\Z_+\}$ being dense in $\C$.

Now assume that $T\in\WNH(\C^2)$ and $T$ is not unitarily equivalent to a diagonal operator. As we have already shown, $|\lambda_1|=|\lambda_2|>1$ and $\frac{\lambda_1}{|\lambda_1|}$, $\frac{\lambda_2}{|\lambda_2|}$ are independent in $\T$. Since $T$ is not unitarily equivalent to a diagonal operator, the eigenspaces $\ker(T-\lambda_1 I)$ and $\ker(T-\lambda_2I)$ are non-orthogonal. By Proposition~\ref{suffnh}, $T\in\NH(\C^2)$. Finally, assume that $T\in \NH(\C^2)$. Then $T\in \WNH(\C^2)$. If $T$ is unitarily equivalent to a diagonal operator, Lemma~\ref{snsnsn0} implies that $T\in\SNH(\C^2)$. These observations amount to the fact that $T\in \NH(\C^2)$ if and only if either $T\in\SNH(\C^2)$ or $T\in \WNH(\C^2)$ and $T$ is not unitarily equivalent to a diagonal operator. The proof of Theorem~\ref{2dim} is complete.

\subsection{Proof of Theorem~\ref{3dim}}

Let $T\in L(\C^3)$. If there are $\lambda_1,\lambda_2\in\sigma(T)$ such that $|\lambda_1|=|\lambda_2|>1$ and $\frac{\lambda_1}{|\lambda_1|}$, $\frac{\lambda_2}{|\lambda_2|}$ are independent in $\T$, then (\ref{suffwnh}.1) is satisfied and Theorem~\ref{suffwnh} implies that $T\in\WNH(\C^3)$. If $\sigma(T)=\{\lambda_1,\lambda_2,\lambda_3\}$ with $|\lambda_1|=|\lambda_2|>|\lambda_3|>1$, $\frac{\lambda_1}{\lambda_2}$ having infinite order in the group $\T$ and $\frac{\lambda_1}{|\lambda_1|}$, $\frac{\lambda_3}{|\lambda_3|}$ being independent in $\T$, then (\ref{suffwnh}.3) is satisfied and Theorem~\ref{suffwnh} implies that $T\in\WNH(\C^3)$. If distinct members of $\sigma(T)$ have distinct absolute values, then $T\notin\WNH(\C^3)$ according to Lemma~\ref{abcde}. If $\frac{\lambda_j}{|\lambda_j|}$ are pairwise dependent for non-zero $\lambda_j$, Lemma~\ref{ddiagon} guarantees that $T\notin\WNH(\C^3)$.
Consider the case $|\lambda_1|>|\lambda_2|=|\lambda_3|=R$ and either $R\leq 1$ or $\frac{\lambda_2}{R}$ and $\frac{\lambda_3}{R}$ are not independent. If $R\leq 1$, then each numerical orbit of $T$ is either bounded or escapes to infinity at the rate $|\lambda_1|^n$ (provided $|\lambda_1|>1$). If $R>1$ and $\frac{\lambda_2}{R}$ and $\frac{\lambda_3}{R}$ are not independent, then $T\notin\WNH(\C^3)$. Indeed, each numerical orbit of $T$ has the shape $\{a\lambda_1^n+b\lambda_2^n+c\lambda_3^n:n\in\Z_+\}$ if $\lambda_2\neq \lambda_3$ and $\{a\lambda_1^n+b\lambda_2^n+cn\lambda_2^n:n\in\Z_+\}$ if $\lambda_2=\lambda_3$. In any case if $a\neq 0$, this is a sequence escaping to infinity. If $a=0$ we fall under the jurisdiction of Lemma~\ref{ddiag} and the numerical orbit is non-dense. If $\lambda_3$ is $0$, the problem is easily reduced to the 2-dimensional situation, already covered by Theorem~\ref{2dim}. It remains to notice that we have considered all the possibilities (up to the ordering of the eigenvalues).

\section{Proof of Theorem~\ref{suffsnhid}\label{s8}}

\begin{lemma} \label{shau} Let $\{e_n\}_{n\in\N}$ be a Schauder basis in a reflexive Banach space $X$, $\{e_n^*\}_{n\in\N}$ be the corresponding sequence of coordinate functionals and $a\in\ell^1_+(\N)$ be such that $\sum\limits_{n=1}^\infty a_n=1$. Then there is $(x,f)\in\Pi(X)$ such that $e_n^*(x)f(e_n)=a_n$ for each $n\in\N$.
\end{lemma}

\begin{proof} Let $X_n=\spann\{e_1,\dots,e_n\}$. By Lemma~\ref{cnvefu}, for every $n\in\N$, we can find $(x_n,g_n)\in\Pi(X_n)$ such that $e_k^*(x_n)g_n(e_k)=a_k$ for $1\leq k<n$ and $e_n^*(x_n)g_n(e_n)=1-a_1-{\dots}-a_{n-1}$. By the Hahn--Banach Theorem, for each $n\in\N$, there is $f_n\in S(X^*)$ such that $f_n\bigr|_{X_n}=g_n$. Then for every $n\in\N$,
$$
(x_n,f_n)\in\Pi(X)\ \ \text{and}\ \ e_k^*(x_n)f_n(e_k)=a_k\ \ \text{for $1\leq k<n$}.
$$
Since $X$ is reflexive, every bounded sequence in $X\times X^*$ has a weakly convergent subsequence. Hence there is a strictly increasing sequence $\{n_m\}_{m\in\N}$ of positive integers such that $\{x_{n_m}\}_{m\in\N}$ converges weakly to $x\in X$ and $\{f_{n_m}\}_{m\in\N}$ converges weakly to $f\in X^*$. Since $B(X)$ is weakly closed in $X$ and $B(X^*)$ is weakly closed in $X^*$, we have $\|x\|\leq 1$ and $\|f\|\leq 1$. According to the above display, for every $k\in\N$, $e_k^*(x_{n_m})f_{n_m}(e_k)=a_k$ for all sufficiently large $m$. Hence,
$$
e_k^*(x)f(e_k)=\lim_{m\to\infty}e_k^*(x_{n_m})f_{n_m}(e_k)=a_k\ \ \text{for each $k\in\N$}.
$$
It follows that
$$\textstyle
f(x)=f\Bigl(\sum\limits_{k=1}^\infty e_k^*(x)e_k\Bigr)=\sum\limits_{k=1}^\infty e_k^*(x)f(e_k)=\sum\limits_{k=1}^\infty a_k=1.
$$
Since $f(x)=1$, $\|x\|\leq 1$ and $\|f\|\leq 1$, $(x,f)\in\Pi(X)$. Thus $(x,f)$ satisfies all desired conditions.
\end{proof}

Assume that $T\in L(X)$ satisfies (\ref{suffsnhid}.3). That is, there exists $\lambda\in\C$ be such that $|\lambda|\geq 1$ and $T-\lambda I$ is a semi-Fredholm operator of positive index. Since the index is locally constant, $T-zI$ is a semi-Fredholm operator of positive index for every $z\in \lambda+\epsilon \D$ for a sufficiently small $\epsilon$. Since every semi-Fredholm operator of positive index is non-injective, $z+\epsilon\D\subseteq \sigma_p(T)$. Since $|\lambda|\geq 1$, there is $r>1$ such that $J=(\lambda+\epsilon \D)\cap (r\T)\neq \varnothing$. Then $J$ is a non-trivial open arc of the circle $r\T$. By Proposition~\ref{znwn}, we can find $z,w\in J\subset \sigma_p(T)$ such that $\{z^n+w^n:n\in\Z_+\}$ is dense in $\C$. Thus (\ref{suffsnh}.1) is satisfied and therefore $T\in \SNH(X)$ by Theorem~\ref{suffsnh}. Hence (\ref{suffsnhid}.3) implies the strong numeric hypercyclicity of $T$.

Assume now that $T\in L(X)$ satisfies (\ref{suffsnhid}.4). That is, $X$ be reflexive and there exists $\lambda\in\C$ such that $|\lambda|\geq 1$ and $T-\lambda I$ is a semi-Frdholm operator of negative index. It is well-known and easy to see that if $X$ is reflexive and $R\in L(X)$, then $R$ is semi-Fredholm if and only if $R^*$ is semi-Fredholm and ${\bf i}(R^*)=-{\bf i}(R)$. Thus $T^*-\lambda I$ is a semi-Fredholm operator of positive index. By the previous part of the proof, $T^*\in \SNH(X^*)$. By Proposition~\ref{ele00}, $T\in \SNH(X)$. Thus (\ref{suffsnhid}.4) implies the strong numeric hypercyclicity of $T$.

Assume that $T\in L(X)$ satisfies (\ref{suffsnhid}.1). That is, $X$ is reflexive and there is a Schauder basic sequence $\{e_n\}_{n\in\N}$ in $X$ such that $Te_n=\lambda_ne_n$ with $\lambda_n\in\C$ for each $n\in\N$ and for some $c\in\ell^1_+(\N)$,
\begin{equation}\label{Mkkm}\textstyle
\Bigl\{\sum\limits_{j=1}^\infty c_j\lambda_j^k:k\in\Z_+\Bigr\}\ \ \text{is dense in $\C$.}
\end{equation}
Without loss of generality, we may assume that $\sum\limits_{n=1}^\infty c_n=1$. Let $Y$ be the closed linear span of the sequence $\{e_n\}_{n\in\N}$. By Lemma~\ref{shau}, there is $(x,g)\in \Pi(Y)$ such that $e_n^*(x)g(e_n)=c_n$ for every $n\in\Z_+$, where $e_n^*\in Y^*$  are the coordinate functionals for the Schauder basis $\{e_n\}_{n\in\N}$. By the Hahn--Banach Theorem, there is $f\in X^*$ such that $f\bigr|_Y=g$ and $(x,f)\in\Pi(X)$. Then 
\begin{equation*}\textstyle
f(T^kx)=g(T^kx)=g\Bigl(\sum\limits_{n=1}^\infty \lambda_n^ke^*_n(x)e_n\Bigr)=\sum\limits_{n=1}^\infty \lambda_n^ke^*_n(x)g(e_n)=\sum\limits_{n=1}^\infty c_n\lambda_n^k\ \ \text{for each $k\in\Z_+$}. 
\end{equation*}
By (\ref{Mkkm}), $O(T,x,f)$ is dense in $\C$ and $T\in\NH(X)$. Since every operator similar to $T$ satisfies the same conditions, $T\in \SNH(X)$. Thus (\ref{suffsnhid}.1) implies the strong numeric hypercyclicity of $T$.

Finally, assume that $T\in L(X)$ satisfies (\ref{suffsnhid}.2). That is, $X$ is reflexive and there is a Schauder basic sequence $\{e_n\}_{n\in\N}$ in $X$ such that $Te_n=\lambda_ne_n$ where $\lambda_n\in\C$ are such that $|\lambda_1|>1$, the sequence $\{|\lambda_n|\}_{n\in\N}$ is $($maybe non-strictly$)$ increasing and the numbers $\frac{\lambda_n}{|\lambda_n|}$ are pairwise distinct elements of $\T$. Passing to a subsequence, if necessary, we can assume that either $\{|\lambda_n|\}_{n\in\N}$ is strictly increasing or that $\{|\lambda_n|\}_{n\in\N}$ is constant. If $\{|\lambda_n|\}_{n\in\N}$ is strictly increasing, Corollary~\ref{ddIaa} provides
$c\in\ell^1_+(\N)$ satisfying (\ref{Mkkm}). If $|\lambda_n|=r>1$ for every $n\in\N$, by Lemma~\ref{ddiaa}, there is $c\in\ell^1_+(\N)$ for which (\ref{Mkkm}) holds. In any case (\ref{suffsnhid}.1) is satisfied and therefore $T\in\SNH(X)$. Thus (\ref{suffsnhid}.2) implies the strong numeric hypercyclicity of $T$. This completes the proof of Theorem~\ref{suffsnhid}.

\section{Proof of Theorem~\ref{clos}\label{s9}}

\begin{lemma}\label{triv} Let $(x,f)\in\Pi(X)$. Then for every $y\in S(X)\cap \ker f$, there exists $g\in X^*$ such that $g(x)=0$, $g(y)=1$ and $\|g\|\leq 2$.
\end{lemma}

\begin{proof} Let $E=\spann\{x,y\}$. Since $f(x)=1$ and $f(y)=0$, $x$ and $y$ are linearly independent and therefore form a basis in $E$. Consider $\phi\in E^*$ defined by $\phi(x)=0$ and $\phi(y)=1$. Then $u=f(u)x+\phi(u)y$ for every $u\in E$. Hence $|\phi(u)|=\|u-f(u)x\|\leq \|u\|+\|f\|\|u\|=2\|u\|$. That is, $\|\phi\|\leq 2$. By the Hahn--Banach Theorem, there is $g\in X^*$ such that $\|g\|\leq 2$ and $g\bigr|_{E}=\phi$. Obviously, $g$ satisfies all desired conditions.
\end{proof}

\begin{lemma}\label{noncl} Let $T\in L(X)$ be such that $T(X)$ is non-closed in $X$ and $\epsilon>0$. Then there is a $2$-dimensional subspace $E$ of $X$ such that $\|T\bigr|_E\|_{L(E,X)}<\epsilon$.
\end{lemma}

\begin{proof} Since $T(X)$ is non-closed, $\inf\{\|Tx\|:x\in S(X)\}=0$. Hence we can pick $x\in S(X)$ for which $\|Tx\|<\frac{\epsilon}{4}$. By the Hahn--Banach theorem there is $f\in X^*$ such that $(x,f)\in\Pi(X)$. Next, observe that $\inf\{\|Tx\|:x\in S(X)\cap\ker f\}=0$. Indeed, otherwise $T(\ker f)$ is closed in $X$ and is a subspace of $T(X)$ of codimension at most 1. Since a subspace of a Banach space with a finite codimensional closed subspace is closed itself, that would have implied that $T(X)$ is closed.

Thus $\inf\{\|Tx\|:x\in S(X)\cap\ker f\}=0$ and therefore we can choose $y\in S(X)\cap \ker f$ such that $\|Ty\|<\frac{\epsilon}4$. By Lemma~\ref{triv}, there is $g\in X^*$ such that $\|g\|\leq 2$, $g(x)=0$ and $g(y)=1$. Clearly $E=\spann\{x,y\}$ is a 2-dimensional subspace of $X$. Moreover, for each $u\in B(E)$,
$$\textstyle
\|Tu\|=\|f(u)Tx+g(u)Ty\|\leq \|f\|\|Tx\|+\|g\|\|Ty\|<\frac{\epsilon}{4}+2\frac{\epsilon}{4}=\frac{3\epsilon}{4}.
$$
Hence $\|T\bigr|_E\|_{L(E,X)}\leq \frac{3\epsilon}{4}<\epsilon$.
\end{proof}

\begin{lemma}\label{appRO} Let $T\in L(X)$, $E$ be a $2$-dimensional subspace of $X$ and $S:E\to X$ be a linear map. Then there is $R\in L(X)$ such that $R\bigr|_E=S$ and $\|T-R\|\leq 3\|S-T\bigr|_E\|_{L(E,X)}$.
\end{lemma}

\begin{proof} Take any $x\in S(E)$. By the Hahn--Banach Theorem, there exists $f\in X^*$ such that $(x,f)\in\Pi(X)$. Pick $y\in S(E)$ such that $f(y)=0$. By Lemma~\ref{triv}, there is $g\in X^*$ such that $\|g\|\leq 2$, $g(x)=0$ and $g(y)=1$. Now we define $R\in L(X)$ be the formula $Ru=Tu+f(u)(Sx-Tx)+g(u)(Sy-Ty)$. It is straightforward to see that $Rx=Sx$ and $Ry=Sy$. Since $\{x,y\}$ is a basis in $E$, $R\bigr|_E=S$. Next, for every $u\in X$, using the relations $\|x\|=\|y\|=\|f\|=1$ and $\|g\|\leq 2$, we have
$$
\|(T-R)u\|=\|f(u)(T-S)x+g(u)(T-S)y\|\leq 3\|u\|\cdot\|S-T\bigr|_E\|_{L(E,X)}.
$$
Hence $\|T-R\|\leq 3\|S-T\bigr|_E\|_{L(E,X)}$.
\end{proof}

\begin{lemma}\label{fdapp} Let $E$ be a $2$-dimensional Banach space and $T\in L(E)$ be such that $\sigma(T)=\{s,t\}$ and $|s|=|t|\geq 1$. Then there is a sequence $\{T_n\}_{n\in\N}$ in $\SNH(E)$ such that $\|T_n-T\|\to 0$.
\end{lemma}

\begin{proof} Since $|s|=|t|\geq 1$, by Proposition~\ref{znwn}, we can choose two sequence $\{s_n\}_{n\in\N}$ and $\{t_n\}_{n\in\N}$ such that $s_n\to s$, $t_n\to t$ and $\{s_n^k+t_n^k:k\in\Z_+\}$ is dense in $\C$ for every $n\in\N$. Since $\sigma(T)=\{s,t\}$, we can choose a linear basis $\{x,y\}$ in $E$ such that $Tx=sx$ and $Ty=ty+ax$ for some $a\in\C$. For each $n\in\N$ consider $T_n\in L(X)$ given by $T_nx=s_n x$ and $T_ny=t_n y+ax$. Since $s_n\to s$ and $t_n\to t$, we have $\|T_n-T\|\to 0$. On the other hand, $\sigma(T_n)=\{s_n,t_n\}$ for each $n\in\N$. By Theorem~\ref{2dim}, each $T_n$ belongs to $\SNH(E)$.
\end{proof}

Since $\WNH(X)\supseteq \NH(X)\supseteq \SNH(X)$, we have 
(\ref{clos}.3)$\Longrightarrow$(\ref{clos}.2)$\Longrightarrow$(\ref{clos}.1). Next, it is easy to see that the set $\Omega(X)$ of $T\in L(X)$ satisfying (\ref{clos}.4) is operator norm open. Thus in order to verify that (\ref{clos}.4) implies (\ref{clos}.3), it suffices to show that $\Omega(X)\cap \WNH(X)=\varnothing$. Let $T\in\Omega(X)$. Then $X=Y\oplus Z$, where $Y$ and $Z$ are closed $T$-invariant subspaces, $T\bigr|_Y\in L(Y)$ has the spectral radius $<1$, $Z$ has dimension $n\in\Z_+$ and $\sigma\bigl(T\bigr|_Z\bigr)=\{z_1,\dots,z_n\}$ satisfies $1\leq |z_1|<{\dots}<|z_n|$. By Lemma~\ref{abcde}, $T\notin \WNH(X)$, which proves the implication (\ref{clos}.4)$\Longrightarrow$(\ref{clos}.3).

It remains to verify that (\ref{clos}.1) implies (\ref{clos}.4). In order to do this, it suffices to show that every $T\in L(X)\setminus\Omega(X)$ is in the operator norm closure of $\SNH(X)$. Let $T\in L(X)\setminus\Omega(X)$. The relation $T\notin \Omega(X)$ can happen for various reasons.

{\bf Case 1:} \ There is a 2-dimensional $T$-invariant subspace $E$ of $X$ such for which $\sigma\bigl(T\bigr|_E\bigr)=\{s,t\}$ and $|s|=|t|\geq 1$. In this case, according to Lemma~\ref{fdapp}, there is a sequence $\{S_n\}_{n\in\N}$ in $\SNH(E)$ such that $\|S_n-T\bigr|_E\|\to 0$. By Lemma~\ref{appRO}, there is a sequence $\{R_n\}_{n\in\N}$ in $L(X)$ such that $R_n\bigr|_E=S_n$ for each $n\in\N$ and $\|R_n-T\|\to 0$. Since the restriction of each $R_n$ to the invariant subspace $E$ is strongly numerically hypercyclic, Proposition~\ref{ele00} implies that $R_n\in \SNH(X)$ for each $n\in\N$. Thus $T$ is in the norm closure of $\SNH(X)$.

Note that the situation of $T$ having a normal eigenvalue $\lambda$ of multiplicity $\geq 2$ such that $|\lambda|\geq 1$ falls into Case~1. The same happens if there are two distinct $s,t\in\sigma_p(T)$ such that 
$|s|=|t|\geq 1$. This means that if Case~1 does not hold, we can pick $\lambda\in\sigma(T)$ such that 
$$
\text{$|\lambda|\geq 1$, $\lambda$ is not a normal eigenvalue of $T$ and $\dim\ker(T-\lambda I)\leq 1$}.
$$
We keep the notation $\lambda$ for such a spectral point for the rest of the proof.

{\bf Case 2:} \ $(T-\lambda I)(X)$ is non-closed in $X$. Let $\epsilon>0$. By Lemma~\ref{noncl}, we can pick a 2-dimensional subspace $E$ of $X$ such that $\|(T-\lambda I)\bigr|_{E}\|<\frac{\epsilon}{4}$. According to Lemma~\ref{appRO}, there is $S\in L(X)$ such that $S\bigr|_E=\lambda I_E$ and $\|T-S\|<\frac{3\epsilon}{4}$. Now $S$ falls under the jurisdiction of Case~1. Thus there is $R\in \SNH(X)$ such that $\|S-R\|<\frac{\epsilon}{4}$. Since $\|T-R\|\leq \|T-S\|+\|S-R\|<\frac{3\epsilon}{4}+\frac{\epsilon}{4}=\epsilon$ and  $\epsilon>0$  is arbitrary, $T$ belongs to the norm closure of $\SNH(X)$.

Thus it remains to consider the case when $(T-\lambda I)(X)$ is closed in $X$.

{\bf Case 3:} \ $\ker(T-\lambda I)$ is one-dimensional and $(T-\lambda I)(X)=X$.
In this case $T-\lambda I$ is a Fredholm operator of index 1. By Theorem~\ref{suffsnhid}, $T\in \SNH(X)$.

{\bf Case 4:} \ $\dim (X/(T-\lambda I)(X))>\dim (\ker(T-\lambda I))$.
In this case $T-\lambda I$ is a semi-Fredholm operator of negative index. Since $X$ is reflexive, Theorem~\ref{suffsnhid} implies that $T\in \SNH(X)$.

If $\ker(T-\lambda I)=\{0\}$ and $(T-\lambda I)(X)=X$, we have $\lambda\notin\sigma(T)$, which is a contradiction. Finally, if $\dim(\ker(T-\lambda I))=\dim (X/(T-\lambda I)(X))=1$, $\lambda$ is a normal eigenvalue of $T$, which is again a contradiction. This completes the proof of the implication (\ref{clos}.1)$\Longrightarrow$(\ref{clos}.4) and that of Theorem~\ref{clos}.

\section{Proof of Theorem~\ref{normaLLL}\label{s10}}

\begin{lemma}\label{normal1} Let $T$ be a bounded normal operator on a Hilbert space $\hh$. Assume also that there exists a sequence $\{z_n\}_{n\in\N}$ in $\sigma(T)$ such that $1<|z_1|<|z_2|<{\dots}$ \ Then $T\in \WNH(X)$.
\end{lemma}

\begin{proof} Let $r_k=\frac13\min\{|z_k|-|z_{k-1}|,|z_{k+1}|-|z_k|\}$ for $k\in\N$, where we assume $z_0=1$. For each $k\in\N$ pick a continuous function $\phi_k:\C\to[0,1]$ such that $\phi_k(z_k)=1$ and $\phi_k$ vanishes outside the disk $z_k+r_k\D$. The functional calculus for normal operators allows us to consider the operators $\phi_k(T)$. Since $z_k\in \sigma(T)$ and $\phi_k(z_k)\neq 0$, $\phi_k(T)\neq 0$ for every $k\in\N$. Thus we can pick $x_k\in S(\hh)\cap \phi_k(T)(\hh)$. The latter inclusion and the fact that $\phi_k$ vanishes outside the disk $z_k+r_k\D$ implies that $(|z_k|-r_k)^n\leq \|T^nx_k\|\leq (|z_k|+r_k)^n$ for every $k\in\N$ and every $n\in\Z_+$. These estimates together with the inequalities $|z_1|-r_1>1$ and $|z_{k+1}|-r_{k+1}>|z_k|+r_k$ for $k\in\N$ imply that $\|T^n x_1\|\to \infty$ and $\frac{\|T^n x_{k+1}\|}{\|T^n x_k\|}\to\infty$ as $n\to\infty$. By Corollary~\ref{sufwe1}, $T\in \WNH(X)$.
\end{proof}

\begin{lemma}\label{posa} Let $T$ be a normal operator on Hilbert space $\hh$ such that the set $\{-|z|:z\in \sigma(T)\}$ is well-ordered and for every $r>1$, the set $\sigma(T)\cap r\T$ has at most one element. Then for each  $x\in\hh$ and $f\in\hh^*$, either $\{f(T^nx):n\in\Z_+\}$ is bounded or $|f(T^nx)|\to\infty$. In particular, $T\notin\WNH(\hh)$.
\end{lemma}

\begin{proof} Since every well-ordered subset of $\R$ is countable, $M=\{z\in\sigma(T):|z|>1\}$ is (finite or) countable. By the Spectral Theorem, $T$ is unitarily equivalent to the orthogonal direct sum of a normal contraction and a diagonal (with respect to an orthonormal basis) operator with only elements of $M$ on the diagonal.
Taking this into account, we see that for each $f\in\hh^*$ and $x\in \hh$,
\begin{equation}\label{equ}\textstyle
f(T^nx)=b_n+\sum\limits_{t\in M} a_tt^n,\ \ \text{where $a\in\ell^1(M)$ and $b\in\ell^\infty(\Z_+)$.}
\end{equation}
If $a=0$, then $\{f(T^nx):n\in\Z_+\}=\{b_n:n\in\Z_+\}$ is bounded. Assume now that $a\neq 0$. Since $\{-|z|:z\in \sigma(T)\}$ is well-ordered, there is $c=\max\{|t|: t\in M,\ a_t\neq 0\}$. Since $\sigma(T)\cap c\T$ is a singleton, there is the unique $t_0\in M$ satisfying $|t_0|=c$. Then 
$$\textstyle
|f(T^nx)|=|a_{t_0}|c^n\biggl|1+\frac{b_n}{a_{t_0}t_0^n}+\sum\limits_{t\in M,\ |t|<c}\frac{a_tt^n}{a_{t_0}t_0^n}\biggr|=|a_{t_0}|c^n(1+o(1))\ \ \text{as $n\to\infty$}.
$$
Since $c>1$, $|f(T^nx)|\to\infty$.
\end{proof}

\begin{lemma}\label{nonono} Let $T$ be a normal operator on Hilbert space $\hh$ such that there is a sequence $\{\lambda_n\}_{n\in\N}$ in $\sigma(T)$ such that $1<|\lambda_1|<|\lambda_2|<{\dots}$ and the numbers $\frac{\lambda_j}{|\lambda_j|}$ are pairwise distinct. Then $T\in\NH(\hh)$.
\end{lemma}

\begin{proof} By the Spectral Theorem we can assume that $\hh=L^2(\mu)$, where $(\Omega,{\cal F},\mu)$ is a measure space and $Tf=gf$ for $f\in\hh$, where $g\in L^\infty(\mu)$. Then the spectrum of $T$ is exactly the essential range of $g$. By Lemma~\ref{ddIaa00}, we can pick a non-negative real valued function $a\in L^1(\mu)$ for which
$$
\text{the set}\ \ \ \biggl\{\int_\Omega g^k a\,d\mu:k\in\Z_+\biggr\}\ \ \text{is dense in $\C$.}
$$
Without loss of generality, we may assume that $a$ is normalized to satisfy $\|a\|_1=1$. Then $x=\sqrt{a}\in S(\hh)$ and $\langle T^nx,x\rangle=\int_\Omega g^k a\,d\mu$ for every $n\in\Z_+$. Thus $\NO(T,x)$ is dense in $\C$ and therefore $T\in\NH(\hh)$.
\end{proof}

Now we are ready to prove Theorem~\ref{normaLLL}. Let $\hh$ be a Hilbert space, $T\in L(\hh)$ and $k\in\N$ be such that $T^k$ is normal. First, assume that there is a sequence $\{\lambda_n\}_{n\in\N}$ in $\sigma(T)$ such that $1<|\lambda_1|<|\lambda_2|<{\dots}$. Then $\{\lambda_n^k\}_{n\in\N}$ is a sequence in $\sigma(T^k)$ satisfying $1<|\lambda_1^k|<|\lambda_2^k|<{\dots}$. Since $T^k$ is normal, Lemma~\ref{normal1} yields the weak numeric hypercyclicity of $T^k$. Then $T$ is weakly numerically hypercyclic, which completes the proof of (\ref{normaLLL}.1).

Next, assume that there is a sequence $\{\lambda_n\}_{n\in\N}$ in $\sigma(T)$ such that $1<|\lambda_1|<|\lambda_2|<{\dots}$ and the numbers $\frac{\lambda_j}{|\lambda_j|}$ are pairwise distinct. Then $\{\lambda_n^k\}_{n\in\N}$ is a sequence in $\sigma(T^k)$ satisfying $1<|\lambda_1^k|<|\lambda_2^k|<{\dots}$ and the set $\bigl\{\frac{\lambda_n^k}{|\lambda_n^k|}:n\in\N\bigr\}$ is infinite. Then there is a strictly increasing sequence $\{n_j\}$ of positive integers such that the numbers $\frac{\lambda_{n_j}^k}{|\lambda_{n_j}^k|}$ are pairwise distinct. Since $T^k$ is normal, Lemma~\ref{nonono} implies that $T^k$ is numerically hypercyclic. Hence $T$ is numerically hypercyclic, which completes the proof of (\ref{normaLLL}.2).

Now assume that $T^k$ is self-adjoint. Then $T^2k$ is self-adjoint and non-negative definite. If $\{-|z|:z\in\sigma(T),\ |z|>1\}$ is not well-ordered, we can pick a sequence $\{\lambda_n\}_{n\in\N}$ in $\sigma(T)$ such that $1<|\lambda_1|<|\lambda_2|<{\dots}$ By (\ref{normaLLL}.1), $T\in\WNH(\hh)$. Assume now that
$\{-|z|:z\in\sigma(T),\ |z|>1\}$ is well-ordered. Since $\sigma(T^{2k})\subset[0,\infty)$, $\{-|z|:z\in\sigma(T^{2k}),\ |z|>1\}$ is well-ordered and $\sigma(T^{2k})\cap r\T$ has at most one element for each $r>1$. By Lemma~\ref{posa}, for every $x\in\hh$ and $f\in\hh^*$, $O(T^{2k},x,f)$ is either bounded or nowhere dense in $\C$. Now let $x\in\hh$ and $f\in \hh^*$. Then $O(T,x,f)=\bigcup\limits_{j=0}^{2k-1}O(T^{2k},T^jx,f)$.
Hence $O(T,x,f)$ is the union of finitely many sets each of which is either bounded or nowhere dense. Thus $O(T,x,f)$ is not dense in $\C$ (the set of accumulation points of $O(T,x,f)$ is bounded). Then $T\notin\WNH(\hh)$, which completes the proof of (\ref{normaLLL}.3).

Finally, assume that there exist $r>1$ and a $T$-invariant subspace ${\cal K}$ such that $U=\frac1rT^k\bigr|_{\cal K}\in L({\cal K})$ is a unitary operator with infinite spectrum.
If $\sigma_p(U)$ is infinite, then there are a sequence $\{z_n\}_{n\in\N}$ of pairwise distinct elements of $\T$ and an orthonormal sequence $\{e_n\}_{n\in\N}$ in ${\cal K}\subseteq\hh$ such that $T^ke_n=Ue_n=rz_ne_n$ for every $n\in\N$. Thus $T^k$ satisfies (\ref{suffsnh}.4) and therefore $T^k\in\SNH(\hh)\subset\NH(\hh)$ by Theorem~\ref{suffsnh}. Hence $T\in\NH(\hh)$. It remains to consider the case when $\sigma_p(U)$ is finite. Since $\sigma(U)$ is infinite, the Spectral Theorem tells us that the restriction of $U$ to an appropriate invariant subspace ${\cal K}_0$ is unitarily equivalent to the multiplication by the argument operator $M\in L(L^2(\mu))$, $Mf(z)=zf(z)$, where $\mu$ is a Borel probability purely non-atomic measure on $\T$. By Proposition~\ref{ele}, in order to prove that $T$ is numerically hypercyclic, it suffices to verify that $rM$ is numerically hypercyclic. By Lemma~\ref{proba}, there is $a\in L^1_+(\mu)$ such that $\|a\|_1=1$ and the set $O=\{r^n\widehat{a\mu}(n):n\in\N\}$ is dense in $\C$. Now it is easy to see that $x=\sqrt{a}\in S(L^2(\mu))$ and $\langle M^nx,x\rangle =\widehat{a\mu}(n)$ for each $n\in\Z$. Hence $\NO(rM,x)=O$ is dense in $\C$. Thus $rM$ is numerically hypercyclic and therefore $T$ is. This completes the proof of (\ref{normaLLL}.4) and that of Theorem~\ref{normaLLL}.

\section{Proof of Propositions~\ref{scalar} and~\ref{scalar2}\label{s11}}

Let $T\in L(\C^2)$. Pick $w\in\T$ such that at least one of the eigenvalues of $wT$ is real. By Theorem~\ref{2dim}, $wT\notin\WNH(\C^2)$. This proves Proposition~\ref{scalar}.
The rest of this section is devoted to the proof of Proposition~\ref{scalar2}, which is more sophisticated. Denote
$$
{\cal M}=\biggl\{(r,z,w)\in(1,\infty)\times\T\times\T:
\left(\!\begin{array}{cc}\!\!rz\!\!&0\\ 0&\!\!rw\!\!\end{array}\!\right)
\in\NH(\C^2)\biggr\}.
$$
By Theorem~\ref{2dim}, we can rewrite the definition of ${\cal M}$ in the following equivalent way.
\begin{equation}\label{calm}
{\cal M}=\bigl\{(r,z,w)\in(1,\infty)\times\T\times\T:\{r^n(z^n+w^n):n\in\Z_+\}\ \ \text{is dense in $\C$}\bigr\}.
\end{equation}

Then the sets $M(z,w)$ featuring in Proposition~\ref{scalar2} can be written as
\begin{align*}
M(z,w)&=\biggl\{r>1:\left(\!\begin{array}{cc}\!\!rz\!\!&0\\ 0&\!\!rw\!\!\end{array}\!\right)\in\NH(\C^2)\biggr\}=\{r>1:(r,z,w)\in{\cal M}\}
\\
&=\bigl\{r>1:\{r^n(z^n+w^n):n\in\Z_+\}\ \ \text{is dense in $\C$}\bigr\}\ \ \text{for $z,w\in\T$}.
\end{align*}
If $\{V_k\}_{k\in\N}$ is a base of the topology of $\C$, then we can write
$$
{\cal M}=\bigcap_{k,m\in\N}\bigcup_{n=m}^\infty \bigl\{(r,z,w)\in(1,\infty)\times\T\times\T:r^n(z^n+w^n)\in V_k\bigr\}.
$$
This immediately implies that ${\cal M}$ is a $G_\delta$ subset of $(1,\infty)\times\T\times\T$ and therefore each $M(z,w)$ is a $G_\delta$-subset of $(1,\infty)$.

\begin{lemma}\label{muzz} The set ${\cal M}$ is a dense $G_\delta$-subset of $(1,\infty)\times\T\times\T$. Furthermore,
$$
W=\bigl\{(z,w)\in\T^2:\text{$M(z,w)$ is dense in $(1,\infty)$}\bigr\}\ \ \ \text{is a dense $G_\delta$ subset of $\T^2$.}
$$
\end{lemma}

\begin{proof} By Proposition~\ref{znwn}, for every $r>1$, the set $r\T\cap{\cal M}$ is dense in $r\T$. Hence $\cal M$ is dense in $(1,\infty)\times\T\times\T$. Let $\{r_n:n\in\N\}$ be a dense subset of $(1,\infty)$. By Proposition~\ref{znwn}, the set $W_n=\{(z,w)\in\T^2:r_n\in M(z,w)\}$ is a dense $G_\delta$ subset of $\T^2$ for every $n\in\N$. Hence $\bigcap\limits_{n\in\N}W_n$ is a dense $G_\delta$ subset of $\T^2$, being also a subset of $W$. Hence $W$ is dense in $\T^2$. It remains to show that $W$ is a $G_\delta$ set. 

Let $\{V_k\}_{k\in\N}$ is a base of the topology of $\C$. It is easy to verify that 
$$
A_{a,b,m,k}=\bigcup_{n=m}^\infty \bigl\{(z,w)\in\T^2:(z^n+w^n)[a,b]\cap V_k\neq\varnothing\bigr\}
$$
is an open subset of $\T^2$ whenever $m,k\in\N$ and $1<a<b<\infty$. It remains to notice that 
$$
W=\bigcap_{m,k\in\N}\bigcap_{a,b\in\Q,\ 1<a<b} A_{a,b,m,k}
$$
to conclude that $W$ is a $G_\delta$ subset of $\T^2$. 
\end{proof}

By Lemma~\ref{2dim}, $M(z,w)=\varnothing$ if $z$ and $w$ are not independent. Actually, the equality $M(z,w)=\varnothing$ is a rather typical occurrence.

\begin{lemma}\label{typy}
The set\ \ $\bigl\{(z,w)\in\T^2:M(z,w)\neq\varnothing\bigr\}$\ \ has zero Lebesgue measure in $\T^2$.
\end{lemma}

\begin{proof} Let $\mu$ be the normalized Lebesgue measure on $\T^2$. By (\ref{lebes}), $\sum\limits_{n=1}^\infty \mu\bigl\{(z,w){\in}\T^2:|z^n{+}w^n|\leq \frac1{n^{2}}\bigr\}<\infty$.
Hence for almost all $(z,w)\in\T^2$, $|z^n+w^n|>n^{-2}$ for all sufficiently large $n$ and therefore $r^n|z^n+w^n|\to\infty$ for every $r>1$.  Hence $M(z,w)=\varnothing$ for almost all $z,w\in\T^2$. \end{proof}

On the other hand, it turns out that if $M(z,w)$ is non-empty, then it is infinite.

\begin{lemma}\label{root} Let $z,w\in\T$, $r\in M(z,w)$ and $k$ be an odd positive integer. Then $r^{1/k}\in M(z,w)$.
\end{lemma}

\begin{proof} In this proof, we shall use the symbol $o_r$ for any sequence $\{\alpha_n\}$ of complex numbers satisfying $\limsup |\alpha_n|^{1/n}\leq \frac1r$.
Since $r\in M(z,w)$, the set $\{y_n:n\in\Z_+\}$ is dense in $\C$, where $y_n=r^n(z^n+w^n)$. Let $A=\{n\in\N:n^{-1}\leq |y_n|\leq n\}$. Then  $\{y_n:n\in A\}$ is dense in $\C$. The inequality $|z^n+w^n|\leq nr^{-n}$ for $n\in A$ yields 
\begin{equation}\label{lili}
\text{$w^n=-z^{n}+o_r$ \ and \ $y_n^2=-|y_n|^2z^{2n}+o_r$ \ as $n\to\infty$, $n\in A$.}
\end{equation}
Since $k$ is odd and $r^n(z^{kn}+w^{kn})=r^n(z^n+w^n)\sum\limits_{j=0}^{k-1} (-1)^j z^{jn}w^{(k-1-j)n}$, 
the equality $y_n=r^n(z^n+w^n)$ and (\ref{lili}) imply $r^n(z^{kn}+w^{kn})=y_nz^{(k-1)n}+o_r=\frac{y_n^k}{|y_n|^{k-1}}+o_r$ as $n\to\infty$, $n\in A$. Hence,
$$\textstyle
r^n(z^{kn}+w^{kn})=\Phi(y_n)+o(1)\ \ \text{as $n\to\infty$, $n\in A$,\ \ where $\Phi:\C\setminus\{0\}\to\C$, $\Phi(x)=\frac{x^k}{|x|^{k-1}}$.}
$$
Since $\Phi$ is a homeomorphism of $\C\setminus\{0\}$ onto itself and $\{y_n:n\in A\}$ is dense in $\C$, $\{\Phi(y_n):n\in A\}$ is dense in $\C$. Then by the above display, 
$\{r^n(z^{kn}+w^{kn}):n\in A\}$ is dense in $\C$. Hence the bigger set $\{r^{n/k}(z^n+w^n):n\in\Z_+\}$ is dense in $\C$ and therefore $r^{1/k}\in M(z,w)$. 
\end{proof}

Finally, we shall verify that $M(z,w)$ can not coincide with $(1,\infty)$ by showing that each $M(z,w)$ has zero Lebesgue measure. We need some preparation. Consider the function $\delta:\R\to[0,1]$, $\delta(x)={\rm dist}\,(x,2\Z-1)$. That is, $\delta(x)$ is the distance from $x$ to the nearest odd integer.

\begin{lemma}\label{somth} Let $\theta\in\R\setminus\Q$, $0<a<b<1$ and $A=\{n\in\N:a^n\leq\delta(n\theta)\leq b^n\}$. Then $\sum\limits_{n\in A}\frac1n<\infty$.
\end{lemma}

\begin{proof} For every $q\in\N$, let $p(q)$ be the unique odd integer $p$ such that $\Bigl|\theta-\frac{p}{q}\Bigr|<\frac1{q}$. Denote $\alpha(q)=\Bigl|\theta-\frac{p(q)}{q}\Bigr|$. Then $0<\alpha(q)<\frac1{q}$ for each $q\in\N$. Consider the set
$$\textstyle
B=\Bigl\{q\in\N:\alpha(q)\leq \frac{b^q}{q}\ \ \text{and \ g.c.d.$(q,p(q))=1$}\Bigr\}.
$$
First, we shall show that
\begin{equation}\label{BB}\textstyle
\sum\limits_{q\in B}\frac1{q}<\infty.
\end{equation}
Assume that (\ref{BB}) fails. Then there are two strictly increasing sequences $\{q_j\}_{j\in\N}$ and $\{r_j\}_{j\in\N}$ in $B$ such that $q_j<r_j$ for each $j\in\N$ and $\frac{q_j}{r_j}\to 1$. Indeed, otherwise members of $B$ written in the increasing order grow exponentially and (\ref{BB}) follows. By definition of $B$, we have
$\theta=\frac{p(q_j)}{q_j}+\beta_j=\frac{p(r_j)}{r_j}+\gamma_j$, where $|\beta_j|\leq \frac{b^{q_j}}{q_j}$ and $|\gamma_j|\leq \frac{b^{r_j}}{r_j}$. In particular, $\frac{p(q_j)}{q_j}-\frac{p(r_j)}{r_j}=\gamma_j-\beta_j$. It follows that $\Bigl|\frac{p(q_j)r_j-p(r_j)q_j}{q_jr_j}\Bigr|\leq \frac{b^{q_j}}{q_j}+\frac{b^{r_j}}{r_j}$. Hence
$|p(q_j)r_j-p(r_j)q_j|\leq r_jb^{q_j}+q_jb^{r_j}<2r_jb^{q_j}$ since $q_j<r_j$. Since $\frac{q_j}{r_j}\to 1$ and $b<1$, we have $r_jb^{q_j}\to 0$ and therefore $|p(q_j)r_j-p(r_j)q_j|\to 0$. Since the latter is a sequence of integers, $p(q_j)r_j-p(r_j)q_j=0$ for all sufficiently large $j$. In particular, there is $j\in \N$  for which $p(q_j)r_j=p(r_j)q_j$. Since $p(r_j)$ and $r_j$ are relatively prime, $r_j$ must divide $q_j$, which is not possible since $r_j>q_j$. This contradiction proves (\ref{BB}).

Now $n\in A$ if and only if there is an odd integer $k$ such that $a^n\leq \delta(n\theta)=|n\theta-k|\leq b^n$. Let $r={\rm g.c.d.}\,(n,k)$. Then we can write $n=rq$ and $k=rp$, where $q\in\N$, $p$ is an odd integer and ${\rm g.c.d.}\,(q,p)=1$. Using this notation we can rewrite the last inequality as $\frac{a^{rq}}{rq}\leq \Bigl|\theta-\frac{p}{q}\Bigr|\leq\frac{b^{rq}}{rq}\leq \frac{b^q}{q}$, which happens if and only if $q\in B$, $p=p(q)$ and $\frac{a^{rq}}{rq}\leq \alpha(q)\leq\frac{b^{rq}}{rq}$. Thus, we have obtained the following description of the set $A$:
$$\textstyle
A=\Bigl\{qr:q\in B,\ r\in 2\N-1,\ \frac{a^{rq}}{rq}\leq \alpha(q)\leq\frac{b^{rq}}{rq}\Bigr\}.
$$
Fix $d>0$ small enough in such a way that $d^m\leq \frac{a^m}{m}$ for every $m\in\N$. Then
$$
A\subseteq A'=\{qr:q\in B,\ r\in\N,\ d^{rq}\leq \alpha(q)\leq b^{rq}\}=\{qr:q\in B,\ r\in\N,\ \beta(q)\leq r\leq c\beta(q)\},
$$
where $\beta(q)=\frac{\log \alpha(q)}{q\log d}$ and $c=\frac{\log d}{\log b}$. Obviously, $\beta(q)>0$ and $c>1$. Estimating a monotonic sum by the corresponding integral, one easily sees that
$$\textstyle
\sum\limits_{m\in\N\atop \beta\leq m\leq c\beta}\frac1m\leq 1+\log c\ \ \text{for every $\beta>0$ and $c>1$.}
$$
Using the above two displays and (\ref{BB}), we obtain
\[\textstyle
\sum\limits_{n\in A}\frac1n\leq \sum\limits_{n\in A'}\frac1n=\sum\limits_{q\in B}\frac1q\sum\limits_{r\in\N\atop \beta(q)\leq r\leq c\beta(q)}\frac1r\leq (1+\log c)\sum\limits_{q\in B}\frac1q<\infty.\qedhere
\]
\end{proof}

\begin{corollary}\label{poWW} Let $z\in\T$ be of infinite order and $A=\{n\in\N:ca^n\leq |1+z^n|\leq db^n\}$, where
$0<a<b<1$ and $c,d>0$. Then $\sum\limits_{n\in A}\frac1n<\infty$.
\end{corollary}

\begin{proof} Since $z$ has infinite order in $\T$, $z=e^{i\pi\theta}$ with $\theta\in\R\setminus\Q$. Pick real numbers $s,t$ such that $0<s<a<b<t<1$ and let $A'=\{n\in\N:s^n\leq \delta(n\theta)\leq t^n\}$. It is an elementary exercise to see that $A\setminus A'$ is finite. By Lemma~\ref{somth}, $\sum\limits_{n\in A'}\frac1n<\infty$ and therefore $\sum\limits_{n\in A}\frac1n<\infty$.
\end{proof}

\begin{lemma}\label{scal} For every $z,w\in\T$, $M(z,w)$ has Lebesgue measure $0$.
\end{lemma}

\begin{proof} If $w^{-1}z$ has finite order in $\T$, Theorem~\ref{2dim} implies that $M(z,w)=\varnothing$. Thus we can assume that $s=w^{-1}z$ is of infinite order. If $r\in M(z,w)$, then the number $r^n|z^n+w^n|=r^n|1+s^n|$ belongs to $[1,2]$ for infinitely many $n\in\N$. Let $0<a<b<1$. The last observation yields
$$
M(z,w)\cap [b^{-1},a^{-1}]\subseteq N_{a,b}=\bigcap_{m=1}^\infty \bigcup_{n=m}^\infty K_{n,a,b},\ \ \text{where}\ \ K_{n,a,b}=\{r\in [b^{-1},a^{-1}]:1\leq r^n|1+s^n|\leq 2\}.
$$
Let $A=\{n\in\N:a^n\leq |1+s^n|\leq 2\,b^n\}$. By Corollary~\ref{poWW}, $\sum\limits_{n\in A}\frac1n<\infty$. Note that $K_{n,a,b}=\varnothing$ if $n\notin A$. On the other hand, if $n\in A$, then $K_{n,a,b}$ is contained in the interval $[|1+s^n|^{-1/n},2^{1/n}|1+s^n|^{-1/n}]$, whose length is $(2^{1/n}-1)|1+s^n|^{-1/n}\leq \frac{2^{1/n}-1}{a}\leq \frac1{an}$. Thus, the Lebesgue measure $\mu$ of the sets $K_{n,a,b}$ is estimated as
$$
\text{$\mu(K_{n,a,b})=0$ if $n\notin A$ and $\mu(K_{n,a,b})\leq \frac1{an}$ if $n\notin A$}.
$$
According to the above two displays, for each $m\in\N$,
$$
\mu(M(z,w)\cap [b^{-1},a^{-1}])\leq \sum_{n=m}^\infty \mu(K_{n,a,b})\leq \frac1a \sum_{n\in A\atop n\geq m}\frac1n\to 0\ \ \text{as $m\to\infty$.}
$$
Hence $\mu(M(z,w)\cap [b^{-1},a^{-1}])=0$ whenever $0<a<b<1$. Thus $\mu(M(z,w))=0$.
\end{proof}

Proposition~\ref{scalar2} follows from Lemmas~\ref{muzz}, \ref{root} and~\ref{scal}.

\section{Further comments and open questions\label{s12}}

The next proposition provides a peculiar collection of diagonal numerically hypercyclic operators on $\C^3$. 

\begin{proposition}\label{znwn1} There is a diagonal $T\in SNH(\C^3)$ such that the restriction of $T$ to each invariant $2$-dimensional subspace is not weakly numerically hypercyclic.
\end{proposition}

\begin{proof} Let $1<r<R$ and $k,m$ be distinct integers. By Lemma~\ref{123}, the set
$$
W=\bigl\{(z,w)\in\T^2:\{R^n(z^{kn}+z^{mn})+r^nw^n:n\in\Z_+\}\ \text{is dense in $\C$}\bigr\}
$$
is a dense $G_\delta$ subset of $\T^2$. Pick $(z,w)$ in $W$ and consider the diagonal operator $T$ on $\C^3$ with the numbers $Rz^k$, $Rz^m$ and $rw$ on the diagonal. Clearly, (\ref{suffsnh}.1) is satisfied for $T$. By Theorem~\ref{suffsnh}, $T\in\SNH(\C^3)$. On the other hand, the possible spectra of the restrictions of $T$ to 2-dimensional invariant subspaces are $\{Rz^k,Rz^m\}$, $\{Rz^k,rw\}$ and $\{Rz^{m},rw\}$. In any case, such a restriction is not weakly numerically hypercyclic according to Theorem~\ref{2dim}.
\end{proof}

As we have already observed, a self-adjoint operator is never numerically hypercyclic. Funnily enough a scalar multiple of a self-adjoint operator can turn out to be numerically hypercyclic.

\begin{proposition}\label{scmusa} Let $T$ be a bounded self-adjoint operator on a Hilbert space $\hh$ such that there exists a sequence $\{\lambda_n\}_{n\in\N}$ in $\sigma(T)$ satisfying $1<\lambda_1<-\lambda_2<\lambda_3<-\lambda_4<{\dots}$ Then for any $w\in\T$ of infinite order, $wT\in\NH(\hh)$.
\end{proposition}

\begin{proof} Let $s_k=\frac13\min\{|\lambda_k|-|\lambda_{k-1}|,|\lambda_{k+1}|-|\lambda_k|\}$ for $k\in\N$, where we assume $\lambda_0=1$. Using the spectral theorem, we can find an orthonormal sequence $\{x_n\}_{n\in\N}$ in $\hh$ such that $\langle T^kx_n,x_n\rangle$ is always between $(\lambda_n-s_n)^k$ and $(\lambda_n+s_n)^k$ and $\langle T^mx_k,T^nx_j\rangle =0$ provided $k\neq j$. Let $E$ be the linear span of $\{x_n:n\in\N\}$ and ${\cal K}$ be the closure of $E$ in $\hh$. Fix $y\in E$ and $t\in\C\setminus\{0\}$. Let $k=\max\{j\in\N:\langle y,x_j\rangle\neq 0\}$. Since $|\lambda_j|+s_j<|\lambda_{j+1}|-s_{j+1}$ for each $j$, $|\langle T^ny,y\rangle|\to\infty$ and $\langle T^ny,y\rangle=o(\langle T^nx_{k+1},x_{k+1}\rangle)$ as $n\to\infty$. Furthermore, since $\lambda_j$ alternate signs, $\langle T^ny,y\rangle$ and $\langle T^nx_{k+1},x_{k+1}\rangle$ have opposite signs for every sufficiently large odd positive integer $n$. Hence the numbers
$\frac{t-\langle T^ny,y\rangle}{\langle T^nx_{k+1},x_{k+1}\rangle}$ are positive for sufficiently large odd $n$ and form a convergent to $0$ sequence. That is, for each sufficiently large odd $n$, we can find $\epsilon_n>0$ such that $\epsilon_n^2=\frac{t-\langle T^ny,y\rangle}{\langle T^nx_{k+1},x_{k+1}\rangle}$. Furthermore, $\epsilon_n\to 0$. The last equality can be easily rewritten as $\langle T^n(y+\epsilon_nx_{k+1}),y+\epsilon_nx_{k+1}\rangle=|t|$. Since $w$ has infinite order, we can find a strictly increasing sequence $\{n_j\}_{j\in\N}$ of odd positive integers such that $w^{n_j}\to \frac{t}{|t|}$. Then $y_j\to y$ and $t_j=\langle (wT)^{n_j}y_j,y_j\rangle\to t$, where $y_j=y+\epsilon_{n_j}x_{k+1}$. Thus $(y,t)$ is in the closure of the set
$$
\Lambda=\{(y,\langle (wT)^ny,y\rangle):y\in {\cal K},\ n\in\N\}.
$$
Since $E$  is dense in $\cal K$ and $y\in E$ and $t\in\C\setminus\{0\}$ are arbitrary, $\Lambda$ is dense in ${\cal K}\times \C$. By Theorem~U, there is $y\in {\cal K}$ such that $\{\langle (wT)^ny,y\rangle:n\in\N\}$ is dense in $\C$. Hence $wT\in\NH(\hh)$.
\end{proof}

The normal operator $T$ in the following example is strongly numerically hypercyclic, while even its weak numeric hypercyclicity does not follow from Theorem~\ref{normaLLL}. The example illustrates the limitations of Theorem~\ref{normaLLL}.

\begin{example}\label{fior} Let $\{r_n\}_{n\in\N}$ be a strictly decreasing sequence in $(1,\infty)$ and $A=\{1,2,3,4,5\}$. Then there exists a sequence $\{a_{j,n}\}_{(j,n)\in A\times\N}$ of finite order elements of $\T$ such that the diagonal operator $T$ on $\ell^2(A\times \N)$ with the numbers $r_na_{j,n}$ on the diagonal is strongly numerically hypercyclic.
\end{example}

\begin{proof} Let $\epsilon,d>0$ be the numbers furnished by Lemma~\ref{penta}. For each $n\in\N$, let $\U_n$ be the only cyclic subgroup of $\T$  of order $n$: $\U_n=\{z\in\T:z^n=1\}$. Clearly, $U_n$ is a $\frac{\pi}{n}$-net in $\T$. Hence there is $n_0\in\N$ such that for each $n\geq n_0$, we can find $a_1,\dots,a_5\in\U_n$ for which $(a_1,\dots,a_5,a_1^2,\dots,a_5^2)\in P_\epsilon$, where $P_\epsilon$ is the set defined in (\ref{pe}). Fix a sequence $\{z_n\}_{n\in\N}$ such that $\{z_n:n\in\N\}$ is dense in $\C$ and $|z_n|<dn$ for every $n\in\N$. 
Using the induction with respect to $n$, we shall construct the sequences $\{a_{j,n}\}_{(j,n)\in A\times\N}$ in $\T$, $\{c_{j,n}\}_{(j,n)\in A\times\N}$ of positive numbers, $\{k_n\}_{n\in\N}$ of positive integers and $\{p_n\}_{n\in\N}$ of prime numbers such that
\begin{itemize}\itemsep=-2pt
\item[(A1)] $a_{j,n}\in\U_{p_n}$ for $1\leq j\leq 5$;
\item[(A2)] $\sum\limits_{j=1}^5 c_{j,n}a_{j,n}^{k_n}=\frac{z_n}{r_n^{k_n}}$, 
$\sum\limits_{j=1}^5 c_{j,q}a_{j,q}^{2k_n}=0$ and $\sum\limits_{j=1}^5 c_{j,n}=\frac{n}{r_n^{k_n}}$; 
\item[(A3)] $p_q$ divides $k_n-2k_q$ for $1\leq q<n$; 
\item[(A4)] $p_n>p_{n-1}$ and $r_n^{k_n-k_{n-1}}>n2^n$ if $n\geq 2$. 
\end{itemize} 

{\bf The basis of the induction.} \ Pick an arbitrary prime number $p_1$ such that $p_1\geq n_0$. Then we can choose $b_{1,1},\dots,b_{5,1}\in\U_{p_1}$ such that $(b_{1,1},\dots,b_{5,1},b_{1,1}^2,\dots,b_{5,1}^2)\in P_\epsilon$. Since $r_1>1$, there is $k_1\in\N$ such that $p_1$ does not divide $k_1$ and $r_1^{-k_1}<1$. Since $z_1\in d\,\D$, Lemma~\ref{penta} provides positive numbers $t_{1,1},\dots,t_{5,1}$ such that $\sum\limits_{j=1}^5 t_{j,1}b_{j,1}=z_1$, $\sum\limits_{j=1}^5 t_{j,1}b_{j,1}^{2}=0$ and $\sum\limits_{j=1}^5 t_{j,1}=1$. Since $b_{j,1}\in\U_{p_1}$, $p_1$ does not divide $k_1$ and $p_1$ is prime, there are $a_{1,1},\dots,a_{5,1}\in\U_{p_1}$ such that $a_{j,1}^{k_1}=b_{j,1}$ for $1\leq j\leq 5$. Set $c_{j,1}=\frac{t_{j,1}}{r_1^{k_1}}$ for $1\leq j\leq 5$. Now it is straightforward to see that (A1) and (A2) with $n=1$ are satisfied. Conditions (A3) and (A4) with $n=1$ are satisfied in a trivial manner. Thus we have our basis of induction. 

{\bf The induction step.} \ Let $n\geq 2$ and assume that $a_{j,q}$, $c_{j,q}$, $p_q$ and $k_q$ with $q<n$ satisfying (A1--A4) are already constructed. First, we choose an arbitrary prime number $p_n>p_{n-1}$. Since $p_n>p_1\geq n_0$, we can pick $b_{1,n},\dots,b_{5,n}\in\U_{p_n}$ such that $(b_{1,n},\dots,b_{5,n},b_{1,n}^2,\dots,b_{5,n}^2)\in P_\epsilon$. Since $r_n>1$ and $p_1,\dots,p_n$ are pairwise distinct primes, the Chinese Remainder Theorem allows us to choose $k_n\in\N$ such that $r_n^{k_n-k_{n-1}}>n2^n$, $p_n$ does not divide $k_n$ and $p_q$ divides $k_n-2k_q$ for $1\leq q\leq n-1$. Since $\frac{z_n}{n}\in d\,\D$, Lemma~\ref{penta} provides positive numbers $t_{1,n},\dots,t_{5,n}$ such that $\sum\limits_{j=1}^5 t_{j,n}b_{j,n}=\frac{z_n}{n}$, $\sum\limits_{j=1}^5 t_{j,n}b_{j,n}^{2}=0$ and $\sum\limits_{j=1}^5 t_{j,n}=1$. Since $b_{j,n}\in\U_{p_n}$, $p_n$ does not divide $k_n$ and $p_n$ is prime, there are $a_{1,n},\dots,a_{5,n}\in\U_{p_n}$ such that $a_{j,n}^{k_n}=b_{j,n}$ for $1\leq j\leq 5$. Set $c_{j,n}=\frac{nt_{j,n}}{r_n^{k_n}}$ for $1\leq j\leq 5$. Now it is straightforward to see that (A1--A4) are satisfied. This completes the construction of the sequences $\{a_{j,n}\}_{(j,n)\in A\times\N}$, $\{c_{j,n}\}_{(j,n)\in A\times\N}$, $\{k_n\}_{n\in\N}$ and $\{p_n\}_{n\in\N}$ satisfying (A1--A4). 

For $(j,n)\in A\times\N$, set $\lambda_{j,n}=r_na_{j,n}$. Since $\sum\limits_{j=1}^5 c_{j,n}=\frac{n}{r_n^{k_n}}$ and $r_n^{k_n}\geq n2^n$, we have $c\in\ell^1_+(A\times\N)$. For each $m\in\N$, we can write 
\begin{equation}\label{alnm}
\sum_{(j,n)\in A\times\N}c_{j,n}\lambda_{j,n}^{k_m}=\sum_{n\in\N} \alpha_{n,m},\ \ \text{where}\ \ 
\alpha_{n,m}=r_n^{k_m}\sum_{j=1}^5 c_{j,n}a_{j,n}^{k_m}. 
\end{equation}
First, note that by (A2), 
\begin{equation}\label{alnm1}
\alpha_{m,m}=r_m^{k_m}\sum_{j=1}^5 c_{j,m}a_{j,m}^{k_m}=r_m^{k_m}\frac{z_m}{r_m^{k_m}}=z_m.
\end{equation}
Since for $n<m$, $p_n$ divides $k_m-2k_n$ and $a_{j,n}^{p_n}=1$, (A2) yields 
\begin{equation}\label{alnm2}
\alpha_{n,m}=r_n^{k_m}\sum_{j=1}^5 c_{j,n}a_{j,n}^{k_m}=r_n^{k_m}\sum_{j=1}^5 c_{j,n}a_{j,n}^{2k_n}=0\ \ \text{if $n<m$}. 
\end{equation}
Finally, if $n>m$, using (A4) and (A2), we obtain 
\begin{equation}\label{alnm3}
|\alpha_{n,m}|=r_n^{k_m}\Bigl|\sum_{j=1}^5 c_{j,n}a_{j,n}^{k_m}\Bigr|\leq r_n^{k_m}\sum_{j=1}^5 c_{j,n}=\frac{nr_n^{k_m}}{r_n^{k_n}}\leq \frac{n}{r_n^{k_n-k_{n-1}}}\leq \frac{1}{2^n}\ \ \text{if $n>m$}. 
\end{equation}
Combining (\ref{alnm}--\ref{alnm3}), we get 
$$
\biggl|z_m-\sum_{(j,n)\in A\times\N}c_{j,n}\lambda_{j,n}^{k_m}\biggr|=\biggl|\sum_{n=m+1}^\infty \alpha_{n,m}\biggr|\leq \sum_{n=m+1}^\infty \frac1{2^n}\to 0\ \ \text{as $m\to\infty$}. 
$$
Since $\{z_m:m\in\N\}$ is dense in $\C$, the above display implies that $\Bigl\{
\sum\limits_{(j,n)\in A\times\N}c_{j,n}\lambda_{j,n}^{k}:k\in\N\Bigr\}$ is dense in $\C$. Thus the diagonal operator $T\in L(\ell^2(A\times\N))$ with the numbers $\lambda_{j,n}$ on the diagonal satisfies (\ref{suffsnhid}.1). By Theorem~\ref{suffsnhid}, $T$ is strongly numerically hypercyclic.
\end{proof} 

Let $V$ be the Volterra operator on $L^2[0,1]$. Exactly as in Example~\ref{vol}, we can show that $re^{-V}\in\WNH(L^2[0,1])$ for every $r>1$. Unexpectedly enough it turns out that these operators are not numerically hypercyclic. 

\begin{example}\label{vol00} Let $r>1$ and $T=re^{-V}$, where $V\in L(L^2[0,1])$ is the Volterra operator. Then $T\notin\NH(L^2[0,1])$. 
\end{example}

\begin{proof} The key element of the proof is the following claim: 
\begin{equation}\label{volest}
\text{$\|I-e^{-cV}\|\leq 1$ for every $c\in\R_+$.}
\end{equation}
We shall derive this estimate from the following observation: 
\begin{equation}\label{volest0}
\text{for every $f\in L^2[0,1]$, the function $F_f:\R_+\to \R_+$, $F_f(c)=\|(I-e^{-cV})f\|$ is 
increasing.} 
\end{equation}
First, we shall verify (\ref{volest0}). Obviously, it is enough to show that $G_f=F_f^2$ is increasing. Differentiating by $c$ the expression $G_f(c)=\langle (I-e^{-cV})f,(I-e^{-cV})f\rangle$, we get $G'_f(c)=\langle Vg_c,g_c\rangle+\langle g_c,Vg_c\rangle$, where $g_c=(I-e^{-cV})f$. Hence 
$G'_f(c)=\langle Pg_c,g_c\rangle$, where $P=V+V^\star$ and $V^\star$ is the Hilbert space adjoint of $V$. 
It is easy to see that $P$ is the orthogonal projection onto the one-dimensional space of constant functions. Then the self-adjoint operator $P$ is non-negative definite and $G'_f(c)=\langle Pg_c,g_c\rangle\geq 0$ for every $c\in\R_+$. Hence $G_f$ is increasing and (\ref{volest0}) follows.

Now we shall prove (\ref{volest}). Assume the contrary. Then there is $a>0$ such that $\|I-e^{-aV}\|>1$. Then there is $f\in S(L^2[0,1])$ such that $\|(I-e^{-aV})f\|=F_f(a)>1$. By (\ref{volest0}), $\lim\limits_{c\to\infty}F_f(c)>1$. On the other hand, in \cite{mo} it is shown that $\|(I-V)^ng\|\to 0$ for each $g\in L^2[0,1]$. According to the main result of \cite{fra} the operators $I-V$ and $e^{-V}$ are similar. Hence $\|e^{-nV}g\|\to 0$ for each $g\in L^2[0,1]$. In particular, $\|e^{-nV}f\|\to 0$. Hence $F_f(n)\to 1$, which contradicts the inequality $\lim\limits_{c\to\infty}F_f(c)>1$. This contradiction completes the proof of (\ref{volest}). 

Now let $r>1$ and $T=re^{-V}$. Then for each $f\in S(L^2[0,1])$, using (\ref{volest}), we have 
$$\textstyle
{\rm Re}\,\langle T^nf,f\rangle=\frac{r^n}{2}\langle (e^{-nV}+e^{-nV^\star})f,f\rangle =
\frac{r^n}{2}(2-\langle ((I-e^{-nV})+(I-e^{-nV^\star}))f,f\rangle)\geq 
r^n(1-\|I-e^{-nV}\|)\geq 0.
$$
Hence $\NO(T,f)$ is contained in the right half-plane and therefore can not be dense in $\C$. Thus $T\notin\NH(L^2[0,1])$. 
\end{proof}

At the expense of a number of technical details entering the proof, one can similarly show that $r(I-cV)$ is weakly numerically hypercyclic and not numerically hypercyclic whenever $r>1$ and $c>0$. We shall raise a number of natural questions.

\begin{question}\label{q1} Let $T\in \WNH(\C^n)$. Is it true that at least one of the conditions $(\ref{suffwnh}.1$--$\ref{suffwnh}.4)$ is satisfied$?$
\end{question}

\begin{question}\label{q2} Assume that $X$ is reflexive, $T\in L(X)$ is not power bounded and $\sigma_p(T)=\varnothing$. Is it true that $T$ is weakly numerically hypercyclic$?$
\end{question}

\begin{question}\label{q3} Let $T\in \WNH(X)$. Is it true that $T^n\in\WNH(X)$ for every $n\in\N?$
\end{question}

\begin{question}\label{q4} Let $T\in \WNH(X)$. Is it true that $rT\in\WNH(X)$ for every $r\geq 1?$
\end{question}

\begin{question}\label{q5} Let $T$ be a unitary operator such that $\sigma(T)$ is infinite and $\sigma_p(T)=\varnothing$. Is $rT$ strongly numerically hypercyclic for each $r>1?$
\end{question}

\begin{question}\label{q6} Let $T\in L(X)$ and $\lambda_1,\lambda_2\in\C$ are such that $\ker(T-\lambda_jI)^2\neq \ker(T-\lambda_jI)$ for $j\in\{1,2\}$, $|\lambda_1|=|\lambda_2|\geq1$ and $\frac{\lambda_1}{|\lambda_1|}$, $\frac{\lambda_2}{|\lambda_2|}$ are independent in $\T$. Is it true that $T$ must be strongly numerically hypercyclic$?$
\end{question}

\begin{question}\label{q7} Characterize normal numerically hypercyclic operators in terms of their spectrum.
\end{question}

\begin{question}\label{q8} Characterize numerically hypercyclic operators on $L^2[0,1]$ commuting with the Volterra operator $V$. In particular, is $I+V$ numerically hypercyclic$?$ 
\end{question} 

The last question is in the spirit of the Invariant Subspace Problem and probably is tough. 

\begin{question}\label{q9} Is there $T\in L(\ell^2(\N))$ such that $\NO(T,x)$ is dense in $\C$ for each $x\in S(\ell^2(\N))?$
\end{question}

\small\rm

\vskip1truecm

\scshape

\noindent Stanislav Shkarin

\noindent Queens's University Belfast

\noindent Department of Pure Mathematics

\noindent University road, Belfast, BT7 1NN, UK

\noindent E-mail address: \qquad {\tt s.shkarin@qub.ac.uk}

\end{document}